\documentclass[10pt,francais]{amsart}
\usepackage[utf8]{inputenc}
\usepackage[french,english]{babel}
\usepackage{amsfonts}
\usepackage{amsmath}
\usepackage{amsthm}
\usepackage{amssymb}

\usepackage[backend=biber]{biblatex}
\usepackage{lmodern}
\usepackage{array}
\usepackage{proof}
\usepackage{graphicx}
\usepackage{version}
\usepackage{enumerate}
\usepackage{graphicx}
\usepackage[all]{xy}
\usepackage{graphicx}
\usepackage{xcolor}
\usepackage{soul}
\usepackage{hyperref}

\theoremstyle{definition}
\newtheorem{defi}{Définition}[subsection]
\newtheorem{exe}[defi]{Exemple}
\newtheorem{rem}[defi]{Remarque}

\theoremstyle{plain}
\newtheorem{prop}[defi]{Proposition}
\newtheorem{lem}[defi]{Lemme}
\newtheorem{cor}[defi]{Corollaire}
\newtheorem{theo}[defi]{Théorème}

\newtheorem{notation}[defi]{Notation}
\newtheorem*{thoeremnonnum}{Théorème}

\newtheorem{theolettre}{Théorème}
\newtheorem{corlettre}[theolettre]{Corollaire}

\calclayout

\newcommand{\var}{\underline{\hspace{7pt}}}

\newcommand{\N}{\mathcal{N}}

\definecolor{darkgreen}{rgb}{0,0.45,0}
\definecolor{darkred}{rgb}{0.75,0,0}
\definecolor{darkblue}{rgb}{0,0,0.6}

\newcommand{\pushoutcorner}[1][dr]{\save*!/#1+1.2pc/#1:(1,-1)@^{|-}\restore}

\newcommand{\note}[1]{{\color{black}#1}}

\newcommand\infcat{\textbf{$\omega$-cat} }
\newcommand\infcatB{  \mbox{$\textbf{$\omega$-cat}_B$}   }
\newcommand\strat{\textbf{Strat}}
\DeclareMathOperator{\dec}{dec}

\newcommand\CDA{\textbf{CDA}~}
\newcommand\CDAB{\mbox{$\textbf{CDA}_B$}~}

\newcommand{\nocontentsline}[3]{}
\newcommand{\tocless}[2]{\bgroup\let\addcontentsline=\nocontentsline#1{#2}\egroup}
\def\commutatif{\ar@{}[d]|{\circlearrowright}}

\author{Félix Loubaton}
\title{Conditions de Kan sur les nerfs des $\omega$-catégories}

\begin{document}
\maketitle

\selectlanguage{french}
\begin{abstract}
On montre que  le nerf de Street $\N(C)$ d'une $\omega$-catégorie stricte $C$ est un complexe de Kan (respectivement une quasi-catégorie) si et seulement si les $n$-cellules de $C$  pour $n\geq 1$ (respectivement $n> 1$) sont faiblement inversibles.
\note{
 De plus, on munit $\N(C)$ d'une structure d'ensemble complicial saturé où les $n$-simplexes marqués correspondent aux morphismes du $n^{i\grave{e}me}$ oriental vers $C$ envoyant l'unique $n$-cellule non triviale du domaine sur une cellule faiblement inversible de $C$.} 
\end{abstract}

\selectlanguage{english}
\begin{abstract}
We show that the Street nerve of a strict $\omega$-category  $C$  is a Kan complex (respectively a quasi-category) if and only if the $n$-cells of $C$ for $n\geq 1$ (respectively $n> 1$) are weakly invertible. 
\note{
Moreover, we equip $\N(C)$ with a structure of saturated complicial set  where the $n$-simplices correspond to morphisms from the $n^{th}$ oriental to $C$ sending the unique non-trivial $n$-cell of the domain to a weakly invertible cell of $C$.
}
\end{abstract}

\selectlanguage{french}

\tableofcontents

\section*{Introduction}

Dans \cite{groth}, Grothendieck introduit le foncteur nerf entre la catégorie des petites catégories et celle des ensembles simpliciaux. Ce foncteur est défini grâce à l'objet cosimplicial qui envoie $[n]$ sur la petite catégorie suivante:
$$ h_0([n]) := 
\xymatrix{
0\ar[r] & 1 \ar[r]& \cdots \ar[r]& n.
}$$
Le nerf d'une petite catégorie $C$ est défini par la formule $\N_{cat}(C)_n:=Hom(h_0([n]),C)$. De plus, ce foncteur admet un adjoint à gauche, qui associe à un ensemble simplicial $X$, la catégorie $h_0(X)$. De nombreuses notions de la théorie des catégories peuvent alors être "traduites" dans le langage des ensembles simpliciaux, ce qui est le point de départ de la théorie  des $(\infty,1)$-catégories. Intéressons nous en particulier à la notion de groupoïde, et à sa "traduction" dans les ensembles simpliciaux. 

L'inclusion $\Lambda^0[2]\to \Delta[2]$ est envoyée par le foncteur $h_0$ sur l'inclusion de catégories:
$$
\xymatrix{
& 1&&   &
& 1 \ar@{}[d]|-{\circlearrowleft}  \ar[rd]^{\sigma_{1,2}}&\\
0\ar@{}[rrrrurr]|-\hookrightarrow\ar[rr]_{\sigma_{0,2}}\ar[ru]^{\sigma_{0,1}}&& 2
&&0\ar[rr]_{\sigma_{0,2}}\ar[ru]^{\sigma_{0,1}}&& 2}
$$
Ainsi, par adjonction, l'ensemble simplicial $\N_{cat}(C)$ a la propriété de relèvement à droite par rapport à $\Lambda^0[2]\to \Delta[2]$, si et seulement si  pour tout couple de morphismes $(f,g)$  de $C$ de même domaine, il existe un morphisme $x:t(g)\to t(f)$ tel que $f = x\circ g$. Dans le formalisme que l'on développera, on dira que  \textit{l'équation  $\textbf{Eq}(\sigma_{0 2} = x\circ \sigma_{0 1})$ admet une solution pour tout choix de paramètre}. On peut alors démontrer simplement que $\N_{cat}(C)$ a cette propriété  de relèvement  si et seulement si tous les morphismes de $C$ admettent des  inverses à gauche. De façon analogue, $\N_{cat}(C)$ a la propriété de relèvement à droite par rapport à $\Lambda^2[2]\to \Delta[2]$,  si et seulement si tous les morphismes de $C$ admettent des  inverses à droite, ou dans notre formalisme, si et seulement si \textit{l'équation  $\textbf{Eq}(\sigma_{0 2} = \sigma_{1 2}\circ x)$ admet une solution pour tout choix de paramètre}. Enfin on en déduit que $C$ est un groupoïde si et seulement si $\N_{cat}(C)$ a la propriété de relèvement à droite par rapport à $\Lambda^i[2]\to \Delta[2]$ pour $i=0,2$.  Le nerf de $C$ a en fait la propriété de relèvement à droite par rapport à toute les inclusions de cornets: 
\begin{thoeremnonnum}[Boardman \& Vogt]
Une catégorie $C$ est un groupoïde si et seulement si l'ensemble simplicial $\N_{cat}(C)$ a la propriété de relèvement par rapport aux inclusions $\Lambda^i[n]\to \Delta[n]$ pour $0\leq i\leq n$.
\end{thoeremnonnum}

Dans \cite{street}, Street définit un nerf de la catégorie des petites $\omega$-catégories \footnote{Une $\omega$-catégorie est la donnée d'un ensemble de $0$-cellule, pour tout couple de $0$-cellules d'un ensemble de $1$-cellules, pour tout couple parallèle  de $1$-cellules, un ensemble de $2$-cellules, etc..., muni de compositions vérifiant des lois d'associativité  et de distributivité strictes. \note{Cette notion est précisément définie dans la définition \ref{def:omegacat}}.} vers la catégorie des ensembles simpliciaux. Ce nerf est construit grâce à l'objet cosimplicial qui à $[n]$, associe le $n$\up{ième} oriental, noté $|[n]|$. Pour les petites dimensions, on peut en donner une représentation graphique:

$$
 \def\arraystretch{2}
\begin{array}{rcl}
[0]&\mapsto &
\xymatrix{
0
}\\
~[1] &\mapsto&
\xymatrix{
0\ar[rr]^{\sigma_{01}}& & 1
} 
\\
\xymatrix{ ~\\ ~ \ar@{}[u]|-{\txt{$[2]$}}} &\xymatrix{ ~\\ ~ \ar@{}[u]|-{\txt{$\mapsto$}}}&
\xymatrix{
&1\ar[rd]^{\sigma_{12}}&\\
0\ar[rr]_{\sigma_{02}}\ar[ru]^{\sigma_{01}}& &\ar@{}[u]^{\rotatebox{90}{$\Rightarrow~~~$}\sigma_{012}~~~~~} 2
}\\
\xymatrix{ ~\\ ~ \ar@{}[u]|-{\txt{$[3]$}}} &\xymatrix{ ~\\ ~ \ar@{}[u]|-{\txt{$\mapsto$}}}&
\xymatrix{
1\ar[rr]^{\sigma_{12}}&&2\ar[rd]^{\sigma_{23}}&
&&&1 \ar[rr]^{\sigma_{12}} 
\ar[rd]|-{\sigma_{13}} && 2 \ar[ld]^{\sigma_{23}}&&&&&&&
\\ 
\ar@{}[rrrrrrrru]|-\Rrightarrow_{\sigma_{0123}}&0\ar[rr]_{\sigma_{03}}\ar[ul]^{\sigma_{01}}\ar[ru]|-{\sigma_{02}}\ar@<-14pt>@{}[u]^------{\rotatebox{140}{$\Rightarrow$}\sigma_{012}} &  &\ar@{}[u]^{\rotatebox{90}{$\Rightarrow~~~$}\sigma_{023}~~~~~} 3
&& 0\ar[rr]_{\sigma_{03}}\ar[ru]^{\sigma_{01}}& &\ar@{}[u]^{\rotatebox{90}{$\Rightarrow~~~$}\sigma_{013}~~~~~} 3 \ar@<-18pt>@{}[u]^------{\rotatebox{40}{$\Rightarrow$}\sigma_{123}} }
\end{array}
$$

Le nerf de Street est alors défini par la formule $\N(C)_n:=Hom(|[n]|,C)$. Par exemple, une $\omega$\nobreakdash -catégorie $C$ a la propriété de relèvement par rapport à l'inclusion $\Lambda^0[2]$ si pour tout couple de $1$-cellules de même domaine $(f,g)$, il existe une $1$-cellule $x$ ainsi qu'une $2$-cellule $y:f\to x*_0 g$. On dira alors que \textit{l'équation $\textbf{Eq}(y:\sigma_{03}\to x*_0 \sigma_{03})$ admet une (pré)-solution pour tout choix de paramètre}. De même, $C$ a la propriété de relèvement par rapport à l'inclusion $\Lambda^0[3]\to \Delta[3]$ si pour tout sextuplet de $1$-cellules $(f,g,h,i,j,k)$, et tout triplet de $2$-cellules $(\alpha,\beta,\gamma)$ telles que $\alpha:g\to i*_0 f$, $\beta:h\to k*_0 g$, et $\gamma:h\to j*_0f$, il existe une $2$-cellule $x:j\to k*_0 i$, ainsi qu'une $3$-cellule 
$$y: (k*_0 \alpha)*_1 \beta \to (x*_0 f)*_1 \gamma.$$
On dira alors que \textit{l'équation 
$$\textbf{Eq}(y: (\sigma_{23}*_0 \sigma_{012})*_1(\sigma_{023})\to (x*_0 \sigma_{01})*_1 \sigma_{013})$$
admet une (pré-)solution pour tout choix de paramètre}.

L’objectif \note{premier} est d'étudier les équations définies par les inclusions de cornets, afin d'en déduire le théorème suivant:
\begin{theolettre}[\ref{theo:theo principale}]
\label{theo:intro1}
Soit $C$ une $\omega$-catégorie. L'ensemble simplicial $\N(C)$ a la propriété de relèvement par rapport aux inclusions ${\Lambda^i[n+1]\to \Delta[n+1]}$ pour tout $n>0$ et tout $0\leq i\leq n$ (resp. pour tout $n>1$ et tout $0<i<n$) si et seulement si les cellules de $C$ de dimension supérieure ou égale à $1$ (resp strictement supérieur à $1$) sont faiblement inversibles.
\end{theolettre}

\note{
Une propriété importante du  nerf catégorique est que l'on peut caractériser son image: un ensemble simplicial y appartient si et seulement si il a la propriété de relèvement unique à droite par rapport aux inclusions de cornets intérieurs. Cependant, pour le nerf de Street, il n'existe pas d'ensemble de monomorphismes $S$ tel qu'un ensemble simplicial soit le nerf de Street d'une $\omega$-catégorie si et seulement si il a la propriété de relèvement unique à droite par rapport aux morphismes de $S$.
Le problème provient du fait que pour une $\omega$-catégorie $C$, il est impossible de distinguer des autres les $n$-simplexes qui correspondent à des morphismes $|[n]|\to C$ qui envoient l'unique $n$-cellule non triviale du domaine sur une unité. Pour pallier ce problème, Robert, puis Street considérèrent un nerf à valeur dans la catégorie des \textit{ensembles simpliciaux stratifiés}. Cette catégorie a pour objet les couples $(K,tK)$ où $K$ est un ensemble simplicial et $tK$ est un sous ensemble des simplexes de $K$ comprenant les simplexes dégénérés. Les morphismes entre $(K,tK)$ et $(L,tL)$ sont les morphismes $f:K\to L$ tel que $f(tK)\subset tL$. Un simplexe dans $tK$ est dit \textit{marqué}. On munit alors $\N(C)$ de la stratification composée des simplexes $|[n]|\to C$ qui envoient l'unique $n$-cellule non triviale du domaine sur une unité. Robert conjecturas alors l'existence d'un ensemble de monomorphismes $S$ tel qu'un ensemble simplicial stratifié soit le nerf de Street d'une $\omega$-catégorie si et seulement si il a la propriété de relèvement unique à droite par rapport aux morphismes de $S$. Cette conjecture a finalement été démontrée par Verity (\cite{verity}). Il est intéressant de noter qu'au niveau des ensembles simpliciaux, ces morphismes sont des inclusions de cornets.

De la même façon que les ensembles simpliciaux ayant la propriété de relèvement à droite par rapport aux inclusions de cornets intérieurs, c'est à dire les quasi-catégories, forment un modèle des $(\infty,1)$-catégories, on voudrait que les ensembles simpliciaux stratifiés ayant la propriété de relèvement à droite par rapport aux morphismes de $S$, appelés \textit{ensembles compliciaux}, soient un modèle des $(\infty,\omega)$-catégories. Dans ce contexte, les simplexes marqués correspondraient à des $n$-cellules (faiblement) inversibles. Cependant, pour que cette interprétation des simplexes marqués soit correcte, il faut ajouter une condition supplémentaire sur la stratification, qui correspond à un analogue de la propriété  ($2$ parmi $6$)\footnote{Si $C$ est une catégorie, une classe $W$ de morphismes vérifie la propriété  ($2$ parmi $6$) lorsque pour tout triplet de morphismes $f,g,h$ tel que $gf$ et $hg$ soient dans $W$, alors $f,g,h$ et $hgf$ sont dans $W$.} que doivent vérifier les $n$-cellules faiblement inversibles. Cela amène à considérer la notion d'\textit{ensemble complicial saturé} (définition \ref{defi:ensembles complicial sature}). Cependant, bien que la stratification évoquée au paragraphe précédent munisse $\N(C)$ d'une structure d'ensemble complicial, elle n'est pas saturée en général.

Dans \cite{riehl}, Riehl évoque la possibilité de considérer d'autres stratifications sur le nerf d'une $\omega$-catégorie, afin d'obtenir une structure d'ensemble complicial saturé. Si $C$ est une $1$-catégorie, une telle stratification est obtenue  en marquant les  $1$-simplexes qui correspondent à des $1$-cellules inversibles  (\cite[proposition 3.1.5]{riehl}). 
De plus, Ozornova et Rovelli ont montré que si $C$ est une $2$-catégorie, une telle stratification est obtenue en marquant les $n$-simplexes qui correspondent à des morphismes $|[n]|\to C$ qui envoient l'unique $n$-cellule non triviale du domaine sur une $1$-cellule $2$-inversible si $n=1$, sur une $2$-cellule inversible si $n=2$ et sur une identité si $n\geq 3$ (\cite[théorème 5.2]{martina}). 
Gagna, Harpaz et Lanari démontre un résultat analogue pour les ensembles simpliciaux échelonnées dans \cite{harpaz}. Grâce à l'étude des inclusions de cornets qu'on se propose d'effectuer, on pourra montrer une généralisation de ces résultats:
\begin{theolettre}[\ref{theo: n(c) est un ensemble complicial}]
\label{theo:intro2}
Soit $C$ une $\omega$-catégorie. Si on définit $t\N(C)$ comme étant l'ensemble des simplexes de $C$ correspondant aux morphismes $|\Delta[n]|\to C$ qui envoient l'unique $n$-cellule non triviale du domaine sur une cellule faiblement inversible (définition \ref{def:definition faiblement inversible}), l'ensemble simplicial stratifié $(\N(C),t\N(C))$ est un ensemble complicial saturé. 
\end{theolettre}

Si on appelle $n$-trivial une $\omega$-catégorie (resp. un ensemble complicial) dont toutes les cellules de dimension strictement supérieures à $n$ sont faiblement inversibles (resp. tous les simplexes de dimension strictement supérieur à $n$ sont marqués), on
 peut  en déduire un analogue du théorème $\ref{theo:intro1}$ pour les ensembles simpliciaux stratifiés:

\begin{corlettre}
Soit $C$ une $\omega$-catégorie. L’ensemble complicial $(\N(C),t\N(C))$ est $n$-trivial si et seulement si $C$ est $n$-trivial.
\end{corlettre}

}

Malheureusement, comme en témoignent les diagrammes présents dans l'article de Street, les orientaux deviennent rapidement très compliqués lorsque la dimension augmente. On doit donc réaliser un important travail préliminaire avant de pouvoir démontrer ces deux théorèmes. 

La première étape va être d'étudier la théorie développée dans \cite{steiner}. Dans cet article, Steiner construit un foncteur $\nu : \CDA\to \infcat$ entre une catégorie composée de complexes de chaînes munis d'une structure additionnelle, appelés les complexes dirigés augmentés, et la catégorie des  $\omega$-catégories strictes. Il montre de plus que ce foncteur admet un adjoint à gauche. Restreint aux complexes dirigés augmentés libres admettant une "bonne" base, il devient une équivalence de catégorie dont le codomaine est composé des $\omega$-catégories admettant un "bon" ensemble de générateurs.

Nous allons construire un autre foncteur $\mu$, entre la catégorie des complexes dirigés augmentés  admettant une "bonne" base et la catégorie des $\omega$-catégories admettant un "bon" ensemble de générateurs, isomorphe à $\nu$, qui utilisera le formalisme  des \textit{chaînes}. C'est alors un  cadre adapté pour définir  un "algorithme" qui exprime les cellules de $\nu K\cong \mu K$ en un composé de générateurs, où $K$ est un complexe dirigé augmenté. 

Les orientaux correspondent alors à l'objet cosimplicial défini par le composé des foncteurs suivants:
$$\Delta \xrightarrow{C_\bullet} \CDAB\xrightarrow{\mu} \infcat$$
où le premier foncteur envoie un ensemble simplicial sur le complexe de chaîne réduit associé, muni d'une structure de complexe dirigé augmenté admettant une "bonne" base.  Grâce à l'algorithme de décomposition, on \note{pourra} alors exprimer les équations que doivent vérifier les $\omega$-catégories pour que leur nerf \note{verifient} la condition de relèvement à droite par rapport aux inclusions de cornets, et après un examen attentif de ces équations,  \note{on en déduira} le théorème \ref{theo:intro1}.

Cette étude approfondie des équations de cornets nous \note{permettra alors de définir la structure d'ensemble complicial saturé sur le nerf de Street d'une $\omega$-catégorie.}

\tocless \subsection*{\textbf{Organisation de l'article.}}
On rappelle dans la première section quelques définitions et résultats sur les $\omega$-catégories et on expose la théorie de Steiner. 

L'objectif de la deuxième section est de construire le foncteur $\mu$ et de donner l'algorithme  de décomposition (théorème \ref{theo:decomposition explicite}).

Dans la troisième section, on présente deux développements. Le premier est un théorème qui donne des conditions suffisantes pour qu'une somme amalgamée dans la catégorie des complexes dirigés induise une somme amalgamée dans les $\omega$-catégories (théorème \ref{theo:theo dans le cas fort quasi-rigide}). Le deuxième est la présentation de la notion d'équation dans une $\omega$-catégorie.

Dans la quatrième section, on se sert du théorème de décomposition  et de la notion d'équation pour montrer que le nerf de Street d'une $\omega$-catégorie $C$ a la propriété de relèvement par rapport aux inclusions de cornets si et seulement si on peut toujours y résoudre des équations d'une certaine forme. L'étude précise de ces équations permet alors de montrer que le nerf de $C$ est un complexe de Kan (respectivement une quasi-catégorie) si et seulement si les $n$-cellules de $C$  pour $n\geq 1$ (respectivement $n>1$) sont faiblement inversibles (théorème  \ref{theo:theo principale}). 

Enfin, la dernière section présente une généralisation de ces résultats aux ensembles compliciaux. On y montre qu'on peut munir $\N(C)$ d'une stratification vérifiant les axiomes des ensembles compliciaux  \note{saturés} (théorème \ref{theo: n(c) est un ensemble complicial}).  L'ensemble complicial $\N(C)$ est alors $n$-trivial si et seulement si  les $k$-cellules  de $C$ pour $k\geq n$ sont faiblement inversibles.

\tocless \subsection*{\textbf{Remerciements.}}
Je tiens à remercier Georges Maltsiniotis, sans qui cet article n'aurait pu exister. C'est lui qui m'a proposé ce problème, et qui, par ses nombreuses relectures attentives, m'a appris à rédiger proprement et rigoureusement. 

\note{Je tiens aussi à remercier François Metayer pour m'avoir communiqué ses notes sur la coinduction, et la.le raporteur.teuse pour les nombreuses, précises et toujours justes remarques et corrections.}

\section{Quelques définitions et rappels}

\subsection{$\omega$-Catégories}
Dans cette partie, on va définir les $\omega$-catégories (strictes) et en donner quelques propriétés. \note{
Cette partie est très inspirée  de \cite{meteyer}, la seule différence est qu'on n'utilise pas le principe de co-induction pour les définitions et résultats liés aux cellules faiblement inversibles.}

\begin{defi}
On définit la  petite catégorie $\textbf{O}$ dont les objets sont les entiers naturels $0,1,2,...$ et dont les morphismes sont engendrés par 
$\delta^-_n,\delta^+_n: n\to n+1$ sujets aux équations:
$$
\begin{array}{rcl}
\delta^+_{n+1}\circ \delta^+_n &= &\delta^-_{n+1}\circ \delta^+_n; \\
\delta^+_{n+1}\circ \delta^-_n &=& \delta^-_{n+1}\circ \delta^-_n.
\end{array}
$$
\end{defi}

\begin{defi}
Un \textit{ensemble globulaire} est un préfaisceau sur $\textbf{O}$. On définit la catégorie des ensembles globulaires: 
$$\textbf{Glob} := \textbf{Set}^{\textbf{O}^{op}}.$$
\end{defi}
Un ensemble globulaire $X$ est donc la donnée d'une famille d'ensembles $X_n:=X(n)$ \note{pour $n\geq 0$}, et de morphismes $d^+_n := X(\delta^+_n) : X_{n+1}\to X_n$ et $d^-_n := X(\delta^-_n) : X_{n+1}\to X_n$  qui vérifient les équations : 
$$
\begin{array}{rcl}
d^+_{n}\circ d^+_{n+1} &= &d^+_{n}\circ d^-_{n+1}; \\
d^-_{n}\circ d^+_{n+1} &=& d^-_{n}\circ d^-_{n+1}.
\end{array}
$$

Une cellule est un élément de l'ensemble $\coprod_{n\in \mathbb{N}} X_n$. Pour une cellule $c$, sa \textit{dimension} est l'entier $n$ tel que $c\in X_n$. On dit alors que $c$ est une \textit{$n$-cellule}. Les morphismes $d^-_n$ et $d^+_n$ sont appelés respectivement $n$-sources et $n$-buts. 

Pour un couple d'entiers  $(n>m)$ et pour $\alpha\in\{-,+\}$, on étend l'application $d_m^\alpha$ à $X_n$:

$$\begin{array}{rccl}
&X_n & \to & X_m\\
~~~~~~~~~~~~~~~~~d_m^\alpha:&x&\mapsto& d_m^{\alpha} \circ d_{m+1}^\alpha \circ ... \circ d^\alpha_{n-1} x
\end{array}$$

Soient $c$ une $n$-cellule de dimension strictement positive et $m$ un entier strictement inférieur à $n$.
La $m$-cellule $a:= d_{m}^-c$ (resp. la $m$-cellule $b:=d_{m}^+c$) est appelée la \textit{$m$-source} (resp. \textit{$m$-but}) de $c$, et on écrit alors $c:a\to_m b$. Dans le cas où $m=n-1$, la cellule $a$ (resp. $b$) est simplement appelée la source (resp. le but) de $c$ et on écrit alors $c:a\to b$.

\begin{defi}
Pour deux entiers $n>m$ et deux $n$-cellules $c$ et $d$, on dit que les cellules $c$ et $d$ sont \textit{$m$-composables} lorsque $d_m^+c = d_m^- d$. Elle sont \textit{$m$-parallèles} lorsque $d_m^+c = d_m^+ d$ et $d_m^-c = d_m^- d$. 
Deux $n$-cellules $(n-1)$-composables (resp. $(n-1)$-parallèles) sont dites \textit{composables} (resp. \textit{parallèles}).
\end{defi}

On peut maintenant définir la notion de $\omega$-catégorie (stricte).
\begin{defi}
\label{def:omegacat}
Une  \textit{$\omega$-catégorie (stricte)} est un ensemble globulaire muni d'opérations de \textit{composition}
$$C_i \times_{C_j} C_i\to C_i ~~~~~~~~~0\leq j <i $$
associant à deux $i$-cellules $j$-composables $c$ et $d$, une $i$-cellule $d*_j c$ ainsi que des \textit{unités} 
$$C_j\to C_i~~~~~~~~~0\leq j <i $$
associant à une $j$-cellule $c$, une  $i$-cellule $1^i_c$. On définit $1_c:= 1^{j+1}_c$. De plus, les compositions et unités doivent satisfaire les axiomes suivants: 
\begin{enumerate}
\item (Associativité) Pour tout couple d'entiers $n>m$ et pour tout triplet de $n$-cellules $(c,d,e)$ telles que les couples $(c,d)$ et $(d,e)$ soient $m$-composables:
$$ e *_m(d*_m c) = ( e *_md )*_m c;$$
\item (Distributivité) Pour tout triplet d'entiers $n>m>k$ et pour toutes  $n$-cellules $c,d,e$ et $f$ telles que les couples $(c,d)$ et $(e,f)$ soient $m$-composables et que les couples $(c,e)$ et $(d,f)$ soient $k$-composables:
$$ (f*_m e)*_k (d*_m  c)= (f*_kd)*_m (e*_k c); $$
\item (Unité I) Pour tout couple d'entiers $n>m$ et toute $n$-cellule $c:a\to_m b$:
$$ 1^n_b*_m c = c*_m 1^n_a = c;$$
\item(Unité II) Pour tout triplet d'entiers $n>m>k$ et tout couple de $m$-cellules $(c,d)$ qui sont $k$-composables:
$$1^n_d*_k 1^n_c = 1^n_{d*_k c};$$
\item (Unité III) Pour tout triplet d'entiers $n>m>k$ et toute $k$-cellule $a$: $$1^n_{1^m_a} = 1^n_a.$$
\end{enumerate}
\end{defi}

\begin{notation}
Pour des entiers $n>m>k$, une $n$-cellule $c :a\to_k b$ et une $m$-cellule $d:b\to_k e$, on note $d*_{k}c$  la $k$-composition $1^{n}_d*_{k} c$.
 De façon symétrique, pour une $m$-cellule $c :a\to_k b$ et une $n$-cellule $d:b\to_ke$, on note $d*_{k}c$  la $k$-composition $d*_{k} 1^{n}_c$.
\end{notation}

\begin{defi}
La catégorie $\infcat$ a comme objets les $\omega$-catégories, et comme morphismes les morphismes d'ensembles globulaires qui préservent les compositions et les unités. 
\end{defi}

\note{
\begin{defi}
\label{def:definition faiblement inversible}
Soit $C$ une $\omega$-catégorie. Un \emph{ensemble d'inversibilité} est un ensemble $E\subset \coprod_{n>0} C_n$ 
tel que 
pour tout $n$ et toute $(n+1)$-cellule $b\in E$, il existe $\tilde{b},c,c'\in E$ telles que $\tilde{b} \in C_{n+1}$, $c,c'\in C_{n+2}$ et
$$
\begin{array}{rcl}
c &:& 1_{d^-_n b}  \to \tilde{b}*_{n}b \\
c' &:&1_{d^+_n b} \to       b*_{n}\tilde{b}\\
\end{array}
$$
Une cellule de dimension strictement positive $a$ est \textit{faiblement inversible} s'il existe un ensemble d'inversibilité $E$ tel que $a\in E$. La cellule $\tilde{a}$ est alors appelée un \textit{inverse faible} de $a$.
\end{defi}

\begin{lem}
\label{lem:characterisation des cellules faiblement inversible type coinduction}
Une $n$-cellule $a$ est faiblement inversible si et seulement si il existe une $n$-cellule $\tilde{a}:d^+_{n-1}a\to d^-_{n-1}a$ ainsi que  deux $(n+1)$-cellules faiblement inversibles $c,c'$ telles que:
$$
\begin{array}{rcl}
c &:& 1_{d^-_{n-1} a}  \to \tilde{a}*_{n-1}a \\
c' &:&1_{d^+_{n-1} a} \to       a*_{n-1}\tilde{a}.\\
\end{array}
$$
\end{lem}
\begin{proof}
Supposons tout d'abord que $a$ est faiblement inversible. Il existe donc un ensemble d'inversibilité $E$ comprenant $a$. Par définition des ensembles d'inversibilité, il existe $\tilde{a},c,c'\in E$ ayant les sources et buts désirés. Enfin, $c$ et $c'$ étant comprises dans $E$, elles sont faiblement inversibles.

Réciproquement, supposons qu'il existe $\tilde{a},c,c'$ vérifiant les conditions voulues. Soient $E$ et $E'$ deux ensembles d'inversibilité tels que $c\in E$ et $c'\in E'$. Remarquons alors que 
$$E'':= \{a,\tilde{a}\}\cup E\cup E'$$
est un ensemble d'inversibilité comprenant $a$. La cellule $a$ est donc faiblement inversible.
\end{proof}
\begin{rem}
Le lemme précédent implique que l'ensemble des cellules faiblement inversibles est un ensemble d'inversibilité. Il est alors l'ensemble d'inversibilité maximal.
\end{rem}

\begin{lem}
\label{lem:les uniés sont faiblement inversible}
Les unités sont faiblement inversibles.
\end{lem}
\begin{proof}
Il suffit de remarquer que l'ensemble des unités est un ensemble d’inversibilité.
\end{proof}

\begin{lem}
\label{lem:cellules faiblment inversible stable par composition horizontales}
Soient $k$ un entier, et $a$ et $b$ deux cellules faiblement inversibles de dimension strictement supérieures à $(k+1)$.
Si $a$ et $b$ sont $k$-composables, alors  $b*_ka$ est faiblement inversible.
\end{lem}
\begin{proof}
Soient $E$ un ensemble d'inversibilité comprenant $a$ et $E'$ un ensemble d'inversibilité comprenant $b$. On définit $E''$ comme étant l’ensemble des cellules pouvant s'écrire sous la forme $f*_k e$ où $e$ (resp. $f$) est une cellule dans $E$  (resp. dans $E'$), de dimension strictement supérieure à $(k+1)$. La cellule $b*_ka$  est comprise dans $E''$, et il suffit donc de montrer que ce dernier est un ensemble d'inversibilité. On se donne donc une $(n+1)$-cellule $g\in E''$. Par définition, il existe $e\in E$ et $f\in E'$ telles que $g:=f*_ke$. Il existe donc $\tilde{e},c,c'\in E$  et $\tilde{f},d,d'\in E$ telles que 
$$
\begin{array}{rclrrcl}
c &:& 1_{d^-_{n} e}  \to \tilde{e}*_{n}e &~~~~&      d &:& 1_{d^-_{n} f}  \to \tilde{f}*_{n}f\\
c' &:&1_{d^+_{n} e} \to       e*_{n}\tilde{e} &~~~~& d' &:&1_{d^+_{n} f} \to   f*_{n}\tilde{f}.\\
\end{array}
$$
Si on définit 
$$\tilde{g}=\tilde{f}*_k \tilde{e},~~~~~~ l:=d*_k c,~~~~~~ l':=d'*_k c'$$ on a alors: 
$$
\begin{array}{rcl}
l &:& 1_{d^-_{n} g}  \to \tilde{g}*_{n}g \\
l' &:&1_{d^+_{n} g} \to       g*_{n}\overline{g}.\\
\end{array}
$$ 
On a donc montré que $E''$ est un ensemble d'inversibilité, ce qui implique que $b*_ka$ est faiblement inversible.
\end{proof}

\begin{lem}
\label{lem:cellules faiblment inversible stable par composition verticales}
Soient $n$ un entier,  et $a$ et $b$ deux $(n+1)$-cellules faiblement inversibles. Si $a$ et $b$ sont $n$-composables, alors  $b*_{n}a$ est faiblement inversible.
\end{lem}
\begin{proof}
On va montrer que l'ensemble $E$, composé des éléments de la forme $a*_{k}b$ où $k$ est un entier quelconque  et $a$ et $b$ sont des $(k+1)$-cellules faiblement inversibles, est un ensemble d'inversibilité. Cela impliquera le résultat. Soit $e$ une $(k+1)$-cellule de $E$. Il existe donc deux $(k+1)$-cellules faiblement inversibles $a$ et $b$ telles que $e:= b*_{k}a$. Soient $\tilde{a},c,c'$ et $\tilde{b},d,d'$ vérifiant:
$$
\begin{array}{rclrrcl}
c &:& 1_{d^-_{k} a}  \to \tilde{a}*_{k}a &~~~~&   d &:& 1_{d^-_{k} b}  \to \tilde{b}*_{k}b\\
c' &:&1_{d^+_{k} a} \to  a*_{k}\tilde{a} &~~~~&   d' &:&1_{d^+_{k} b} \to   b*_{k}\tilde{b}.\\
\end{array}
$$
On définit alors 
$$\tilde{e}=\tilde{a}*_k\tilde{b},~~~~l:= (\tilde{a}*_k d *_k a)*_{k+1} c,~~~~ l':= (b*_k c' *_k \tilde{b})*_{k+1} d'.$$
Le lemme \ref{lem:cellules faiblment inversible stable par composition horizontales} implique que $\tilde{a}*_k d *_k a$ et $b*_k c' *_k \tilde{b}$ sont faiblement inversibles. Les $(k+2)$-cellules $l$ et $l'$ étant des $(k+1)$-compositions de $(k+2)$-cellules faiblement inversibles, elles sont dans $E$. Enfin, par construction on a:
$$
\begin{array}{rcl}
l &:& 1_{d^-_{k+1} e}  \to \tilde{e}*_{k+1}e   \\
l' &:&1_{d^+_{k+1} e} \to      e*_{k+1}\tilde{e}. \\
\end{array}
$$
L'ensemble $E$ est donc bien un ensemble d'inversibilité.
\end{proof}

\begin{defi}
Deux $n$-cellules parallèles $a,b$ sont \textit{$\omega$-équivalentes} s'il existe une $(n+1)$-cellule inversible $c:a\to b$. On note cette relation $\sim$.
\end{defi}

\begin{prop}
Soient deux entiers $n>m\geq 0$. La relation $\sim$ satisfait les propriétés suivantes: 
\begin{enumerate}
	\item (Réflexivité) Pour toute $n$-cellule $a$,  $$a\sim a;$$
	\item (Symétrie) Pour tout couple de $n$-cellules parallèles $(a,b)$, $$a\sim b \mbox{ implique } b\sim a;$$
	\item (Transitivité) Pour tout triplet de $n$-cellules $(a,b,c)$   deux à deux parallèles,  
	$$a\sim b \mbox{ et } b\sim c \mbox{ implique } a\sim c;$$
	\item (Compatibilité avec les compositions) Pour tout couple de $n$-cellules $m$-composables $(a,b)$ et $(c,d)$,
	 $$b\sim d \mbox{ et }a\sim c \mbox{ implique } b*_ma\sim d*_mc.$$
\end{enumerate}
\end{prop}
\begin{proof}
La réflexivité est une conséquence du lemme \ref{lem:les uniés sont faiblement inversible}, le symétrie est immédiate, le transitivité provient du lemme \ref{lem:cellules faiblment inversible stable par composition verticales}, et enfin, le  lemme \ref{lem:cellules faiblment inversible stable par composition horizontales} implique la compatibilité de $\sim$ avec les compositions. 
\end{proof}

\begin{lem}
\label{lem:characterisation des cellules faiblement inversible type coinduction, numéro 2}
Une $(n+1)$-cellule $a$ est faiblement inversible si et seulement si il existe une $(n+1)$-cellule $\tilde{a}:d^+_{n}a\to d^-_{n}a$ telle que $1_{d^-_{n} a}  \sim \tilde{a}*_{n}a$ et $1_{d^+_{n} a} \sim  a*_{n}\tilde{a}$.
\end{lem}
\begin{proof}
C'est une conséquence immédiate de la définition de $\sim$ et du lemme \ref{lem:characterisation des cellules faiblement inversible type coinduction}.
\end{proof}

\begin{rem}
Dans \cite{meteyer}, la notion de cellule faiblement inversible est défini par co-induction, et les propriétés basiques que les cellules faiblement inversibles vérifient sont aussi démontrées par co-induction. Cependant, il n'est pas précisé quelles sont les règles pour utiliser ce principe de déduction. C'est une note non publiée de Metayer, dont l'objet est d'expliciter ce principe, qui a inspiré la présentation alternative donnée dans cet article.
\end{rem}

}

\begin{defi}
Soit $c : a\to b$ une $n$-cellule de dimension strictement positive.
\begin{enumerate}
\item La cellule $c$ est \textit{faiblement inversible à gauche} lorsqu'il existe $\tilde{c}:b\to a$ telle que $\tilde{c} *_{n-1} c \sim 1_b$. La cellule $\tilde{c}$ est alors un \textit{inverse faible à gauche} de $c$.
\item La cellule $c$ est \textit{faiblement inversible à droite} lorsqu’il existe $\tilde{c}:b\to a$ telle que $c *_{n-1} \tilde{c} \sim 1_a$. La cellule $\tilde{c}$ est alors un \textit{inverse faible à droite} de $c$.
\end{enumerate}
\end{defi}

\begin{rem}
Une cellule faiblement inversible à droite et à gauche est  faiblement inversible. En effet, soient $n$ un entier strictement positif et $c : a\to b $ une $n$-cellule admettant un inverse faible à droite $\tilde{c}$ et un inverse faible à gauche $\tilde{c}'$. 
On a  alors 
$$ \tilde{c}*_{n-1} c \sim \tilde{c}' *_{n-1}c *_{n-1} \tilde{c}*_{n-1} c \sim \tilde{c}'*_{n-1} c \sim 1_a.$$
La cellule $\tilde{c}$ est donc aussi un inverse à gauche, et $c$ est donc faiblement inversible.
\end{rem}

\begin{prop}
\label{prop:propdedivision a droite} Une $n$-cellule $a$   faiblement inversible a la \emph{propriété de division à droite} :
\begin{enumerate}
\item Pour toute $n$-cellule  $ b$ telle que $d_{n-1}^- a  = d_{n-1}^- b$, il existe une $n$-cellule $x$ telle que 
$$x*_{n-1} a \sim b.$$
De plus si $y$ vérifie la même relation, alors $y\sim x$. On dit alors que la solution est \emph{faiblement unique}.

\item  Soit $m> n$. Pour toute  $m$-cellule $b$ et  tout couple de $(m-1)$-cellules parallèles $s,t$ tel que $s*_{n-1}a = d_{m-1}^-b$ et $t*_{n-1}a = d_{m-1}^+b$, il existe  une $m$-cellule $x : s\to t$ qui vérifie:
$$x*_{n-1} a \sim b.$$
De plus la solution est faiblement unique.

\end{enumerate}
Similairement, la  $n$-cellule $a$ a la propriété de division à gauche.
\end{prop}
\begin{proof}
Pour $a$ une $n$-cellule faiblement inversible, on définit les propositions suivantes:
$$
\def\arraystretch{1.4}
\begin{array}{lcl}
\mathcal{E}_{n,m}(a) &:\equiv&\mbox{La cellule $a$ a la proriété de division à droite et à gauche pour les $m$-cellules;
 }\\
 \mathcal{E}_{n,m} &:\equiv&\mbox{Pour  toute  $n$-cellule faiblement inversible $a$, $\mathcal{E}_{n,m}(a)$; }
 \\
\mathcal{U}_{n,m}(a)&:\equiv& \mbox{La division à droite et à gauche des $m$-cellules par $a$ est faiblement unique;} \\
\mathcal{U}_{n,m} &:\equiv&\mbox{Pour   $n$-cellule faiblement inversible $a$, $\mathcal{U}_{n,m}(a)$. }
\end{array}$$
\note{Le premier point de la proposition correspond donc à la conjonction de $\mathcal{E}_{n,n}$ et $\mathcal{U}_{n,n}$ pour tout $n$, et le deuxième à la conjonction de $\mathcal{E}_{n,m}$ et $\mathcal{U}_{n,m}$ pour tout $n<m$.}

On peut démontrer directement $\mathcal{E}_{n,n}$ et $\mathcal{U}_{n,n}$ pour tout $n$. En effet, soient $a$ et $b$ deux $n$-cellules vérifiant les conditions du premier point. On note $\tilde{a}$ un inverse faible de $a$.
On a alors
 $$x*_{n-1} a \sim b\mbox{~~ si et seulement si~~ }x \sim b*_{n-1} \tilde{a}.$$ Cela prouve alors à la fois l'existence et l'unicité faible de la solution.
 
On va maintenant montrer  $\mathcal{E}_{n,n+k}$ et $\mathcal{U}_{n,n+k}$ pour tout $n$ et tout $k$ par récurrence sur $k$.  
L'initialisation correspond au cas $k=1$. Donnons nous une $n$-cellule $a$ et une $(n+1)$-cellule $b$ vérifiant les conditions voulues. Soient $\tilde{a}$ un inverse faible de $a$ et $r : a*_{n-1} \tilde{a}\xrightarrow{\sim} 1_{d^+_{n-1}a} $ une cellule faiblement inversible.

Commençons par montrer $\mathcal{U}_{n,n+1}(a)$. On suppose donc qu'il existe une $(n+1)$-cellule $x:s\to t$ qui vérifie $x*_{n-1} a \sim b$. On a alors $$ x *_n (s*_{n-1}r)=  (t*_{n-1} r)  *_n (x*_{n-1}a*_{n-1}\tilde{a})  \sim (t*_{n-1} r)  *_n (b*_{n-1}\tilde{a}) . $$ Les cellules $x$  et $s*_{n-1}r$ sont de même dimension et la cellule $s*_{n-1}r$ est faiblement inversible. On peut donc utiliser $\mathcal{U}_{n+1,n+1}(s*_{n-1}r)$ qui implique que  la cellule $x$, si elle existe, est faiblement unique. De façon analogue, on montre l'unicité faible de la division à gauche par $a$. On a donc prouvé pour tout $n$,  $\mathcal{U}_{n,n+1}$.

Montrons maintenant $\mathcal{E}_{n,n+1}(a)$. Remarquons  que l'on a :
$$
\begin{array}{rcl}
d_{n}^-((t*_{n-1} r)  *_n (b*_{n-1}\tilde{a}))&= &d_{n}^-(b*_{n-1}\tilde{a}) = s*_{n-1}*a*_{n-1}\tilde{a}\\
d_{n}^-(s*_{n-1}r) &=&s*_{n-1}*a*_{n-1}\tilde{a}
\end{array}
$$

On peut alors utiliser $\mathcal{E}_{n+1,n+1}(s*_{n-1}r)$, qui implique l'existence d'une $(n+1)$-cellule $y:s\to t$ vérifiant $$ y *_n (s*_{n-1}r)  \sim (t*_{n-1} r)  *_n (b*_{n-1}\tilde{a}) . $$
Or 
$$  y *_n (s*_{n-1}r)=  (t*_{n-1} r)  *_n (y*_{n-1}a*_{n-1}\tilde{a})$$
Les cellules $y*_{n-1}a*_{n-1}\tilde{a}$ et $t*_{n-1} r$ sont de même dimension, et $t*_{n-1} r$ est faiblement inversible.
La proposition $\mathcal{U}_{n+1,n+1}(t*_{n-1} r)$ implique que $(y*_{n-1}a*_{n-1}\tilde{a}) \sim (b*_{n-1}\tilde{a})$. 
Or, comme la dimension de $\tilde{a}$ est $n$, on peut utiliser $\mathcal{U}_{n,n+1}(\tilde{a})$ qui implique alors que $y*_{n-1} a\sim b $. La cellule $y$ est donc bien une solution. On a donc prouvé $\mathcal{E}_{n,n+1}$ pour tout $n$.

Supposons maintenant $\mathcal{E}_{n,n+k}$ et $\mathcal{U}_{n,n+k}$ pour tout $n$. Donnons nous une $n$-cellule $a$ et une ${(n+k+1)}$-cellule $b$ vérifiant les conditions voulues. Soient $\tilde{a}$ un inverse faible de $a$ et $r : a*_{n-1} \tilde{a}\xrightarrow{\sim} 1_{d^+_{n-1}a} $ une cellule faiblement inversible.

Commençons par montrer $\mathcal{U}_{n,n+k+1}(a)$. Pour cela, on suppose qu'il existe une ${(n+k+1)}$-cellule $x:s\to t$ qui vérifie $x*_{n-1} a \sim b$. On a alors,  pour tout $\alpha\in\{-,+\}$, $d_{n-1}^\alpha x=d_{n-1}^\alpha s =d_{n-1}^\alpha t$, et 
 $$ x *_n (d_{n-1}^-s*_{n-1}r)=  (d_{n-1}^+t*_{n-1} r)  *_n (x*_{n-1}a*_{n-1}\tilde{a})  \sim (d_{n-1}^+t *_{n-1} r)  *_n (b*_{n-1}\tilde{a}) . $$ La cellule $d_{n-1}^-s*_{n-1}r$ est de dimension $n+1$ et est faiblement inversible. La proposition $\mathcal{U}_{n+1,n+k+1}(d_{n-1}^-s*_{n-1}r)$ implique alors que $x$, s'il existe, est faiblement unique. On a donc montré $\mathcal{U}_{n,n+k+1}$ pour tout $n$.

Prouvons maintenant $\mathcal{E}_{n,n+k+1}(a)$. Remarquons cette fois qu'on a:
$$\def\arraystretch{1.4}
\begin{array}{rcl}
d_{n+k}^-((d_{n-1}^+t *_{n-1} r)  *_n (b*_{n-1}\tilde{a})) &=&  (d_{n-1}^+t *_{n-1} r)  *_n (s*_{n-1}a*_{n-1}\tilde{a})\\
 &=&  (d_{n-1}^+s *_{n-1} r)  *_n (s*_{n-1}a*_{n-1}\tilde{a})\\
 &=& s*_{n-1}r  =   s*_{n} (d_{n-1}^-s *_{n-1} r)\\
 d_{n+k}^+((d_{n-1}^+t *_{n-1} r)  *_n (b*_{n-1}\tilde{a})) &=&  (d_{n-1}^+t *_{n-1} r)  *_n (t*_{n-1}a*_{n-1}\tilde{a})\\
 &=& t*_{n-1}r  \\&= &  t*_{n} (d_{n-1}^-t *_{n-1} r)= t*_{n} (d_{n-1}^-s *_{n-1} r)
\end{array}
$$
Par hypothèse de récurrence, on peut utiliser  $\mathcal{E}_{n+1,n+k+1}(d_{n-1}^-s*_{n-1}r)$ qui implique qu'il existe $y:s\to t$ vérifiant $$ y *_n (d_{n-1}^-s*_{n-1}r)  \sim (d_{n-1}^+t *_{n-1} r)  *_n (b*_{n-1}\tilde{a}) . $$ 
Or 
$$y *_n (d_{n-1}^-s*_{n-1}r) =  (d_{n-1}^+t*_{n-1} r)  *_n (y*_{n-1}a*_{n-1}\tilde{a}).$$
La proposition $\mathcal{U}_{n+1,n+k+1} (d_{n-1}^+t*_{n-1} r)$ implique que $(y*_{n-1}a*_{n-1}\tilde{a}) \sim (b*_{n-1}\tilde{a})$. 
La cellule $\tilde{a}$ est de dimension $n$. On peut donc utiliser $\mathcal{U}_{n,n+k+1}(\tilde{a})$ et on obtient  $y*_{n-1} a\sim b $. La cellule $y$ est donc une solution. On a donc prouvé $\mathcal{E}_{n,n+k+1}$.
\end{proof}

\begin{cor}
\label{cor;compatibilite de faible inversible}
Soient deux entiers $m\geq n>0$,  $a$ une $n$-cellule faiblement inversible  et $b$ une $m$-cellule telles que $a$ et $b$ soient $(n-1)$-composables. Alors si $b *_{n-1} a$ est faiblement inversible, $b$ est aussi faiblement inversible. 
\end{cor}

\begin{proof}
On se place tout d'abord dans le cas où $m=n$. Soit $\tilde{a}$ un inverse faible de $a$. On a alors $b\sim (b *_{n-1} a)*_{n-1} \tilde{a}$. La propriété d'être faiblement inversible étant stable par composition, cela prouve que $b$ est faiblement inversible. Supposons maintenant que $m>n$. 
Notons $c$ un inverse faible de la $m$-cellule $b *_{n-1} a$. On a donc $d_{m-1}^-c =d_{m-1}^+ b *_{n-1}a   $ et $d_{m-1}^+c =d_{m-1}^- b *_{n-1}a $. Par l'existence de la division à droite par $a$, il existe $b':d_{m-1}^+ b\to d_{m-1}^- b$ telle que $b'*_{n-1}a \sim c$.  On a  alors $(b  *_{m-1}b') *_{n-1} a  = (b *_{n-1} a) *_{m-1}(b' *_{n-1} a) \sim 1_{d^+_{m-1} b} *_{n-1} a $. Par l'unicité de la division à droite par $a$, cela implique que $ b  *_{m-1}b'\sim 1_{d^+_{m-1} b}$. On montre de façon analogue que $ b' *_{m-1}b \sim 1_{d^-_{m-1} b}$. La $m$-cellule $b$ est donc bien faiblement inversible. 
\end{proof}

On a en fait montré quelque chose d'un peu plus fort : en reprenant les notations de l’énoncé, si  $b *_{n-1} a$ est faiblement inversible à droite (resp. à gauche) alors $b$ est faiblement  inversible à droite (resp. à gauche)

\begin{defi}
Une $\omega$ catégorie est \textit{$n$-triviale} lorsque toutes les cellules de dimension strictement supérieure $n$ sont faiblement inversibles.
\end{defi}

\subsection{Rappel de la théorie de Steiner}
Tous les résultats de cette section sont dus à Steiner \cite{steiner}.

\begin{defi}
Un \textit{complexe dirigé augmenté} $(K,K^*,e)$ est la donnée d'un complexe de groupes abéliens $K$, avec une  augmentation $e$: $$\mathbb{Z}\xleftarrow{e} K_0 \xleftarrow{\partial_0} K_1\xleftarrow{\partial_1} K_2 \xleftarrow{\partial_2} K_3\xleftarrow{\partial_3}... $$
et d'un ensemble gradué $K^* = (K^*_n)_{n\in\mathbb{N}}$ tel que  pour tout $n$, $K_n^*$ est un sous-monoïde de $K_n$. Un morphisme de complexes dirigés entre $(K,K^*,e)$ et $(L,L^*,e')$ est la donnée d'un morphisme de complexes augmentés de groupes abéliens: $f : (K,e)\to (L,e')$ tel que $f(K^*_n)\subset L^*_n$ pour tout $n$. On note \CDA la catégorie des complexes dirigés augmentés. 
\end{defi}

Steiner construit alors une adjonction 
$$\lambda : \xymatrix{ \infcat \ar@/^13pt/[r]& \ar@/^13pt/[l] \CDA } : \nu. $$
Le foncteur $\lambda$ est le plus simple à définir :

\begin{defi}
Soit $C$ une  $\omega$ catégorie.
On note $(\lambda C)_n $ le groupe abélien engendré par l’ensemble $\{[x]_n: x\in C_n\}$ et les relations 
$$ [x*_m y]_n \sim [x]_n + [y]_n  \mbox{ pour $m<n$ }.$$

On définit  le morphisme $\partial_n : (\lambda C)_{n+1} \to (\lambda C)_n$ sur les générateurs par la formule: 
$$\partial_n([x]_{n+1}) := [d_n^+ x]_{n} - [d_n^- x]_{n}.$$
\end{defi}

Le morphisme $\partial$ est alors une différentielle. On  définit une augmentation $e :  (\lambda C)_{0}\to \mathbb{Z}$ en posant sur les générateurs :  $ e([x]_0)  =1$. 
Soit $(\lambda C)_n^*$  le sous-monoïde additif engendré par les éléments $[x]_n$. Ces données définissent  un complexe dirigé augmenté $\lambda C := (\{(\lambda C)_n \}_{n\in \mathbb{N}},\{(\lambda C)^*_n \}_{n\in \mathbb{N}},e )$. Cette assignation  se relève en un  foncteur: 
$$\begin{array}{ccccc}
\lambda &:& \infcat&\to&\CDA\\
&&C&\mapsto&\lambda C.
\end{array}$$

On va maintenant définir le foncteur $\nu$. Dans la suite, on se fixe un complexe dirigé augmenté $(K,K^*,e)$.

\begin{defi}
Un \textit{tableau de Steiner} (ou plus simplement tableau) de dimension $n$ est la donnée d'une double suite finie: 
$$\left(\begin{matrix}
x^-_0 &x^-_1&x^-_2&x^-_3 &...&x_n^-\\
x^+_0 &x^+_1&x^+_2&x^+_3 &...&x_n^+
\end{matrix}\right)$$
telle que 
\begin{enumerate}
\item $x^-_n=x^+_n$;
\item Pour tout $i\leq n$ et $\alpha\in\{-,+\}$, $x_i^\alpha$ est un élément de $K^*_i$;
\item Pour tout $0<i\leq n$, $\partial_{i-1}(x_i^\alpha)= x_{i-1}^+ - x_{i-1}^-$; 
\end{enumerate}
Un tableau est dit \textit{cohérent} si $e(x^+_0) = e(x^-_0) = 1$.
\end{defi}

\begin{defi}
On définit l'ensemble globulaire $\nu K$ dont les $n$-cellules sont les tableaux cohérents de dimension $n$. Les applications sources et buts sont définies pour $k<n$ par la formule: $$d^\alpha_k\begin{pmatrix}
x^-_0 &x^-_1&x^-_2&...&x^-_n\\
x^+_0 &x^+_1&x^+_2&...& x^+_n
\end{pmatrix} = \begin{pmatrix}
x^-_0 &x^-_1&x^-_2&...& x^-_{k-1}&x^\alpha_k\\
x^+_0 &x^+_1&x^+_2&...& x^+_{k-1}&x^\alpha_k\end{pmatrix}$$
\end{defi}

On munit l'ensemble globulaire $\nu K$ d'une structure de $\omega$-catégorie:

\begin{defi}
On a une structure évidente de groupe sur les tableaux: 
$$\begin{pmatrix}
x^-_0 &x^-_1&...& x^-_n\\
x^+_0 &x^+_1&...& x^+_n
\end{pmatrix}
+
\begin{pmatrix}
y^-_0 &y^-_1&...& y^-_n\\
y^+_0 &y^+_1&...& y^+_n
\end{pmatrix}
=
\begin{pmatrix}
x^-_0+y^-_0 &x^-_1+ y^-_1&...&x^-_n+ y^-_n \\
x^+_0+y^+_0 &x^+_1+ y^+_1&...&x^+_n +y^+_n 
\end{pmatrix}
$$
\label{defi:definition des composition et unites de nu k}

\begin{itemize}
\item
Pour deux tableaux cohérents $x$ et $y$ tels que $d^-_k(x) =d^+_k(y) = z$, on définit leur $k$-composition par la formule suivante : 
$$x*_k y := x- z + y .$$ Plus explicitement : 
$$\begin{pmatrix}
x^-_0 &...& x^-_n\\
x^+_0 &...& x^+_n
\end{pmatrix}
*_k
\begin{pmatrix}
y^-_0 &...& y^-_n\\
y^+_0 &...& y^+_n
\end{pmatrix}
 := 
\begin{pmatrix}
y^-_0&...&y_k^-& y_{k+1}^- + x_{k+1}^- & ...& y_{n}^- + x_{n}^-\\
x^+_0 &...&x_k^+& y_{k+1}^+ + x_{k+1}^+ & ...& y_{n}^+ + x_{n}^+ 
\end{pmatrix}
$$

\item
Pour un entier $m>n$, on définit le tableau $1^m_x$ de taille $m$ : 
$$1^m_x :=
\begin{pmatrix}
x^-_0 &...& x^-_n& 0 &...&0\\
x^+_0 &...& x^+_n& 0 &...&0	
\end{pmatrix}$$
\end{itemize}

L'ensemble globulaire $\nu K$, muni des compositions et unités de la définition \ref{defi:definition des composition et unites de nu k} est alors une $\omega$-catégorie. 
\end{defi}

\begin{defi}
On définit le foncteur $\nu : \CDAB\to\infcat$ qui associe à un complexe dirigé augmenté $K$, la $\omega$-catégorie $\nu K$, et à un morphisme de complexes dirigés augmentés $f :K\to L$, le morphisme de $\omega$-catégories
$$
\begin{array}{rccc}
\nu f : &\nu K &\to& \nu L\\
& \left(\begin{matrix}
x^-_0  &...&x_n^-\\
x^+_0&...&x_n^+
\end{matrix}\right) 
&\mapsto&
\left(\begin{matrix}
f_0(x^-_0)  &...&f_n(x_n^-)\\
f_0(x^+_0)&...&f_n(x_n^+)
\end{matrix}\right) 
\end{array}
$$

\end{defi}

\begin{theo}
\label{theo:ajdonction de steiner avec unite et counite explicite}
Les foncteurs $\lambda$ et $\nu$ sont adjoints l'un de l'autre : 
$$\lambda : \xymatrix{
    \infcat  \ar@/^/[r] \ar@{}[r]|-{\bot}  &
    \CDA \ar@/^/[l] }: \nu.$$
Pour une $\omega$-catégorie $C$, l'unité de l'adjonction est donnée par la transformation naturelle : 
$$\begin{array}{rrcl}
~~~~~\eta :&  C &\to & \nu \lambda C \\
& x\in C_n &\mapsto  & 
\begin{pmatrix}
[d^-_0(x)]_0&...&[d^-_{n-1}(x)]_{n-1}&[x]_n\\
[d^+_0(x)]_0&...& [d^+_{n-1}(x)]_{n-1}&[x]_n
\end{pmatrix}
\end{array}
$$
Pour un complexe dirigé augmenté $K$, la co-unité est donnée par :
$$\begin{array}{rrcl}
\pi :&  \lambda \nu K &\to & K~~~~~~~~~~~~~~~~ \\
& [x ]_n \in (\lambda \nu K)_n&\mapsto  &  x_n^+ = x_n^-
\end{array}
$$

\end{theo}
\begin{proof}
Voir \cite[théorème 2.11]{steiner}.
\end{proof}

On va maintenant s’intéresser à une sous catégorie des complexes dirigés augmentés correspondant à ceux qui admettent une "bonne base".
\begin{defi}
Une \textit{base}  pour un complexe dirigé augmenté $(K,K^*,e)$ est la donnée d'un ensemble gradué $B = (B_n)_{n\in\mathbb{N}}$ tel que pour tout $n$, $B_n$ soit à la fois une base du monoïde $K_n^*$ et  du groupe $K_n$. 
\end{defi}

\begin{rem}
	Les éléments de $B_n$ peuvent être caractérisés comme les éléments minimaux de  $K_n^*\backslash\{0\}$ pour la relation d'ordre suivante: 
	$$x\leq y \mbox{ ssi } y-x \in K_n^*$$
	Cela prouve que si une base existe, elle est unique. 
\end{rem}

Tout élément de  $K_n$ peut alors s'écrire de façon unique comme une somme $\Sigma_{b\in B_n} \lambda_b b$. Cela nous incite à définir de nouvelles opérations : 

\begin{defi}
Pour un élément $x := \sum_{b\in B_n} \lambda_b b$ de $K_n$, on définit la \textit{partie positive}  et la \textit{partie négative}: 
$$
\begin{array}{rcl}
(x)_+ &:=& \sum_{b\in B_n, \lambda_b> 0} ~\lambda_bb\\
(x)_- &:=& \sum_{b\in B_n, \lambda_b< 0} -\lambda_bb
\end{array}
$$
On a alors  $x = (x)_+ - (x)_-$. Un élément $x$ est \textit{positif} (resp. \textit{négatif}) \note{lorsque $x =(x)_+$} (resp. \note{lorsque $x =-(x)_-$}).
Soit $y = \sum_{b\in B_n} \mu_b b$, on définit : 
$$
\begin{array}{rcl}
x\vee y &:=& \sum_{b\in B_n} \mbox{ max}(\lambda_b, \mu_b) ~b\\
x\wedge y &:=& \sum_{b\in B_n} \mbox{ min}(\lambda_b, \mu_b)~ b \\
x \diagdown y &:=& (x-(y)_+)_+
\end{array}
$$
\end{defi}
En utilisant ces notations, on pose : 
$$
\begin{array}{rcl}
\partial_n^+(\var) &:=& (\partial_n(\var))_+ : K_{n+1}\to K^*_n\\
\partial_n^-(\var) &:= &(\partial_n(\var))_- : K_{n+1}\to K^*_n
\end{array}
$$
Lorsqu'un élément $b$ de la base est dans le support de $x$, c'est-à-dire que $\lambda_b\neq 0$, on dit que $b$ appartient à $x$, ce que l'on note $b\in x$.

\begin{defi}
\label{defi:tableau_de_steiner_associe_a_un singloton}
Soit $a\in K^*_n$. On définit par récurrence descendante sur $k\leq n$ : 
 $$ \begin{array}{rclc}
 \langle a\rangle_k^\alpha &:= & a & \mbox{si $k = n$}\\
 &:= & \partial_k^\alpha\langle a\rangle^\alpha_{k+1} & \mbox{sinon}
\end{array} 
$$
Le tableau associé à $a$ est alors : 
$$\langle a\rangle := \begin{pmatrix}
\langle a\rangle^-_0 &...&\langle a\rangle^-_{n-1}&a\\
\langle a\rangle^+_0 &...&\langle a\rangle^+_{n-1}&a
\end{pmatrix}$$
\end{defi}
 
\begin{defi}
La base est dite \textit{unitaire} lorsque pour tout $b\in B$, le tableau $\langle b\rangle$ est cohérent.
\end{defi}

On définit la relation $\odot_n$ sur $B$ comme étant la plus petite relation transitive et réflexive telle que pour tout couple d'élément de la base $a,b$ de dimension supérieure ou égale à $n$:

$$a\odot_n b \mbox{ si }\langle a\rangle_n^-\wedge\langle b\rangle_n^+ \neq 0$$

\begin{defi}
Une base est dite \textit{sans boucles} lorsque  pour tout $n$,  la relation $\odot_n$ est un ordre (partiel) sur $B$.
\end{defi}

On définit maintenant la sous-catégorie $\CDAB$ de $\CDA$ composée de complexes dirigés augmentés qui admettent une base unitaire et sans boucles. 
On va maintenant décrire l'analogue de la notion de base pour les $\omega$-catégories. 

\begin{defi}
Une $\omega$-catégorie  $C$ est \textit{générée par composition} par un ensemble $E\subset C$ lorsque toute  cellule peut s’écrire comme une composition d'éléments de $E$ et d'unités itérées d'éléments de $E$.  Cet ensemble est \textit{une base} si $\{[e]_{d(e)}\}_{e\in E}$ est une base du complexe dirigé augmenté  $\lambda C$. 
\end{defi}

On donne maintenant les analogues des notions de base sans boucles et unitaire.
\begin{defi}
Une base $E$ d'une $\omega$-catégorie est : 
\begin{enumerate}
\item \textit{Sans boucles} lorsque  $\{[e]_{d(e)}\}_{e\in E}$ l'est.
\item \textit{Atomique} lorsque $[d_n^+ e]_n \wedge [d_n^- e]_n  = 0$ pour tout $e\in E$ et tout entier $n$ strictement inférieur à la dimension de $e$.  
\end{enumerate}
\end{defi}

\begin{prop}
 Si une base sans boucles $E$ est atomique alors $\{[e]\}_{e\in E}$ est unitaire.
 \end{prop}
\begin{proof}
 Voir \cite[proposition 4.6]{steiner}
 \end{proof}

On définit alors la  catégorie \infcatB comme étant la sous-catégorie pleine de  \infcat composée des $\omega$-catégories admettant une base atomique et sans boucles.

 \begin{theo}
 \label{theo}
 Une fois restreinte, l'adjonction 
$$\lambda : \xymatrix{
    \infcatB  \ar@/^/[r] \ar@{}[r]|-{\bot}  &
    \CDAB \ar@/^/[l] }: \nu.$$
devient une équivalence adjointe, c'est-à-dire que :
$$ \lambda \circ \nu \cong  id_{\CDAB}~~~~~~~ id_{\infcatB }\cong \nu \circ \lambda$$
\end{theo}
\begin{proof}
Voir \cite[théorème 5.11]{steiner}.
\end{proof}

Si $K$ est un complexe dirigé augmenté admettant une base  unitaire et sans boucles $B$, alors la $\omega$-catégorie $\nu K$ admet une base atomique et sans boucles donnée par l’ensemble $\langle B\rangle := \{\langle b\rangle,b\in B\}$. Réciproquement si une $\omega$-catégorie $C$ admet une base atomique et sans boucles $E$, alors le complexe dirigé augmenté $\lambda C$ admet une base unitaire et sans boucles donné par la famille d'ensembles $[E_n] := \{[e]_{d(e)}, e\in E_n\}$. 
\note{Les isomorphismes
$$\lambda \nu K\cong K \mbox{~~~ et ~~~} C\cong \nu\lambda C$$
induisent des isomorphismes:
$$[\langle B\rangle ]\cong B  \mbox{~~~ et ~~~} E \cong \langle [E]\rangle.$$
Les bases, si elles existent, sont uniques pour les complexes dirigés augmentés, elles le sont donc
aussi pour les $\omega$-catégories admettant une base atomique et sans boucles.}
\section{Chaînes}
\subsection{Définition et propriétés des chaînes }
On fixe un complexe dirigé augmenté $K$  admettant une base sans boucles et unitaire $B$.

\begin{defi}
Une \textit{chaîne} est une somme  $\sum_{b\in B} \lambda_b b$
telle que la famille $\{\lambda_b\}_{b\in B}$ soit composée d'entiers positifs, nuls sauf un nombre fini.  

Pour une chaîne $a$, on note $supp(a)\subset B$ son support. Le degré  d'une chaîne non nulle $a$ est le maximum de l'ensemble $\{|b|, b\in supp(a)\}$ et est noté  $|a|$. On étend le degré à toutes les chaînes en posant $|0| = -1$.
\end{defi}

\begin{defi}
\label{defi:definition de la relation d'ordre sur les chaines}
Pour une chaîne quelconque $a := \sum_{b\in B} \lambda_b b$, et un élément $b\in B$, l'entier $\lambda_b$ sera appelé \textit{l'indice de $b$ dans $a$}, et noté $\lambda_b^a$. Un élément est dans le support de $a$ quand son indice est non nul.
Pour deux chaînes $a,a'$, un élément $b\in B$, et $\mu\in\mathbb{N}$ on a : 
$$
\begin{array}{rcl}
\lambda_b^{\mu a} &=& \mu\lambda_b^a\\
\lambda_b^{a + a'} &=& \lambda_b^a + \lambda_b^{a'}\\
\lambda_b^{a \wedge a'} &=& \min(\lambda_b^a, \lambda_b^{a'})\\
\lambda_b^{a \vee a'} &=& \max(\lambda_b^a, \lambda_b^{a'})\\
\lambda_b^{a \diagdown a'} &=& |\lambda_b^a - \lambda_b^{a'}|_+\\
\end{array}$$
où pour un entier relatif $k$, $|k|_+:=\max(k,0)$.

On peut alors définir une relation d'ordre sur les chaînes:
$$a~\leq~a' \mbox{~~~ssi~~~} \forall b\in B,~ \lambda_b^a \leq \lambda_b^{a'}$$
On définit aussi $a\leq 1$ lorsque $\forall b\in B, ~\lambda_b^a \leq 1$. Cette relation d'ordre vérifie alors les propriétés suivantes: pour toutes chaînes $a,a',a''$,
$$
\def\arraystretch{1.4}
\begin{array}{ccl}
a'\leq a \mbox{ ~et~ } a''\leq a &\Rightarrow& a'\wedge a''\leq a'\vee a'' \leq a\\
a\leq a' \mbox{ ~et~ } a\leq a'' &\Rightarrow& a\leq a'\wedge a''\leq a'\vee a'' \leq a'+a''\\
a'\leq a \mbox{ ~et~ } a''\leq a \mbox{~et~ } supp(a')\cap supp(a'')=\emptyset  &\Rightarrow& a'+ a'' \leq a\\
\end{array}
$$
\end{defi}

\begin{defi}
Soient $\sum_{b\in B} \lambda_b b$ une chaîne et $k$ un entier. Le \textit{$k$-reste} de $a$, noté $r_k(a)$, est défini par 
$$r_k(a) :=\sum_{b\in B}\xi^k_b \lambda_b b,$$
où $\xi_b^k = 1$ lorsque $|b|\leq k$ et nul sinon. 

La partie \textit{$k$-homogène} de $a$, notée $(a)_k$, est définie par
$$(a)_k :=\sum_{b\in B}\sigma^k_b \lambda_b b,$$
où $\sigma^k_b = 1$ lorsque $|b|= k$ et nul sinon. 

Une chaîne $a$ vérifiant   $a =(a)_{|a|}$ est dite \textit{homogène}. 
\end{defi}

Tout élément de $K^*$ correspond à une chaîne homogène, réciproquement, toute chaîne homogène correspond à un élément de $K^*$.

\begin{defi}
\label{defi:des sources et but}
On définit alors les applications \textit{sources} et \textit{buts} par co-induction. Pour un entier $n$ et $\alpha\in\{-,+\}$ : 
$$
d_n^\alpha(a) := \left\{
\begin{array}{ll}
 a & \mbox{ si $|a|\leq n $ } \\ 
  \partial_n^\alpha ( (d^\alpha_{n+1} a)_{n+1}~ ) + r_n (d^\alpha_{n+1} a) & \mbox{ sinon}\\
 \end{array} \right.$$
Pour une chaîne quelconque $a$, un entier $n$ et $\alpha\in\{-,+\}$, le degré de $d_n^\alpha(a)$ est inférieur ou égal à $n$ et on a donc par construction $$ d^\alpha_{n}(a) = d^\alpha_{n}(d^\alpha_{n+1}(a)).$$ 
\end{defi}

\begin{rem}
\label{rem:remarque sur le reste}
Pour une chaîne $a$, $n$ un entier et $\alpha\in\{-,+\}$, on a l'égalité:
$$r_{n}(a) = r_n(d_{n+1}^\alpha (a)).$$
\end{rem}

\begin{rem}
\label{remarquer dans le cas coherent}
Si $a$ est une chaîne homogène, $d_n^\alpha(a)$ est aussi une chaîne homogène et on a alors $d_{n}^\alpha a = \partial_n^\alpha ( (d^\alpha_{n+1} a))$ pour tout $n<|a|$. Il est alors intéressant de remarquer la ressemblance avec la définition \ref{defi:tableau_de_steiner_associe_a_un singloton}. On a en effet pour tout $k\leq n$ et $\alpha\in\{-,+\}$,
$$ d^\alpha_k(a)=\langle a\rangle^\alpha_k.$$
\end{rem}

\begin{defi}
Deux chaînes $a$ et $b$ sont \textit{$n$-parallèles} lorsque
$$ d_n^-(a)=d_n^-(b) \mbox{  et  } d_n^+(a)=d_n^+(b).$$
\end{defi}

\begin{lem}
Soit $a$ une chaîne de degré $n+1$. On a alors 
$$d^+_{n-1}\circ d^+_{n}(a) = d^+_{n-1} \circ d^-_{n}(a)\mbox{ ~~~~ et~~~~  } d^-_{n-1} \circ d^+_n(a)= d^-_{n-1} \circ d^-_n(a).$$ 
\end{lem}
\begin{proof}
En appliquant la définition des applications sources et buts, on obtient:
$$\def\arraystretch{1.4}\begin{array}{lcl}
d^+_{n-1}\circ d^+_{n}(a) &=& d^+_{n-1} \big(\partial_{n}^+((a)_{n+1}) + r_n(a)\big)\\
&=& \partial_{n-1}^+\big(\partial_{n}^+((a)_{n+1}) + (a)_n\big) + r_{n-1}\big(\partial_{n}^+((a)_{n+1}) + r_n(a)\big) \\
 &=& \partial_{n-1}^+\big(\partial_{n}^+((a)_{n+1}) + (a)_n\big) + r_{n-1}(a) \\
\mbox{de même:} \\
 d^+_{n-1}\circ d^-_{n}(a) &=&\partial_{n-1}^+\big(\partial_{n}^-((a)_{n+1}) + (a)_n\big) + r_{n-1}(a) \\
\end{array}
$$
Or 
$$
\def\arraystretch{1.4}\begin{array}{lcl}
\partial_{n-1}\big(\partial_{n}^+((a)_{n+1}) + (a)_n\big)  - \partial_{n-1}\big(\partial_{n}^-((a)_{n+1}) + (a)_n\big)  &=& \partial_{n-1}\big(\partial_{n}^+((a)_{n+1}) - \partial_{n}^-((a)_{n+1}) \big)\\
&=& \partial_{n-1} \partial_{n}((a)_{n+1})  = 0
\end{array}
$$
d'où
$$
\def\arraystretch{1.4}\begin{array}{ll}
&\partial_{n-1}^+\big(\partial_{n}^+((a)_{n+1}) + (a)_n\big)  = \partial_{n-1}^+\big(\partial_{n}^-((a)_{n+1}) + (a)_n\big)\\
\Leftrightarrow & d^+_{n-1} \circ d^+_n(a)= d^+_{n-1} \circ d^-_n(a)
 \end{array}$$

On montre de façon analogue l'égalité $ d^-_{n-1} \circ d^+_n(a)= d^-_{n-1} \circ d^-_n(a)$.
\end{proof}

\begin{prop}
\label{prop:d_n est une differencielle}
Pour tout entier $n>0$, 
$$d^+_{n-1}\circ d^+_{n} = d^+_{n-1} \circ d^-_{n}\mbox{ ~~~~ et~~~~  } d^-_{n-1} \circ d^+_n = d^-_{n-1} \circ d^-_n.$$
\end{prop}
\begin{proof}
Soit $a$ une chaîne quelconque. Si $n$ est supérieur ou égal au degré de $a$ alors $d^-_{n}(a) =d^+_{n}(a) = a$, et l'égalité est donc trivialement vraie.
On va montrer le résultat dans le cas général pour toute chaîne $a$ et un entier $n$, par récurrence sur $k:=|a|-n$.

Le lemme précédent  correspond au cas $|a|-n = 1 \Leftrightarrow |a| = n+1$, et est donc l’initialisation de notre récurrence.

Supposons maintenant le résultat vrai  au rang $k$, et donnons nous une chaîne $a$ et un entier $n$ vérifiant $|a|-n = k+1$. On a donc $|a|-(n+1) =k$ et l'hypothèse de récurrence nous permet d’affirmer que 
$ d^-_{n}(d^+_{n+1}(a))= d^-_{n}(d^-_{n+1}(a))$.
La définition des applications sources et buts implique quant à elle que pour $\alpha\in\{-,+\}$, $ d^\alpha_{n}(a) = d^\alpha_{n}(d^\alpha_{n+1}(a))$. 
Comme $ |d^+_{n+1}(a)| = n+1$,  le lemme précédent indique que $$d^+_{n-1}\circ d^+_{n}(d^+_{n+1}(a)) = d^+_{n-1}\circ d^-_{n}(d^+_{n+1}(a)).$$ 
En mettant tout ensemble, on obtient le résultat voulu:
$$d^+_{n-1}\circ d^+_{n}(a) = d^+_{n-1}\circ d^+_{n}(d^+_{n+1}(a)) =
 d^+_{n-1}\circ d^-_{n}(d^+_{n+1}(a)) =
  d^+_{n-1}\circ d^-_{n}(d^-_{n+1}(a))=
   d^+_{n-1}\circ d^-_{n}(a)$$
L'autre égalité se montre de façon analogue.
\end{proof}

On va se servir de la proposition précédente pour définir une augmentation.

\begin{prop}
Soit $a$ une chaîne quelconque, alors $d_0^+(a)$ et $d_0^-(a)$ sont des chaînes  homogènes de degré zero, et 
$$e(d_0^+ a) =  e(d_0^- a).$$
\end{prop}
\begin{proof}
La proposition \ref{prop:d_n est une differencielle} indique  que pour tout $\alpha\in\{-,+\}$, $d_0^\alpha(a) = d_0^\alpha(d_1^+ a)$. Quitte à remplacer $a$ par  $d_1^+ a$, on peut supposer que $a$ est de degré $1$. On définit alors $r := r_0(a)$ et $a' :=a - r$. 
On a alors 
$$\def\arraystretch{1.4}
\begin{array}{rcl}
e(d_0^+ a)  &= &e (\partial_0^+(a') + r) \\
&=& e (\partial_0^+(a')) + e(r)\\
& =& e (\partial_0^-(a')) + e(r) \\
&=&e(d_0^- a).
\end{array}$$
\end{proof}
\begin{defi}
Pour une chaîne quelconque $a$, on définit $e(a) := e(d_0^+ a) =  e(d_0^- a).$
\end{defi}

\begin{lem}
\label{lem:d_n ne voit pas les truc en dessous de n}
Soient $a,c$ deux chaînes, $n$ un entier tel que $|a|\geq n >|c|$ et $\alpha\in\{-,+\}$. Alors $d_n^\alpha(a+c)=d_n^{\alpha}(a)+c$.
\end{lem}
\begin{proof}
On va procéder par une récurrence descendante sur $|c|<n\leq |a|$. Si $n=|a|$, alors $d^{\alpha}_n(a+c)= a+c = d^{\alpha}_n(a)+c$. Supposons maintenant le résultat vrai pour $n>|c|+1$ . On a alors 
$$\def\arraystretch{1.4}
\begin{array}{rcl}
d^\alpha_{n-1}(a+c) &= &\partial^\alpha_{n-1}\big( (d_n^\alpha (a+c))_{n}\big) + r_n(d^\alpha_n (a+c) )\\
&= &\partial^\alpha_{n-1}\big( (d_n^\alpha a +c )_{n}\big) + r_n(d^\alpha_n a + c )\\
&= &\partial^\alpha_{n-1}\big( (d_n^\alpha a )_{n}\big) + r_n(d^\alpha_n a) +  c \\
&= &d^\alpha_{n-1}(a) + c.
\end{array}$$
Ainsi, le résultat est vrai pour $(n-1)$, ce qui conclut la preuve.
\end{proof}

\begin{prop}
Soient $a,c$ deux chaînes, $n := min(|a|,|c|)$ et $\alpha\in\{-,+\}$. On suppose de plus que $r_{n-1}(a+c)= 0$. On a alors
$$d^\alpha_{n-1}(a+c) = (d^\alpha_{n-1}(a) \diagdown d^{-\alpha}_{n-1} (c)) + (d^\alpha_{n-1}(c) \diagdown d^{-\alpha}_{n-1} (a)).$$
\label{prop:presque lineraite des sources et but}
\end{prop}
\begin{proof}
\note{On prouve cette formule pour $\alpha:=+$, l'autre cas étant similaire.}
Supposons tout d'abord que $|a|=|c|$.  
 Les deux chaînes sont alors homogènes.
On a alors :
$$
\def\arraystretch{1.4}
\begin{array}{rcl}
d^+_{n-1}(a+c) &=& \partial^+_{n-1}(a+c)\\
&=&(\partial_{n-1}(a+c))_{+}\\
&=&(\partial_{n-1}^+(a) - \partial_{n-1}^-(a) + \partial_{n-1}^+(c) -\partial_{n-1}^-(c))_+\\
\end{array}$$
Comme les supports de $ \partial_{n-1}^+(a)$ et $ \partial_{n-1}^{-}(a)$ (resp. $ \partial_{n-1}^+(c)$ et $ \partial_{n-1}^{-}(c)$) sont disjoints par construction, on a bien
$$
\def\arraystretch{1.4}
\begin{array}{rcl}
d^\alpha_{n-1}(a+c) &=&(\partial_{n-1}^+(a) - \partial_{n-1}^{-}(c))_+ +(  \partial_{n-1}^+(c) -\partial_{n-1}^{-}(a))_+\\
&=&(\partial^+_{n-1}(a) \diagdown \partial^{-}_{n-1} (c)) + (\partial^+_{n-1}(c) \diagdown \partial^{-}_{n-1} (a)),\\
\end{array}
$$
\note{et comme $d^+_{n-1}(a) =\partial^+_{n-1}(a)$ et $d^+_{n-1}(c) =\partial^+_{n-1}(c)$, cela démontre la formule.}

Supposons maintenant que $|a|>|c|$. \note{On a donc par définition: }
$$
\def\arraystretch{1.4}
\begin{array}{rcl}
d^+_{n-1}(a+c) &=& \partial^+_{n-1}(d^+_n(a+c)) \\
&=&\partial^+_{n-1}(d^+_n(a)+c) \mbox{ (\ref{lem:d_n ne voit pas les truc en dessous de n})}\\
&=&d_{n-1}^+(d^+_n(a)+c)
\end{array}$$
Or $|d_n^+(a)|=|c|=n$, on s'est donc ramené au cas précédent, et on obtient : 
$$
\def\arraystretch{1.4}
\begin{array}{rcl}
d^\alpha_{n-1}(a+c) &=&(\partial_{n-1}^+(d^+_n(a)) - \partial_{n-1}^{-}(c))_+ +(  \partial_{n-1}^+(c) -\partial_{n-1}^{-}(d^+_n(a)))_+\\
&=&(d_{n-1}^+(a) - d_{n-1}^{-}(c))_+ +(  d_{n-1}^+(c) -d_{n-1}^{-}(a))_+\\
 &=&(d^+_{n-1}(a) \diagdown d^{-}_{n-1} (c)) +
(d^+_{n-1}(c) \diagdown d^{-}_{n-1} (a)).
\end{array}$$
\end{proof}

En particulier, pour deux chaînes $a$ et $c$ vérifiant les hypothèses de la proposition, $d^\alpha_{n-1}(a+c)$ ne dépend que de $d^+_{n-1}(a),d^-_{n-1}(a)$ et $d^+_{n-1}(c),d^-_{n-1}(c)$. 

\begin{lem}
\label{lem:le reste est le wedge de d- et d+}
Soient $a$ une chaîne, $n$ un entier inférieur ou égale au degré de $a$. On a alors $r_{n-1}(a)=d^-_{n-1}(a)\wedge d^+_{n-1}(a)$.
\end{lem}
\begin{proof}
\note{L'égalité $d_{n-1}^+d_{n}^+(a)=d_{n-1}^+d_{n}^-(a)$, prouvée en \ref{prop:d_n est une differencielle}, implique que
$$\partial_{n-1}^+((d_{n}^+ a)_{n}) + r_{n-1}(d_{n}^+ a) = \partial_{n-1}^+((d_{n}^- a)_{n}) + r_{n-1}(d_{n}^- a)$$
et donc 
$$
\def\arraystretch{1.4}
\begin{array}{rcl}
d^-_{n-1}(a)\wedge d^+_{n-1}(a) &=& \big(\partial^-_{n-1}((d^-_n a)_{n}) + r_{n-1}(d^-_na)\big)\wedge \big(\partial^+_{n-1}((d^+_n a)_{n}) + r_{n-1}(d^+_na)\big)\\
&=& \big(\partial^-_{n-1}((d^-_n a)_{n}) + r_{n-1}(d^-_na)\big)\wedge \big(\partial_{n-1}^+((d_{n}^- a)_{n}) + r_{n-1}(d_{n}^- a)\big).\\
\end{array}$$
Or $\partial^-_{n-1}((d^-_n a)_{n}) \wedge \partial^+_{n-1}((d^-_n a)_{n})=0$, et on en déduit donc:
$$d^-_{n-1}(a)\wedge d^+_{n-1}(a) = r_{n-1}(d^-_na).$$
Enfin, on a remarqué en \ref{rem:remarque sur le reste} que  $r_{n-1}(d_n^- a)=r_{n-1}(a)$, et cela conclut donc la preuve.}\end{proof}

\begin{cor}
\label{cor:lembis}
Soient $a$ et $a'$  deux chaînes de même degré, et $c$ une chaîne quelconque. Soient $\alpha\in\{-,+\}$ et $n:= min(|a|,|c|)$. On suppose que pour tout $\beta\in\{-,+\}$,  $d^\beta_{n-1}(a)=d^\beta_{n-1}(a')$. On a alors :
$$d^\alpha_{n-1}(a+c) =d^\alpha_{n-1}(a'+c).$$
\end{cor}

\begin{proof}
Donnons nous trois chaînes $a,a',c$ vérifiant les conditions du corollaire. Supposons tout d'abord que $r_{n-1}(a+c)=r_{n-1}(a'+c) = 0$.
On a donc:
$$
\def\arraystretch{1.4}
\begin{array}{rcl}
d^\alpha_{n-1}(a+c) 
&=&(d^\alpha_{n-1}(a ) \diagdown d^{-\alpha}_{n-1} (c )) +
(d^\alpha_{n-1}(c ) \diagdown d^{-\alpha}_{n-1} (a ))\\
&=&(d^\alpha_{n-1}(a' ) \diagdown d^{-\alpha}_{n-1} (c )) +
(d^\alpha_{n-1}(c ) \diagdown d^{-\alpha}_{n-1} (a' ))\\
&=&d^\alpha_{n-1}(a' +c ) \\
\end{array} 
$$
Plaçons nous maintenant dans le cas général. 
Le lemme \ref{lem:le reste est le wedge de d- et d+} implique que
$$r_{n-1}(a) = d^-_{n-1}(a)\wedge d^+_{n-1}(a) = d^-_{n-1}(a')\wedge d^+_{n-1}(a')=r_{n-1}(a').$$
En posant $$\begin{array}{lcl}
\tilde{a}&:=& a - r_{n-1}(a)\\
\tilde{a}'&:=& a' - r_{n-1}(a')\\
\tilde{c}&:=& c - r_{n-1}(c),\\
\end{array}$$
on a alors $d^\alpha_{n-1}(\tilde{a} +\tilde{c} )  = d^\alpha_{n-1}(\tilde{a}' +\tilde{c} ) $ et donc
$$
d^\alpha_{n-1}(a+c) = d^\alpha_{n-1}(\tilde{a}+\tilde{c}) + r_{n-1}(a) + r_{n-1}(c)
=d^\alpha_{n-1}(\tilde{a}'+\tilde{c}) + r_{n-1}(a') + r_{n-1}(c)
=d^\alpha_{n-1}(a'+c).
$$

\end{proof}

\begin{defi}
Pour une chaîne $a$, on définit le \textit{degré de composition} $|a|_c$  : 
$$
|a|_c = \mbox{sup}\{ n \in \mathbb{N}\cup\{-1\}| \exists b,b' \in supp(a) \mbox{ tels que } b \neq b',|b|\geq |b'| >n\}
$$
où on choisit la convention 
$\mbox{sup}(\emptyset) :=-1$.
\end{defi}

\begin{rem}
\label{rem:sur le degre de composition}
Si le support d'un chaîne est vide ou réduit à un seul élément, son degré de composition sera $-1$. Dans le cas contraire, pour obtenir le  degré de composition, on soustrait $1$ au degré de l'élément du support admettant le deuxième plus grand degré. \note{Par exemple, si $\lambda,\lambda'$ sont deux entiers non nuls, et $b,b'$ deux éléments de la base, le degré de composition de la chaîne $\lambda b + \lambda' b'$ est $((min(|b|,|b'|)-1)$.}

Pour une chaîne  $a$, il existe au plus un $b$ dans le support de $a$ tel que $|b|>|a|_c+1$. 
\end{rem}

\note{\begin{lem}
\label{lem:pseudolinearite}
Soient  $a:=\sum_{i\leq m}  \lambda_i b_i$  une chaîne dont le $|a|_c$-reste est nul, et $\alpha\in\{-,+\}$. On a l’inégalité suivante :
$$d^\alpha_{|a|_c}(\sum_{i\leq m}  \lambda_i b_i)\leq \sum_{i\leq m}  \lambda_i d^\alpha_{|a|_c}(b_i)$$
\end{lem}
\begin{proof}
Supposons tout d'abord que $a$ est une chaîne homogène. Cela implique en particulier que $|a|-1=|a|_c$. On peut alors montrer le résultat par récurrence sur $m$ en utilisant la proposition \ref{prop:presque lineraite des sources et but}.

Supposons maintenant que $a$ n'est pas une chaîne homogène. Il existe donc un unique élément de la base $b$ de degré strictement supérieur à $(|a|_c+1)$, un entier non nul $\lambda$, et une chaîne homogène $a'$  de degré $|a|_c+1$ tels que 
$$a := a' +\lambda b.$$ 
De plus, on a $|b|>|a_c|+1$. La proposition  \ref{prop:presque lineraite des sources et but} implique alors que 
$$d^\alpha_{|a|_c}(a) = d^\alpha_{|a|_c}(a'+\lambda b)\leq  d^\alpha_{|a|_c}(a') + \lambda d^\alpha_{|a|_c}(b).$$
Enfin, on remarque que $|a'|_c = |a|_c$, et en appliquant la formule à la chaîne homogène $a'$, on conclut la preuve.
\end{proof}
}

\note{\begin{rem}
Pour  $a:=\sum_{i\leq m}  \lambda_i b_i$  une chaîne dont le $|a|_c$-reste est nul, et $\alpha\in\{-,+\}$, on peut déduire du lemme \ref{lem:pseudolinearite} l'inclusion suivante des supports:
$$supp(d^\alpha_{|a|_c}(\sum_{i\leq m}  \lambda_i b_i))\subset \bigcup_{i\leq m} supp(d^\alpha_{|a|_c}(b_i)).$$
\end{rem}}

\note{
\begin{defi}
Soit $(E,\triangleleft)$ un ensemble muni d'une relation d'ordre partiel.
Soit $F$ un sous ensemble de $E$. Un élément $v\in F$ est un \textit{minimum partiel} de $F$ si pour tout $u\in F$, $u\triangleleft v$ est faux. Un tel élément n'est pas forcement unique.

On dit qu'une suite finie $\{u_n\}_{n< m}$ dans $E$ est \textit{ordonnée}, si pour tout $k<m$, $u_k$ est un minimum partiel du sous ensemble $\{u_n,k\leq n< m\}$. Remarquons que pour tout sous ensemble fini $F$ de $E$, il existe au moins une suite $u_\_:\{0,1..,|F|-1\}\to F$ bijective et ordonnée. Cette suite n'est cependant pas unique.
\end{defi}
}

\begin{defi}
\label{defi:forme ordonne}
Soit $a$ une chaîne. Alors il existe un entier $n$ et une \note{suite ordonnée $\{b_i\}_{i\leq n}$ pour l'ensemble $B$ muni de la relation d'ordre partiel $\odot_{|a|_c}$ et tel que }
$$a = \sum_{i\leq n} \lambda_{i} b_i + r_{|a|_c}(a)$$
où les $\lambda_i$ sont des entiers non nuls. On dira alors que la chaîne est \textit{écrite sous forme ordonnée}.

Pour $0\leq k\leq n+1$ on définit alors $a_{<k}:=\sum_{i<k} \lambda_{i} b_i$ et $a_{\geq k}:= \sum_{k\leq i \leq n} \lambda_{i} b_i$. Ces deux chaînes vérifient l'égalité suivante : 
$$a = a_{<k}+a_{\geq k} + r_{|a|_c}(a)$$
\end{defi}

\begin{rem}\note{
Soit $a:=\sum_{i\leq n} \lambda_{i} b_i + r_{|a|_c}(a)$ une chaîne écrite sous forme ordonnée. 
En déroulant les définitions, cela implique que pour tout $j<k\leq n$, il n'existe \textit{pas} de suite finie $\{d_i\}_{i\leq m}$, composée d'éléments de $B$, et telle que 
$$ b_k \odot_{|a|_c} d_0 \odot_{|a|_c} d_1\odot_{|a|_c}...\odot_{|a|_c} d_m\odot_{|a|_c} b_j.$$
}
\end{rem}

\begin{prop}
Soit $a := \sum_{i\leq m} \lambda_{i} b_i +  r_{|a|_c}(a)$ une chaîne sous forme ordonnée de degré de composition supérieur ou égal à 0. 
On a alors
$$e(a) = e\big(d^-_{|a|_c}(a_{<k}) \vee d_{|a|_c}^+(a_{\geq k}) +  r_{|a|_c}(a)\big),$$
et pour tout $\alpha\in\{-,+\}$, pour tout $0\leq k\leq m+1$ et  pour tout $n<|a|_c$
$$d^\alpha_{n}a = d^\alpha_{n}\big(d^-_{|a|_c}(a_{<k}) \vee d_{|a|_c}^+(a_{\geq k}) +  r_{|a|_c}(a)\big).$$
\label{prop:nom provisoire}
\end{prop}
Pour démontrer ce résultat, on a besoin d'un lemme : 

\begin{lem}

\label{lem:source dans le cas ordone}
Soient $a:=\sum_{i\leq m} \lambda_{i} b_i$ une chaîne sous forme ordonnée de degré de composition supérieur ou égal à $0$  telle que $ r_{|a|_c}(a) = 0$, et  $k\leq m+1$ un entier. On a  alors
$$d^+_{|a|_c} a =  d^+_{|a|_c}(a_{<k}) + ( d_{|a|_c}^+(a_{\geq k})\diagdown d^-_{|a|_c}(a_{<k})).$$
Ce résultat peut aussi s'exprimer de la façon suivante :  pour tout élément $b$ de la base,
$$\lambda_b^{d^+_{|a|_c} a}=  \lambda_b^{d^+_{|a|_c}(a_{<k})}+ |\lambda_b^{ d_{|a|_c}^+(a_{\geq k})}- \lambda_b^{d^-_{|a|_c}(a_{<k})}|_+.$$
\end{lem}
\begin{proof}
Pour $k$ égal à $0$ où $m+1$, l'égalité que l'on cherche à montrer est immédiatement vérifiée car une des chaînes $a_{<k}$ et $a_{\geq k}$ est égale à $a$ et l'autre est nulle.
Supposons donc $0<k\leq m$. Cela implique notamment qu'aucune des chaînes $a_{<k}$ et $a_{\geq k}$ n'est nulle et que $a = a_{<k}+ a_{\geq k}$ n'est pas réduite à un seul élément.

Selon la remarque \ref{rem:sur le degre de composition}, on sait qu'il existe au plus un $i$ tel que $|b_i|>|a|_c+1$, et que tous les autres $b_i$ sont de degré $|a|_c+1 $. On a donc une des chaînes $a_{< k}$ et $a_{\geq k}$ qui est de degré $(|a|_c+1)$ et l'autre de degré supérieur ou égal. Cela peut se résumer dans l'égalité suivante : $|a|_c =  min(|a_{< k}|,|a_{\geq k}|) - 1$.
On peut donc appliquer la proposition \ref{prop:presque lineraite des sources et but} et on obtient 
$$
d^+_{|a|_c} a =  (d^+_{|a|_c}(a_{<k})\diagdown d_{|a|_c}^-(a_{\geq k}) )+  (d_{|a|_c}^+(a_{\geq k})\diagdown d^-_{|a|_c}(a_{<k}))$$

Intéressons nous maintenant à la chaîne $d^+_{|a|_c}(a_{<k})\diagdown d_{|a|_c}^-(a_{\geq k})$. Par hypothèse, la chaîne $a:=\sum_{i\leq m} \lambda_{i} b_i$ est écrite sous forme ordonnée, et donc pour tout $i<j$, $d^+_{|a|_c}(b_i)\wedge d^-_{|a|_c}(b_j) = 0 $. En appliquant l'inégalité du lemme \ref{lem:pseudolinearite}, on obtient :
$$
\def\arraystretch{1.4}
\begin{array}{rcl}
d^+_{|a|_c}(a_{<k})\wedge d_{|a|_c}^-(a_{\geq k} )
&\leq&
 \sum_{i<k}d^+_{|a|_c}(\lambda_i b_i)  ~\wedge ~\sum_{i\geq k}d^-_{|a|_c}(\lambda_i b_i)\\
 &\leq &\sum_{i<k\leq j}~d^+_{|a|_c}(\lambda_i b_i)\wedge d^-_{|a|_c}(\lambda_j b_j)\\
 &\leq & 0.
\end{array}$$
Cela implique que $d^+_{|a|_c}(a_{<k})\diagdown d_{|a|_c}^-(a_{\geq k}) = d^+_{|a|_c}(a_{<k})$. On obtient donc le résultat voulu :

$$
d^+_{|a|_c} a =  d^+_{|a|_c}(a_{<k})+  (d_{|a|_c}^+(a_{\geq k})\diagdown d^-_{|a|_c}(a_{<k}))\\
$$

\end{proof}

\begin{proof}[Preuve de la proposition \ref{prop:nom provisoire}]
Donnons nous une chaîne $a:=\sum_{i\leq m}\lambda_i b_i + r_{|a|_c}(a)$, écrite sous forme ordonnée. Supposons tout d'abord que $|a|_c\geq 1$.
On commence par démontrer le résultat pour $n=|a|_c-1$. Remarquons tout d'abord que 
$$d^-_{|a|_c}(a_{< k}) \vee d_{|a|_c}^+(a_{\geq  k})= 
d^-_{|a|_c}(a_{ < k}) +  (d_{|a|_c}^+(a_{\geq  k})-d^-_{|a|_c}(a_{< k}))_+ .
$$

De plus, on sait que les chaînes $d^-_{|a|_c}(a_{< k})$ et $d^+_{|a|_c}(a_{< k})$  sont $(|a|_c-1)$- parallèles. En appliquant consécutivement  le lemme \ref{lem:source dans le cas ordone} et  le corollaire \ref{cor:lembis}, on obtient:
$$
\def\arraystretch{1.4}
\begin{array}{ccl}
d^\alpha_{|a|_c-1} \big( d^-_{|a|_c}(a_{< k}) \vee d_{|a|_c}^+(a_{\geq  k}) \big) &=&d^\alpha_{|a|_c-1} \big( d^-_{|a|_c}(a_{ < k}) +  (d_{|a|_c}^+(a_{\geq  k})-d^-_{|a|_c}(a_{< k}))_+ \big)\\
&=& d^\alpha_{|a|_c-1} \big(d^+_{|a|_c}(a_{ < k}) +  (d_{|a|_c}^+(a_{\geq  k})-d^-_{|a|_c}(a_{< k}))_+  \big)\\
&=& d^\alpha_{|a|_c-1}\big(d^+_{|a|_c}((a)_{|a|_c})  \big) = d^\alpha_{[a|_c-1}((a)_{|a|_c}) 
\end{array}
$$

En utilisant une dernière fois le corollaire \ref{cor:lembis} on obtient bien que
$$d^\alpha_{|a|_c-1} \big( d^-_{|a|_c}(a_{< k}) \vee d_{|a|_c}^+(a_{\geq  k})  + r_{|a|_c}(a)\big)  = 
d^\alpha_{[a|_c-1}((a)_{|a|_c} + r_{|a|_c}(a))$$
On a donc montré que les chaînes $a$ et $( d^-_{|a|_c}(a_{< k}) \vee d_{|a|_c}^+(a_{\geq  k})  + r_{|a|_c}(a))$ étaient $(|a|_c-1)$-parallèles. Cela implique donc que pour tout $m<|a|_c$, elles sont $m$-parallèles.

Plaçons nous maintenant dans le cas où $|a|_c=0$. En utilisant la linéarité de $e$, l'égalité $e\circ d_0^+ =e\circ d_0^-$ et le lemme \ref{lem:source dans le cas ordone}, on obtient:
$$
\def\arraystretch{1.4}
\begin{array}{ccl}
e \big( d^-_{0}(a_{< k}) \vee d_{0}^+(a_{\geq  k}) \big) &=& e \big( d^-_{0}(a_{ < k}) +  (d_{0}^+(a_{\geq  k})-d^-_{0}(a_{< k}))_+ \big)\\
&=& e \big( d^-_{0}(a_{ < k})\big)  +  e\big((d_{0}^+(a_{\geq  k})-d^-_{0}(a_{< k}))_+ \big)\\
&=& e \big( d^+_{0}(a_{ < k})\big)  +  e\big((d_{0}^+(a_{\geq  k})-d^-_{0}(a_{< k}))_+ \big)\\
&=& e \big(d^+_{0}((a)_{|a|_c})  \big)
\end{array}
$$
Et donc 
$$
\def\arraystretch{1.4}
\begin{array}{rcl}
e \big( d^-_{0}(a_{< k}) \vee d_{0}^+(a_{\geq  k})  + r_{0}(a)\big)  &=&e \big( d^-_{0}(a_{< k}) \vee d_{0}^+(a_{\geq  k})\big)   + e(r_{0}(a))\\
&=&e \big(d^+_{0}((a)_{|a|_c})  \big)  + e(r_{0}(a))\\
&=& e(d^+_{0}(a)
\end{array}
$$
On a donc obtenu le résultat voulu.
\end{proof}

\begin{defi}
Une chaîne $a$ est \textit{cohérente} lorsque $e(a) = 1$.
\end{defi} 

Il est immédiat que pour une chaîne cohérente $a$ et un entier quelconque $n$, les chaînes $d_n^+(a)$ et $d_n^-(a)$ sont aussi cohérentes. De plus si deux chaînes sont $n$-parallèles pour un entier $n$ quelconque, l'une est cohérente si et seulement si l'autre l'est. Enfin la proposition précédente implique qu'une chaîne $a := \sum_{i\leq m} \lambda_{i} b_i +  r_{|a|_c}(a)$ écrite sous forme ordonnée  et de degré de composition supérieur à zéro est cohérente si et seulement si $d^-_{|a|_c}(a_{<k}) \vee d_{|a|_c}^+(a_{\geq k}) +  r_{|a|_c}(a)$ l'est.

\begin{prop}
Soit $b\in B$. Le singleton $b$ est une chaîne cohérente. 
\end{prop}
\begin{proof}
Cela découle de la remarque \ref{remarquer dans le cas coherent}.
\end{proof}

\begin{prop}
\label{prop:coherende de degre moin un est reduit a un singloton}
Une chaîne cohérente non nulle de degré de composition $-1$ est réduite à un élément.
\end{prop}
\begin{proof}
Soit $a := \sum_{0\leq i\leq n} b_i$ une chaîne cohérente de degré de composition $-1$. On sait qu'il existe au plus un élément de dimension supérieure à $0$ et, quitte à changer l'indexation, on peut donc supposer que pour tout $i>0$, $b_i$ est de dimension $0$.

On a alors $d_0^+(a) = d_0^+(b_0) + \sum_{1< i\leq n}b_i$ et donc $e(d_0^+(a)) = n$. Par hypothèse $a$ est cohérente, et donc $n=1$.
\end{proof}

\begin{prop}
Soit $a $ une chaîne cohérente. Alors $a\leq 1$ (au sens de la définition  \ref{defi:definition de la relation d'ordre sur les chaines}).
\label{prop:pas de double}
\end{prop}

\begin{proof}[Démonstration de la proposition \ref{prop:pas de double}]
On va montrer le résultat par récurrence sur le degré de composition.
Pour l’initialisation, plaçons nous dans le cas où $|a|_c=-1$.  Selon la proposition \ref{prop:coherende de degre moin un est reduit a un singloton}, la chaîne $a$ est réduite à un singleton et on a donc $a\leq 1$.

Supposons maintenant le résultat vrai pour les chaînes de degré de composition $m$, et montrons le résultat pour celles de degré de composition $m+1$. On se donne donc une chaîne cohérente $a:=\sum_{i\leq n} \lambda_i b_i$ écrite sous forme ordonnée et vérifiant $|a|_c = m+1$, et un entier $k\leq n$ quelconque. On veut montrer que  $\lambda_k=1$. Selon la proposition \ref{prop:nom provisoire} la chaîne 
$$c:=d^-_{|a|_c}(a_{<k}) \vee d_{|a|_c}^+(a_{\geq k}) +  r_{|a|_c}(a)$$
est cohérente car $a$ l'est. Comme son degré de composition est égal à $m$, on peut appliquer l'hypothèse de récurrence qui implique que $c\leq 1$.

Cependant, $c\geq d_{|a|_c}^+(a_{\geq k})$ et selon le lemme \ref{lem:source dans le cas ordone},
$$d_{|a|_c}^+(a_{\geq k})=  d_{|a|_c}^+(\lambda_k b_k) +(d_{|a|_c}^+(a_{\geq k+1}) \diagdown d_{|a|_c}^-(\lambda_k b_k) )\geq d_{|a|_c}^+(\lambda_k b_k)= \lambda_k d_{|a|_c}^+( b_k). $$
On obtient donc $\lambda_k d_{|a|_c}^+( b_k)\leq c\leq 1$ et donc $\lambda_k= 1$.
\end{proof}

\begin{rem}
La proposition précédente montre que pour une chaîne cohérente $a$, son  écriture ordonnée est de la forme 
$a= \sum_{i\leq n} b_i + r_{|a|_c}(a)$. Deux chaînes cohérentes sont donc égales si et seulement si elles ont le même support.
\end{rem}

\begin{prop}
Soit $a := \sum_{i\leq n} b_i + r_{|a|_c}(a)$ une chaîne cohérente écrite sous forme ordonnée de degré de composition supérieur ou égal à 0. Alors 
\begin{enumerate}
\item Pour tout $(k<l)$ et pour tout $\alpha\in\{-,+\}$, $d_{|a|_c}^\alpha b_k\wedge d_{|a|_c}^\alpha b_l= 0$;
\item Pour tout $k$ et pour tout $\alpha\in\{-,+\}$,
 $d_{|a|_c}^\alpha b_k \wedge  r_{|a|_c}(a)  =0$. 
\end{enumerate}
 \label{prop: plus de fourche }
\end{prop}
\begin{proof}

Donnons nous une chaîne $a:=\sum_{i\leq m}\lambda_i b_i + r_{|a|_c}(a)$, écrite sous forme ordonnée, et un entier $k\leq n$. 
Selon la proposition \ref{prop:nom provisoire} la chaîne 
$$c:=d^-_{|a|_c}(a_{<k}) \vee d_{|a|_c}^+(a_{\geq k}) +  r_{|a|_c}(a)$$
est cohérente car $a$ l'est et donc la proposition \ref{prop:pas de double} indique que $c\leq 1$. D'où l'inégalité,
\begin{equation}
1\geq c\geq d_{|a|_c}^+(a_{\geq k})+  r_{|a|_c}(a).
\label{equation}
\end{equation}
Donnons nous maintenant un élément $b'$ dans le support de $d_{|a|_c}^+(b_k)$. Par hypothèse, pour tout  $k<l$, $d_{|a|_c}^-(b_l)\wedge d_{|a|_c}^+(b_k)= 0$. De plus, on a forcément  $d_{|a|_c}^-(b_k)\wedge d_{|a|_c}^+(b_k)= 0$. Pour tout $k\leq l$, on a donc  $\lambda_{b'}^{d_{|a|_c}^-(b_l)} = 0$. 

Cependant, le lemme \ref{lem:source dans le cas ordone}, implique que pour tout élément $b$ de la base, 
$$\lambda_{b}^{d_{|a|_c}^+(a_{\geq l})}  = \lambda_{b}^{d_{|a|_c}^+(b_l)} + |\lambda_{b}^{d_{|a|_c}^+(a_{\geq l+1})}- \lambda_{b}^{d_{|a|_c}^-(b_l)}|_+ .
$$
Appliquée à l'élément $b'\in supp(d_{|a|_c}^+(b_k))$ et à un entier $k\leq l$, cette égalité devient
$$\lambda_{b'}^{d_{|a|_c}^+(a_{\geq l})} =\lambda_{b'}^{d_{|a|_c}^+(b_l)} + \lambda_{b'}^{d_{|a|_c}^+(a_{\geq l+1})}
$$
et par suite
$$ \lambda_{b'}^{d_{|a|_c}^+(a_{\geq k})} =\lambda_{b'}^{d_{|a|_c}^+(b_k)}+ \lambda_{b'}^{d_{|a|_c}^+(b_{k+1})}+...+  \lambda_{b'}^{d_{|a|_c}^+(b_m)} =1 +\sum_{k<l} \lambda_{b'}^{d_{|a|_c}^+(b_l)}$$
L'inégalité \ref{equation} appliquée  à $b'$ devient alors
$$1\geq \lambda_{b'}^{c} \geq 1 +\sum_{l>k} \lambda_{b'}^{d_{|a|_c}^+(b_l)} + \lambda_{b'}^{ r_{|a|_c}(a)}
\mbox{~~~ et donc ~~~} 
\left\{\def\arraystretch{1.4}\begin{array}{cl}
 \lambda_{b'}^{ r_{|a|_c}(a)}=0& \\
 \lambda_{b'}^{d_{|a|_c}^+(b_l)} =0& \mbox{pour tout $k<l$} \\
\end{array}\right.
$$
Ce calcul étant vrai pour tout élément $b'$ du support de $d_{|a|_c}^+(b_k)$, on a pour tout $k<l$
$$supp(d_{|a|_c}^+(b_k))\cap supp(d_{|a|_c}^+(b_l))=\emptyset \mbox{~~~~et~~~~}
supp(d_{|a|_c}^+(b_k))\cap supp(r_{|a|_c}(a))=\emptyset,$$
et donc
$$d_{|a|_c}^+ b_k\wedge d_{|a|_c}^+ b_l= 0 \mbox{~~~~et~~~~} d_{|a|_c}^+ b_k\wedge r_{|a|_c}(a)= 0 .$$

La démonstration dans le cas $\alpha = -$ est analogue.
\end{proof}

\subsection{La $\omega$-catégorie des chaînes}
On va utiliser les chaînes cohérentes pour décrire une autre façon d'associer à un complexe dirigé augmenté $K$ admettant une base sans boucles et unitaire, une $\omega$-catégorie $\mu K$. On démontrera le théorème \ref{theo:muK est une omega cat} qui affirme que cette assignation est un foncteur isomorphe à $\nu$, c'est-à-dire que $\mu\cong\nu : \CDAB\to \infcat$. L’intérêt de ce nouveau formalisme est qu'il sera plus simple d'exprimer et de démontrer le théorème \ref{theo:decomposition explicite} de décomposition des cellules de $\mu K\cong \nu K$ en éléments de la base.

\begin{defi}
On définit l'ensemble globulaire $\mu K$, dont les cellules de dimension $n$ sont les chaînes cohérentes de degré inférieur ou égal à $n$. Les morphismes sources et buts sont ceux définis en \ref{defi:des sources et but}.
\end{defi}

On fixe un complexe dirigé augmenté $K$ admettant une base sans boucles et unitaire $B$.
On veut maintenant montrer que l'on peut munir l'ensemble globulaire des chaînes cohérentes d'une structure de $\omega$-catégorie. Il est facile de définir les $k$-compositions et les unités. Pour montrer qu'elles vérifient les conditions  de distributivité et d'associativité, on va construire un isomorphisme entre les chaînes cohérentes et les tableaux cohérents, qui respectera les sources, les buts et les compositions, ce qui impliquera le résultat.

\begin{defi}
\label{defi: d s}
Soit $n$ un entier. On définit: 
$$
\begin{array}{rrcl}
\def\arraystretch{1.4}
\phi_n : &\nu K_n&\to& \mu K _n\\
 & x &\mapsto&  x^+_n + \sum_{k<n} (x^+_k - \partial^+_k(x^+_{k+1})) \\ &&& = x^-_n + \sum_{k<n} (x^-_k - \partial_k^-(x^-_{k+1}))
\\
\\
\psi_n : &\mu K_n&\to& \nu K _n\\
 & a &\mapsto&  \psi(a)
\end{array}
$$
où le tableau $\psi(a)$ est défini par :
$$~~~~~~~~\psi(a)	 :=
\left(\begin{matrix}
(d_{0}^-(a))_{0} &...&(d_{n-1}^-(a))_{n-1}&(a)_n\\
(d_{0}^+(a))_{0} &...&(d_{n-1}^+(a))_{n-1}&(a)_n
\end{matrix}\right)$$ 
\end{defi}

Il est immédiat que l'image par $\psi$ d'une chaîne cohérente  est un tableau cohérent. La propriété analogue pour $\phi$ découle de la proposition suivante.

\begin{prop}
$\phi$ et $\psi$ commutent avec les applications sources et buts. 
\end{prop}
\begin{proof}
Soient $x$ un tableau de dimension $n$ et $\alpha\in\{-,+\}$.
$$
\def\arraystretch{1.4}
\begin{array}{rcl}
\phi_{n-1}(d^\alpha_{n-1}(x)) &=& x^\alpha_{n-1}+ \sum_{k<n-1} (x^+_k - \partial_k^+(x_{k+1}^+)) = x^\alpha_{n-1}+ \sum_{k<n-1} (x^-_k - \partial_k^-(x_{k+1}^-))\\
&=& (\partial_{n-1}^\alpha x_n) + x^\alpha_{n-1}-(\partial_{n-1}^\alpha x_n)  + \sum_{k<n-1} (x^\alpha_k - \partial_k^\alpha(x_{k+1}^\alpha))\\
&=& d_{n-1}^\alpha (x_n + x^\alpha_{n-1}-(\partial_{n-1}^\alpha x_n) + \sum_{k<n-1} (x^\alpha_k - \partial_k^\alpha(x_{k+1}^\alpha)))\\
&=& d_{n-1}^\alpha (\phi_{n}(x))\\
\end{array}
$$

Il est immédiat que $\psi$ respecte les sources et buts. En effet, pour une chaîne de degré inférieur ou égal à $n$ et pour  $k<n$ et $\alpha,\beta\in\{-,+\}$, 
$$
\def\arraystretch{1.4}
\begin{array}{ll}
\mbox{si }k<n-1 ~~~~&(\psi_{n-1}(d_{n-1}^\beta a))_k^\alpha = d_{k}^\alpha(d_{n-1}^\beta a) = d_{k}^\alpha(a) = (\psi_n (a))_k^\alpha
= (d_{n-1}^\beta (\psi_n(a)))^\alpha_{k}\\
\mbox{si }k=n-1 &(\psi_{n-1}(d_{n-1}^\beta a))_{n-1}^\alpha  = d_{n-1}^\beta a = (\psi_n (a))_{n-1}^\beta = (d_{n-1}^\beta( \psi_n(a)))^\alpha_{n-1}
\end{array}
$$
\end{proof}
On peut donc en déduire  directement la proposition suivante :
\begin{prop}
Les applications $\phi$ et $\psi$ définissent des morphismes d'ensembles globulaires. 
\end{prop}

\begin{prop}
Les morphismes $\phi$ et $\psi$ sont des inverses l'un de l'autre.
\end{prop}
\begin{proof}
Soit $x$ un tableau de dimension $n$.  Montrons que  pour tout $k\leq n$ et $\alpha\in\{-,+\}$, $\big(\psi\circ\phi (x)\big)_k^\alpha = x_k^{\alpha}$. 

Le cas $k=n$ est simple:
$$\def\arraystretch{1.4}\begin{array}{rcl}
\big(\psi\circ\phi (x)\big)_n^\alpha&=& (x_n^+ + \sum_{k<n} (x^+_k - \partial^+_{k}(x^+_{k+1})))_n \\
&=& x^+_n.
\end{array}$$

Remarquons que pour $k<n$, on a:

$$\def\arraystretch{1.4}\begin{array}{rcl}
\big(\psi\circ\phi (x)\big)_k^\alpha&=& (d_k^\alpha \big(\psi\circ\phi (x)\big))_{k}^\alpha \\
&=&\big(\psi\circ\phi (d_k^\alpha x)\big)_{k}^\alpha.
\end{array}$$

Le tableau $d_k^\alpha x$ est de dimension $k$. En utilisant le calcul précédent, on obtient
$$ \big(\psi\circ\phi (d_k^\alpha x)\big) = (d_k^\alpha x)_k^\alpha= x_{k}^\alpha .$$

Montrons maintenant par récurrence sur les degrés des chaînes cohérentes que $\phi\circ\psi$ est égal à l'identité. Pour une chaîne de degré $0$, le résultat est immédiat. Supposons le vrai pour les chaînes de degré strictement inférieur à $n$ et donnons nous $a$ une chaîne cohérente de degré  $n$. Par hypothèse de récurrence, $\phi(\psi(d_{n-1}^+ a)) = d_{n-1}^+ a$, et donc 
$$
 (d^+_{n-1} a)_{n-1}  + \sum_{k<n-1} ((d^+_k a)_k - \partial_k^+((d^+_{k+1} a)_{k+1}))= d_{n-1}^+ a .
$$
On a donc : 
$$
\def\arraystretch{1.4}
\begin{array}{rcl}
\phi(\psi(a)) &= & (a)_n + \sum_{k<n} ((d^+_k a)_k - \partial_{k}^+((d^+_{k+1} a)_{k+1})) \\ 
&= & (a)_n + (d^+_{n-1} a)_{n-1} - \partial_{n-1}^+(( a)_{n}) +  \sum_{k<n-1} ((d^+_k a)_k - \partial_{k}^+((d^+_{k+1} a)_{k+1})) \\ 
& =& (a)_n + d_{n-1}^+ a - \partial_{n-1}^+ ((a)_n) \\
&=& (a)_n + r_{n-1}(a) \\
&=& a.
\end{array}
$$
\end{proof}

\begin{rem}
La proposition précédente implique donc qu'un tableau $x$ est cohérent si et seulement si son image par $\psi$ est cohérente. On peut alors déduire de la proposition \ref{prop:pas de double} que pour tout entier $n$ inférieur à la dimension de $x$ et pour tout $\alpha\in\{-,+\}$, 
$x^\alpha_n\leq 1$. De plus la chaîne $\partial_{n-1}^{\alpha}(x_n^{\alpha})$ est inférieure à $x_{n-1}^{\alpha}$, on a donc aussi $\partial_{n-1}^{\alpha}(x_n^{\alpha})\leq 1$.
\end{rem}

On peut maintenant promouvoir l'assignation $\mu$ en un foncteur.
\begin{defi}
\label{defi : fonctorialite de mu}
On définit le foncteur 
$$\begin{array}{rccc}
\mu : &\CDAB &\to& \textbf{Glob}\\
&K&\mapsto &\mu K\\
& f:K\to K'&\mapsto &\psi\circ\nu(f)\circ\phi:\mu K\to\mu K'.
\end{array}
$$
Pour un morphisme $f:K\to K'$, on a alors 
$$
\begin{array}{rrcl}
\mu(f) :& \mu K&\to&\mu K'\\
&a&\mapsto   
 & f_{|a|}(a_{|a|}) + \sum_{n<[a|} \bigg(f_n\big((d^+_n a)_n\big) - \partial_{n}^+f_{n+1}\big((d^+_{n+1} a)_{n+1})\big)\bigg)\\
&&&= f_{|a|}(a_{|a|}) + \sum_{n<[a|} \bigg(f_n\big((d^-_n a)_n\big) - \partial_{n}^-f_{n+1}\big((d^-_{n+1} a)_{n+1})\big)\bigg)\\

\end{array}
$$
\end{defi}

\begin{defi}
\label{defi:composition dans les chaines}
Pour $a$ et $a'$ deux chaînes telles que $ d_k^-(a) =d_k^+(a') =: c$, on définit leur $k$-composition  de la façon suivante : 
$$a *_k a' := (a -c+a')_+.$$
\end{defi}

\begin{rem}
\label{rem:rem sur la compo explicite des chaines}
Soient $a$ et $a'$ deux chaînes telles que $ d_k^-(a) =d_k^+(a')$, on a alors pour tout $n>k$, 
$$(a*_k a')_n =(a)_n+(a')_n,$$ 
et donc pour tout élément $b$ de la base de dimension strictement supérieure à $k$:
$$\lambda_b^{a *_k a' } = \lambda_b^{a  } + \lambda_b^{a' }.$$
\end{rem}

\begin{prop}
\label{prop:phi et psi preservent les compositions}
Les morphismes $\phi$ et $\psi$ préservent les compositions.
\end{prop}

\note{
Pour démontrer cette proposition, on  a besoin de plusieurs lemmes:
\begin{lem}
\label{lem:sources des composition, cas horizontal 1}
Soient $k<n$ deux entiers, et $a$ et $a'$ deux chaînes cohérentes telles que  $n+1\geq \max(|a|,|a'|)$, et telles que $ d_{k}^-(a) =d_{k}^+(a')$. Alors, pour tout $\alpha\in\{-,+\}$,  $$ \partial_{n}^\alpha((a)_{n+1}) \wedge \partial_{n}^{-\alpha}((a')_{n+1})=0.$$
\end{lem}
\begin{proof}
Si $n\geq|a|$ ou $n\geq|a'|$, la formule est trivialement vraie. On se place donc dans le cas où $|a|=|a'|=n+1$.
Supposons qu'il existe un élément de la base $b$ de dimension $n$ inclus dans $ \partial_{n}^{\alpha}((a)_{n+1})\wedge  \partial_{n}^{-\alpha}((a')_{n+1})$, et montrons que cela mène à une contradiction. Tout d'abord, remarquons que cela implique:
$$\lambda_b^{\partial^{-\alpha}_{n}((a)_{n+1})}=0 ~~~~~~~~~~\lambda_b^{\partial^{\alpha}_{n}((a')_{n+1})}=0.$$
La chaîne $\partial_{n}^{\alpha}((a)_{n+1})$ étant incluse dans $(d_{n}^{\alpha}(a))_{n}$, cela implique que
$$\lambda_b^{(d_{n}^{\alpha}(a))_n}=1.$$
De plus, le lemme \ref{lem:pseudolinearite} appliqué à la chaîne $\partial_{n}^{\alpha}((a)_{n+1})$  implique qu'il existe un élément $b'\in (a)_{n+1}$ tel que $b$ appartienne à $\partial_{n}^{\alpha}(b')$.

Intéressons nous à la chaîne $d$ définie par la formule:
$$d:=\phi(\psi(a)*_{k}\psi(a')).$$
Étant l'image  d'un tableau cohérent, cette chaîne est cohérente. De plus, par la définition de $\phi$ et $\psi$, elle vérifie l'égalité suivante:
$$d := (a)_{n+1} + (a')_{n+1} + (d^\alpha_{n}(a))_{n} + (d^\alpha_{n}(a'))_{n} - \partial^\alpha_{n}((a)_{n+1}+(a')_{n+1}) +r_{n-1}(d).$$
On a alors
$$\def\arraystretch{1.4}
\begin{array}{rcl}
\lambda_b^d 
 &\geq& \lambda_b^{(d^\alpha_{n}(a))_n }- \lambda_b^{\partial^\alpha_{n}((a)_{n+1}+(a')_{n+1}) }\\
&\geq & \lambda_b^{(d^\alpha_{n}(a))_n }- |\lambda_b^{\partial^\alpha_{n}((a)_{n+1})} -\lambda_b^{\partial^{-\alpha}_{n}((a')_{n+1}) }|_+ -|\lambda_b^{\partial^\alpha_{n}((a')_{n+1})} -\lambda_b^{\partial^{-\alpha}_{n}((a)_{n+1})  }|_+\\
&\geq & 1- |1 -1|_+ -|0 -0|_+ =1, 
\end{array}$$
et on a donc $b\in r_{n}(d)$. Par la définition de $b'$, on a alors:
$$
\begin{array}{rcl}
\partial^\alpha_{n}(b')\wedge r_{n}(d)&\geq & b\\
&\neq& \emptyset
\end{array}$$
L'élément $b'$ étant dans le support de $d$, c'est en contradiction avec le deuxième point de la proposition \ref{prop: plus de fourche }.
\end{proof}

\begin{lem}
\label{lem:sources des composition, cas horizontal 2}
Soient $k<n$ deux entiers, et $a$ et $a'$ deux chaînes cohérentes telles que $d_{k}^-(a) =d_{k}^+(a'):=c$. Alors, pour tout $\alpha\in\{-,+\}$,  $$d_{n}^\alpha(a*_k a') =  (d_{n}^\alpha(a)*_kd_{n}^\alpha(a')).$$
\end{lem}
\begin{proof}
Si $n\geq \max(|a|,|a'|)$, l'égalité est vérifiée. Plaçons nous maintenant dans le cas où $n=\max(|a|,|a'|)-1$.
On a alors 
$$\def\arraystretch{1.4}
\begin{array}{rcl}
d_{n}^{\alpha}((a+a'-c)_+) &:=& \partial_{n}^\alpha( (a)_{n+1} + (a')_{n+1})+r_{n}((a+a'-c)_+)\\
                       &:=& (\partial_{n}^\alpha( (a)_{n+1})\diagdown \partial_{n}^{-\alpha}( (a')_{n+1})) + (\partial_{n}^\alpha( (a')_{n+1})\diagdown \partial_{n}^{-\alpha}( (a)_{n+1}))+r_{n}((a+a'-c)_+)\\
\end{array}$$
Or le lemme \ref{lem:sources des composition, cas horizontal 1} implique que 
$$ \partial_{n}^\alpha( (a)_{n+1})\diagdown \partial_{n}^{-\alpha}( (a')_{n+1})=\partial_{n}^\alpha( (a)_{n+1})~~~~~\partial_{n}^\alpha( (a')_{n+1})\diagdown \partial_{n}^{-\alpha}( (a)_{n+1})=\partial_{n}^\alpha( (a')_{n+1})$$
et  on remarque de plus que 
$$r_{n}((a+a'-c)_+) = (a)_n+(a')_n+r_{n-1}((a+a'-c)_+).$$
On en déduit:
$$\def\arraystretch{1.4}
\begin{array}{rcl}
d_{n}^{\alpha}((a+a'-c)_+) &:=& (\partial_{n}^\alpha( (a)_{n+1}) + (\partial_{n}^\alpha( (a')_{n+1})+(a)_n+(a')_n+r_{n-1}((a+a'-c)_+) \\
                       &:=& (d_{n}^\alpha( a) + d_{n}^\alpha( a') -c)_+
\end{array}$$
d'où $$d_{n}^{\alpha}(a*_ka')=d_{n}^\alpha( a) *_kd_{n}^\alpha( a').$$
Pour le cas $n<\max(|a|,|a'|)-1$, en utilisant le cas déja démontré, on a alors:
$$d_{n}^{\alpha}(a*_ka')= d_{n}^{\alpha}d_{n+1}^{\alpha}...d_{m}^{\alpha}(a*_ka')=d_{n}^{\alpha}d_{n+1}^{\alpha}...d_{m}^{\alpha}( a) *_kd_{n}^{\alpha}d_{n+1}^{\alpha}...d_{m}^{\alpha}( a')=d_{n}^\alpha( a) *_kd_{n}^\alpha( a'),$$
où $m= \max(|a|,|a'|)-1$.
\end{proof}

\begin{lem}
\label{lem:sources des composition, cas vertical 1}
Soient $n$ un entier, et $a$ et $a'$ deux chaînes cohérentes telles que  $ \max(|a|,|a'|)=n+1$ et  telles que $ d_{n}^-(a) =d_{n}^+(a')$. Alors, $ \partial_{n}^+((a)_{n+1}) \wedge \partial_{n}^-((a')_{n+1})=0$.
\end{lem}
\begin{proof}
La démonstration de ce lemme est très similaire à celle du lemme \ref{lem:sources des composition, cas horizontal 1}.
Si $n\geq|a|$ ou $n\geq|a'|$,  l'égalité est vérifiée. On se place donc dans le cas où $|a|=|a'|=n+1$.
Supposons qu'il existe un élément de la base $b$ de dimension $n$ inclus dans $ \partial_{n}^{+}((a)_{n+1})\wedge  \partial_{n}^{-}((a')_{n+1})$, et montrons que cela mène à une contradiction. Tout d'abord, remarquons que cela implique:
$$\lambda_b^{\partial^{-}_{n}((a)_{n+1})}=0 ~~~~~~~~~~\lambda_b^{\partial^{+}_{n}((a')_{n+1})}=0.$$
La chaîne $\partial_{n}^{+}((a)_{n+1})$ étant incluse dans $(d_{n}^{+}(a))_{n}$, cela implique que
$$\lambda_b^{(d_{n}^{+}(a))_n}=1.$$
De plus, le lemme \ref{lem:pseudolinearite} appliqué à la chaîne $\partial_{n}^{\alpha}((a)_{n+1})$  implique qu'il existe un élément $b'\in (a)_{n+1}$ tel que $b$ appartienne à $\partial_{n}^{\alpha}(b')$.

Intéressons nous à la chaîne $d$ définie par la formule:
$$d:=\phi(\psi(a)*_{n}\psi(a')).$$
Étant l'image  d'un tableau cohérent, cette chaîne est cohérente. De plus, par la définition de $\phi$ et $\psi$, elle vérifie l'égalité suivante:
$$d := (a)_{n+1} + (a')_{n+1} + (d^+_{n}(a))_{n}  - \partial^+_{n}((a)_{n+1}+(a')_{n+1}) +r_{n-1}(d).$$
On a alors
$$\def\arraystretch{1.4}
\begin{array}{rcl}
\lambda_b^d 
 &=& \lambda_b^{(d^+_{n}(a))_n }- \lambda_b^{\partial^+_{n}((a)_{n+1}+(a')_{n+1}) }\\
&\geq & \lambda_b^{(d^+_{n}(a))_n }- |\lambda_b^{\partial^+_{n}((a)_{n+1})} -\lambda_b^{\partial^{-}_{n}((a')_{n+1}) }|_+ -|\lambda_b^{\partial^+_{n}((a')_{n+1})} -\lambda_b^{\partial^{-}_{n}((a)_{n+1})  }|_+\\
&\geq & 1- |1 -1|_+ -|0 -0|_+ =1, 
\end{array}$$
et on a donc $b\in r_{n}(d)$. Par la définition de $b'$, on a alors:
$$
\begin{array}{rcl}
\partial^+_{n}(b')\wedge r_{n}(d)&\geq & b\\
&\neq& \emptyset
\end{array}$$
L'élément $b'$ étant dans le support de $d$, c'est en contradiction avec le deuxième point de la proposition \ref{prop: plus de fourche }.
\end{proof}

\begin{lem}
\label{lem:sources des composition, cas vertical 2}
Soient $n\leq k$ un entier, et $a$ et $a'$ deux chaînes cohérentes telles que $ d_{k}^-(a) =d_{k}^+(a'):=c$. Alors,  
$$d_{n}^+(a*_k a') = d_n^+(a)~~~~~~d_{n}^-(a*_k a') = d_n^-(a').$$
\end{lem}
\begin{proof}
Si $k=n$ et $n\geq \max(|a|,|a'|)$, les formules sont trivialement vraies. Plaçons nous maintenant dans le cas où $n=\max(|a|,|a'|)-1$ et $k=n$.
On a alors 
$$\def\arraystretch{1.4}
\begin{array}{rcl}
d_{n}^{+}(a*_{n-1}a')&:=&d_{n}^{+}((a+a'-c)_+) \\
     &=& \partial_{n}^+( (a)_{n+1} + (a')_{n+1})+r_{n}((a+a'-c)_+)\\
                       &:=& (\partial_{n}^+( (a)_{n+1})\diagdown \partial_{n}^{-}( (a')_{n+1})) + (\partial_{n}^+( (a')_{n+1})\diagdown \partial_{n}^{-}( (a)_{n+1}))+r_{n}((a+a'-c)_+)\\
\end{array}$$
Or le lemme \ref{lem:sources des composition, cas horizontal 1} implique que 
$$ \partial_{n}^+( (a)_{n+1})\diagdown \partial_{n}^{-}( (a')_{n+1})=(\partial_{n}^+( (a)_{n+1}).$$
On remarque de plus que 
$$\def\arraystretch{1.4}
\begin{array}{rcl}
r_{n}((a+a'-c)_+) &= &(r_n(a) + r_n(a')- d_{n}^-(a))_+\\
&=&(r_n(a')- \partial_{n}^-((a)_{n+1}))_+\\
&=&r_n(a')\diagdown\partial_{n}^-((a)_{n+1}),\\
\end{array}$$
et donc
$$\def\arraystretch{1.4}
\begin{array}{rcl}
d_{n}^{+}(a*_{n-1}a')&= &\partial_{n}^+( (a)_{n+1})+ (\partial_{n}^+( (a')_{n+1} + r_n(a'))\diagdown\partial_{n}^-((a)_{n+1})\\
&=&\partial_{n}^+( (a)_{n+1}) + (d_{n}^+(a'))\diagdown\partial_{n}^-((a)_{n+1})\\
&=&\partial_{n}^+( (a)_{n+1}) + (d_{n}^-(a))\diagdown\partial_{n}^-((a)_{n+1})\\
&=&\partial_{n}^+( (a)_{n+1}) + r_n(a)\\
&=&d_{n}^{+}(a) \\
\end{array}$$

Pour le cas $n\leq k$ et $n\leq \max(|a|,|a'|)-1$, en utilisant le lemme \ref{lem:sources des composition, cas horizontal 2} et le cas que l'on a déja prouvé,	 on a
$$\def\arraystretch{1.4}
\begin{array}{rcl}
d_{n}^{+}(a*_{k}a')&=& d_{n}^{+}d_{k}^{+}d_{k+1}^{+}(a*_{k}a')\\	
&=& d_{n}^{+}d_{k}^{+}(d_{k+1}^{+}(a)*_{k}d_{k+1}^{+}(a'))\\
&=& d_{n}^{+}d_{k}^{+}d_{k+1}^{+}(a) =d_{n}^{+}(a) \\
\end{array}$$

Pour tout $n\leq k$, on montre de façon analogue l'égalité
$$d_{n}^{-}(a*_{k}a')=d_{n}^{-}(a).$$
\end{proof}

\begin{proof}[Démonstration de la proposition \ref{prop:phi et psi preservent les compositions}]
Comme on sait que $\phi$ et $\psi$ sont inverses l'une de l'autre, il suffit de montrer que  $\psi$ préserve les compositions. On se donne donc deux entiers $m>k$, deux chaînes cohérentes $a,a'\in (\mu K)_m$,  tels que $ d_{k}^-(a) =d_{k}^+(a')$.
Pour tout $\alpha\in\{-,+\}$,  pour tout $n>k$, le lemme \ref{lem:sources des composition, cas horizontal 2} implique que 
$$\def\arraystretch{1.4}
\begin{array}{rcl}
\psi(a*_ka')^\alpha_n&:=& (d_n^\alpha(a*_ka'))_n\\
&=&(d_n^\alpha(a)*_kd_n^\alpha(a')))_n\\
&=&(d_n^\alpha(a))_n+ (d_n^\alpha(a'))_n\\
&=&\psi(a)^\alpha_n+ \psi( a')^\alpha_n.
\end{array}
$$	
Pour tout $n\leq k$, le lemme \ref{lem:sources des composition, cas vertical 2} implique que 
$$\def\arraystretch{1.4}
\begin{array}{rcl}
\psi(a*_ka')^+_n&:=& (d_n^+(a*_ka'))_n\\
&=&(d_n^+(a))_n\\
&=& \psi(a)^+_n\\
\end{array}
~~~~~~~~
\def\arraystretch{1.4}
\begin{array}{rcl}
\psi(a*_ka')^-_n&:=& (d_n^-(a*_ka'))_n\\
&=&(d_n^-(a'))_n\\
&=& \psi(a)^-_n.\\
\end{array}
$$	
On a donc 
$\psi(a*_ka')=\psi(a)*_k\psi(a')$.
\end{proof}
}

\begin{theo}
\label{theo:muK est une omega cat}
L'ensemble globulaire $\mu K$ \note{est muni d'une structure de $\omega$-catégorie, où les compositions sont celles de la définition \ref{defi:composition dans les chaines}. En tant que $\omega$-catégorie, $\mu K$ est  isomorphe à $\nu K$. Enfin, }le foncteur $\mu$ se relève en un foncteur à valeur dans $\infcat$, que l'on note aussi $\mu$, et qui est isomorphe à $\nu$ restreint à $\CDAB$.
\end{theo}

On en déduit directement le corollaire suivant:
\begin{cor}
\label{cor:ajdonction entre chaine et lambda avec unite et counite explicite}
Les foncteurs $\lambda_{|\infcatB}$ et $\mu$ forment une équivalence adjointe 
$$\lambda : \xymatrix{
    \infcatB  \ar@/^/[r] \ar@{}[r]|-{\bot}  &
    \CDAB \ar@/^/[l] }: \mu.$$
Pour une $\omega$-catégorie $C$, l'unité de l'adjonction est donnée par la transformation naturelle : 
$$\def\arraystretch{1.4}
\begin{array}{rrcl}
\eta :&  C &\to & \mu \lambda C \\
& x\in C_n &\mapsto  &  [x]_n + \sum_{k<n} ([d_k^+ x]_k - \partial_k^+([d_{k+1}^+ x]_{k+1})) \\
&&&=[x]_n + \sum_{k<n} ([d_k^- x]_k - \partial_k^-([d_{k+1}^- x]_{k+1})) 
\end{array}
$$
Pour un complexe dirigé augmenté $K$, la co-unité est donnée par :
$$\begin{array}{rrcl}
\pi :&  \lambda \mu K &\to & K \\
& [a]_n\in (\lambda \mu K)_n &\mapsto  &  (a)_n
\end{array}
$$

\end{cor}

Dans son article, Steiner montre que pour un complexe $K$ admettant une base sans boucles et unitaire, la $\omega$-catégorie $\nu K$ est engendrée par composition. La proposition suivante va donner une forme d'algorithme pour créer une telle décomposition.

\begin{theo}
\label{theo:decomposition explicite}
Soit  $a := \sum_{i\leq m} b_i +  r_{|a|_c}(a)$ une chaîne cohérente écrite sous forme ordonnée (au sens de la définition \ref{defi:forme ordonne}). On a alors  une décomposition de la cellule $a$ de la forme 

$$\begin{array}{rl}
a = &b_0 + ~(d_{|a|_c}^+(\sum_{0<i\leq m}b_i )\diagdown d_{|a|_c}^- b_0)~+ r_{|a|_c}(a) \\

&*_{|a|_c} ...\\
&*_{|a|_c} b_k +~(d_{|a|_c}^-(\sum_{i<k}b_i)\diagdown d_{|a|_c}^+ b_k) \vee (d_{|a|_c}^+(\sum_{k<i\leq m}b_i) \diagdown d_{|a|_c}^- b_k)~+ r_{|a|_c}(a)\\
&*_{|a|_c} ...\\
&*_{|a|_c} b_m +~(d_{|a|_c}^-(\sum_{i<m}b_i)\diagdown d_{|a|_c}^+ b_m )~+ r_{|a|_c}(a)\\
\end{array}$$

\end{theo}
\begin{proof}
On définit :  $$\beta_k  := b_k +~(d_{|a|_c}^-(a_{<k})\diagdown d_{|a|_c}^+ b_k) \vee (d_{|a|_c}^+(a_{\geq k+1}) \diagdown d_{|a|_c}^- b_k)~+ r_{|a|_c}(a)$$

Pour montrer qu'on a la décomposition voulue, il suffit de vérifier les trois points suivants: 	
\begin{enumerate}
\item Les compositions sont bien définies, c'est-à-dire
$d_{|a|_c}^- \beta_k = d_{|a|_c}^+ \beta_{k+1}$;
\item Les chaînes $\beta_k$ sont cohérentes;
\item On a bien $a =\beta_0 *_{|a|_c} \beta_1 *_{|a|_c}... *_{|a|_c} \beta_m $.
\end{enumerate}
Pour le point $(1)$, on calcule que pour $k<m$: 
$$
\def\arraystretch{1.4}\begin{array}{rcl}
d_{|a|_c}^- \beta_k &=&d_{|a|_c}^-  ( b_k +~(d_{|a|_c}^-(a_{<k})\diagdown d_{|a|_c}^+ b_k) \vee (d_{|a|_c}^+(a_{\geq k+1}) \diagdown d_{|a|_c}^- b_k)~+ r_{|a|_c}(a)) \\
&=& d_{|a|_c}^- b_k +~(d_{|a|_c}^-(a_{<k})\diagdown d_{|a|_c}^+ b_k) \vee (d_{|a|_c}^+(a_{\geq k+1}) \diagdown d_{|a|_c}^- b_k)~+ r_{|a|_c}(a)  \\
&=& (d_{|a|_c}^- b_k + d_{|a|_c}^-(a_{<k})\diagdown d_{|a|_c}^+ b_k)) \vee d_{|a|_c}^+(a_{\geq k+1}) +  r_{|a|_c}(a)\\
&=& d_{|a|_c}^-(a_{<k+1}) \vee d_{|a|_c}^+(a_{\geq k+1}) + r_{|a|_c}(a) \\
\end{array}
$$
De façon analogue, pour $k>0$, 
$$d_{|a|_c}^+ \beta_k = d_{|a|_c}^-(a_{<k} )\vee d_{|a|_c}^+(a_{\geq k})+ r_{|a|_c}(a) $$

Cela prouve donc le point $(1)$. De plus, selon la proposition \ref{prop:nom provisoire} les chaînes $d_{|a|_c}^-(a_{<k} )\vee d_{|a|_c}^+(a_{\geq k})+ r_{|a|_c}(a)$ sont cohérentes, et donc les  $\beta_k$ le sont aussi. Cela prouve le point $(2)$.

Enfin pour le point $(3)$, on a : 
$$\begin{array}{rcl}
\beta_0 *_{|a|_c}... *_{|a|_c} \beta_m &=& \sum_{ k \leq m} b_k + \big( \sum_{k \leq m} (d_{|a|_c}^-(a_{<k})\diagdown d_{|a|_c}^+ b_k) \vee (d_{|a|_c}^+(a_{\geq k+1}) \diagdown d_{|a|_c}^- b_k) \\
&&- \sum_{1 \leq k \leq m}d_{|a|_c}^-(a_{<k} )\vee d_{|a|_c}^+(a_{\geq k})
+ r_{|a|_c}(a)\big)_+.\\
\end{array}
$$
Or selon le lemme \ref{lem:pseudolinearite}, $$(d_{|a|_c}^-(a_{<0})\diagdown d_{|a|_c}^+ b_0) \vee (d_{|a|_c}^+(a_{\geq 1})\diagdown d_{|a|_c}^- b_0) =(d_{|a|_c}^+(a_{\geq 1})\diagdown d_{|a|_c}^- b_0)\leq \sum_{1\leq k} d_{|a|_c}^+(b_k), $$
et pour $k> 0$,
$$(d_{|a|_c}^-(a_{<k})\diagdown d_{|a|_c}^+ b_k) \vee (d_{|a|_c}^+(a_{\geq k+1}) \diagdown d_{|a|_c}^- b_k) + d_{|a|_c}^+ b_k = d_{|a|_c}^-(a_{<k} )\vee d_{|a|_c}^+(a_{\geq k}),$$
et donc 
$$
\sum_{k \leq m} (d_{|a|_c}^-(a_{<k})\diagdown d_{|a|_c}^+ b_k) \vee (d_{|a|_c}^+(a_{\geq k+1}) \diagdown d_{|a|_c}^- b_k)
\leq \sum_{1\leq k \leq m} d_{|a|_c}^-(a_{<k} )\vee d_{|a|_c}^+(a_{\geq k}).
$$
Le point $2$ de la proposition \ref{prop: plus de fourche }  indique que pour tout $k$ et pour $\alpha\in\{-,+\}$,  $r_{|a|_c}(a)\wedge d_{|a|_c}^\alpha(b_k) = 0$, et le lemme \ref{lem:pseudolinearite} implique alors que 
$r_{|a|_c}(a)\wedge d_{|a|_c}^\alpha(a_{<k}) \leq \sum_{i< k} r_{|a|_c}(a)\wedge d_{|a|_c}^\alpha(b_i) = 0$, et on a de même 
$r_{|a|_c}(a)\wedge d_{|a|_c}^\alpha(a_{\geq k})=0 $ et donc pour tout $k$,
$r_{|a|_c}(a)\wedge (d_{|a|_c}^-(a_{<k} )\vee d_{|a|_c}^+(a_{\geq k}))=0.$
En combinant tout on obtient
$$
\big(\sum_{k \leq m} (d_{|a|_c}^-(a_{<k})\diagdown d_{|a|_c}^+ b_k) \vee (d_{|a|_c}^+(a_{\geq k+1}) \diagdown d_{|a|_c}^- b_k) -\sum_{1\leq k \leq m} d_{|a|_c}^-(a_{<k} )\vee d_{|a|_c}^+(a_{\geq k}) ~~+ 
r_{|a|_c}(a)\big)_+=r_{|a|_c}(a),
$$
et donc enfin,
$$
\beta_0 *_{|a|_c}... *_{|a|_c} \beta_m = \sum_{ k \leq m} b_k+ r_{|a|_c}(a) = a.$$
\end{proof}

\begin{cor}
\label{cor:decomposition explicite}
Soit $a$ une chaîne cohérente. Alors  $a$ est  une composition d'éléments de la base. De plus, les $|a|$-cellules apparaissant dans cette décomposition sont les $b\in (a)_{|a|}$.
\end{cor}
\begin{proof}
On va montrer le résultat par récurrence sur le degré de composition de $a$. L’initialisation correspond au cas où le degré de composition est égal à $-1$, et la proposition \ref{prop:coherende de degre moin un est reduit a un singloton} implique directement le résultat.

Supposons donc le résultat vrai pour $n-1$ et donnons nous une chaîne $a := \sum_{i\leq m} b_i +  r_{|a|_c}(a)$ de degré de composition $n$. On définit alors

 $$\beta_k  := b_k +~(d_{|a|_c}^-(a_{<k})\diagdown d_{|a|_c}^+ b_k) \vee (d_{|a|_c}^+(a_{\geq k+1}) \diagdown d_{|a|_c}^- b_k)~+ r_{|a|_c}(a)$$
 
Selon le théorème \ref{theo:decomposition explicite}, on a l'égalité 
$$a = \beta_0 *_{|a|_c} \beta_1 *_{|a|_c}... *_{|a|_c} \beta_m.$$
Les chaînes $\beta_k$ sont de degré de composition strictement inférieur à $n$ et on peut donc leur appliquer l'hypothèse de récurrence. On en déduit que $a$ aussi peut s'exprimer comme une composition d'éléments de la base. 
\end{proof}

\begin{exe}
\label{exe : delta 4}
On note $C_\bullet(\Delta[4])$ le complexe de chaînes réduit associé à $\Delta[4]$:
 $$C_n(\Delta[4]) :=\mathbb{Z} \{\sigma_v : v\in \Delta[4]_n \mbox{~et~}v \mbox{ non dégénéré }\}$$
$$\begin{array}{rccl}
\partial_{n+1}  :& C_{n+1}(\Delta[4])&\to &C_n(\Delta[4]) \\
&v&\mapsto& \sum (-1)^i d_i v
\end{array}
$$
où par convention dans cette somme $d_iv = 0$ si $d_iv$ est un simplexe
dégénéré.

On définit aussi, pour tout $n$, le monoïde additif $C_n^*(\Delta[4])$ engendré par les $n$-simplexes non dégénérés, et une augmentation $e:C_0(\Delta[4])\to \mathbb{Z}$ qui envoie les $0$-simplexes sur $1$. Le triplet $(C_n(\Delta[4]),C_n^*(\Delta[4]),e)$ est alors un complexe dirigé augmenté. De plus, ce complexe admet une base, donnée par l'ensemble des simplexes non dégénérés de $\Delta[4]$.

Dans \cite[Exemple 3.8]{steiner}, Steiner montre que cette base est unitaire et sans boucles.

Servons nous du théorème précédent pour décomposer la source et le but de la $4$-cellule ${\sigma_{0 1 2 3 4} \in \mu(C_\bullet(\Delta[4]))_4)}$ en composition d'éléments de la base.
$$\sigma_{0 1 2 3 4} ~:~\sigma_{0 2 3 4} + \sigma_{0 1 2 4}\to \sigma_{1 2 3 4}+\sigma_{0 1 3 4}+\sigma_{0 1 2 3}$$
Calculons les sources et buts des $3$-cellules correspondant aux $3$-simplexes apparaissant dans la formule précédente.
$$
\begin{array}{rcrcrcl}
{\sigma_{0 2 3 4}}& :& \sigma_{0 3 4}+\sigma_{0 2 3}\to \sigma_{2 3 4}+\sigma_{0 2 4}&:&\sigma_{0 4} \to \sigma_{0 2}+\sigma_{2 3}+\sigma_{3 4}&:& \sigma_0\to \sigma_4\\
{\sigma_{0 1 2 4}}& :& \sigma_{0 2 4}+\sigma_{0 1 2}\to \sigma_{1 2 4}+\sigma_{0 1 4}&:& \sigma_{0 4} \to \sigma_{0 1}+\sigma_{1 2}+\sigma_{2 4}&:& \sigma_0\to \sigma_4\\
\newline\\
{\sigma_{1 2 3 4}}& :& \sigma_{1 3 4}+\sigma_{1 2 3}\to \sigma_{2 3 4}+\sigma_{1 2 4}&:&\sigma_{1 4} \to \sigma_{1 2}+\sigma_{2 3}+\sigma_{3 4}&:& \sigma_1\to \sigma_4\\
{\sigma_{0 1 3 4}}& :& \sigma_{0 3 4}+\sigma_{0 1 3}\to \sigma_{1 3 4}+\sigma_{0 1 4}&:&\sigma_{0 4} \to \sigma_{0 1}+\sigma_{1 3}+\sigma_{3 4}&:& \sigma_0\to \sigma_4\\
{\sigma_{0 1 2 3}}& :&\sigma_{0 2 3}+\sigma_{0 1 2}\to \sigma_{1 2 3}+\sigma_{0 1 3}&:& \sigma_{0 3} \to \sigma_{0 1}+\sigma_{1 2}+\sigma_{2 3} &:& \sigma_0\to \sigma_3 \\
\end{array}
$$
et donc 
$$\sigma_{0 1 2 4} \odot_2 \sigma_{0 2 3 4}  ~~~~~~\sigma_{1 2 3 4}\odot_2 \sigma_{0 1 3 4} \odot_2 \sigma_{0 1 2 3}.
$$
En appliquant le théorème \ref{theo:decomposition explicite}, on obtient : 
$$
\begin{array}{rcl}
\sigma_{0 2 3 4} + \sigma_{0 1 2 4} &=& (\sigma_{0 1 2 4}+  \sigma_{2 3 4}) *_2 (\sigma_{0 2 3 4} + \sigma_{0 1 2})\\
\sigma_{1 2 3 4}+\sigma_{0 1 3 4}+\sigma_{0 1 2 3} &=& (\sigma_{1 2 3 4}+\sigma_{0 1 4})*_2 (\sigma_{0 1 3 4} + \sigma_{1 2 3})*_2 (\sigma_{0 1 2 3}+\sigma_{0 3 4}).
\end{array}
$$
Calculons les sources et buts des $2$-cellules correspondant aux $2$-simplexes apparaissant dans les formules précédentes.
$$
\begin{array}{rcrrl}
{\sigma_{234}}& :& \sigma_{2 4}\to \sigma_{2 3}+\sigma_{3 4}&:& \sigma_2 \to\sigma_4\\
{\sigma_{0 1 2}}& :& \sigma_{0 2}\to \sigma_{0 1}+\sigma_{1 2}&:&\sigma_0 \to\sigma_2\\
{\sigma_{0 1 4}}& :& \sigma_{0 4}\to \sigma_{0 1}+\sigma_{1 4}&:&\sigma_0 \to\sigma_4\\
{\sigma_{1 2 3}}& :& \sigma_{1 3}\to \sigma_{1 2}+\sigma_{2 3}&:&\sigma_1 \to\sigma_3\\
{\sigma_{0 3 4}}& :& \sigma_{0 4}\to \sigma_{0 3}+\sigma_{3 4}&:&\sigma_0 \to\sigma_4\\
\end{array}
$$
et donc 
$$
\begin{array}{rcrrrrr}
\sigma_{2 3 4} &\odot_1 & \sigma_{0 1 2 4} &~~~& \sigma_{0 1 2} &\odot_1 & \sigma_{0 2 3 4}\\
  \sigma_{1 2 3 4}&\odot_1 &\sigma_{0 1 4}&~~~&
 \sigma_{1 2 3} &\odot_1 &\sigma_{0 1 3 4}\\
 \sigma_{0 1 2 3}&\odot_1 &\sigma_{0 3 4}.
\end{array}
$$
En appliquant le théorème \ref{theo:decomposition explicite}, on obtient : 
$$
\begin{array}{rcl}
\sigma_{0 1 2 4}+  \sigma_{2 3 4} &=&( \sigma_{2 3 4} + \sigma_{0 1}+ \sigma_{1 2})*_1\sigma_{0 1 2 4} \\
\sigma_{0 2 3 4} + \sigma_{0 1 2} &=& (\sigma_{0 1 2}+ \sigma_{2 3}+  \sigma_{3 4}) *_1\sigma_{0 2 3 4} \\
\sigma_{1 2 3 4}+\sigma_{0 1 4} &=& (\sigma_{1 2 3 4}+\sigma_{01}) *_1 \sigma_{0 1 4}\\
\sigma_{0 1 3 4} + \sigma_{1 2 3} &=&  ( \sigma_{1 2 3} + \sigma_{0 1}+\sigma_{34}) *_1\sigma_{0 1 3 4} \\
\sigma_{0 1 2 3}+\sigma_{0 3 4} &=& (\sigma_{0 1 2 3}+ \sigma_{3 4}) *_1\sigma_{0 3 4}.
\end{array}
$$
Calculons les sources et buts des $1$-cellules correspondant aux $1$-simplexes apparaissant dans les formules précédentes.
$$
\begin{array}{rcl}
\sigma_{0 1} &:& \sigma_0\to\sigma_1\\
\sigma_{12} &:& \sigma_1\to\sigma_2\\
\sigma_{23} &:& \sigma_2\to\sigma_3\\
\sigma_{34} &:& \sigma_3\to\sigma_4\\
\end{array}
$$
et donc 
$$
\begin{array}{rcl}
\sigma_{2 3 4} \odot_0 \sigma_{1 2} \odot_0 \sigma_{0 1}&~~~&
\sigma_{3 4}  \odot_0 \sigma_{2 3}\odot_0  \sigma_{0 1 2}\\
\sigma_{1 2 3 4}\odot_0 \sigma_{01}&~~~&
\sigma_{34} \odot_0\sigma_{1 2 3} \odot_0 \sigma_{0 1}\\
\sigma_{3 4}\odot_0 \sigma_{0 1 2 3}.\\
\end{array}
$$
En appliquant le théorème \ref{theo:decomposition explicite}, on obtient : 
$$
\begin{array}{rcl}
 \sigma_{2 3 4} + \sigma_{0 1}+ \sigma_{1 2} &=&\sigma_{2 3 4} *_0 \sigma_{1 2} *_0 \sigma_{0 1}\\
 \sigma_{0 1 2}+ \sigma_{2 3}+  \sigma_{3 4}&=&\sigma_{3 4}  *_0 \sigma_{2 3}*_0  \sigma_{0 1 2}\\
 \sigma_{1 2 3 4}+\sigma_{01} &=& \sigma_{1 2 3 4}*_0 \sigma_{01}\\
 \sigma_{1 2 3} + \sigma_{0 1}+\sigma_{34}&=&\sigma_{34} *_0\sigma_{1 2 3} *_0 \sigma_{0 1}\\
 \sigma_{0 1 2 3}+ \sigma_{3 4}&=&\sigma_{3 4}*_0 \sigma_{0 1 2 3}.
\end{array}
$$
En regroupant tout on obtient : 
$$
\begin{array}{rcl}
\sigma_{0 2 3 4} + \sigma_{0 1 2 4} &=& ((\sigma_{2 3 4} *_0 \sigma_{1 2} *_0 \sigma_{0 1})*_1\sigma_{0 1 2 4}) *_2 ((\sigma_{3 4}  *_0 \sigma_{2 3}*_0  \sigma_{0 1 2}) *_1\sigma_{0 2 3 4} )\\
\sigma_{1 2 3 4}+\sigma_{0 1 3 4}+\sigma_{0 1 2 3} &=& (( \sigma_{1 2 3 4}*_0\sigma_{01} ) *_1 \sigma_{0 1 4})*_2 ((\sigma_{34} *_0\sigma_{1 2 3} *_0 \sigma_{0 1}) *_1\sigma_{0 1 3 4})\\&&*_2 ((\sigma_{3 4}*_0 \sigma_{0 1 2 3}) *_1\sigma_{0 3 4})).
\end{array}
$$
\end{exe}

\section{Développements sur les $\omega$-catégories et les complexes dirigés augmentés} 
\subsection{Carrés cocartésiens de complexes dirigés augmentés}
\note{
\begin{defi}
Soit $f:M\to N$  un morphisme entre deux complexes dirigés augmentés admettant des bases unitaires et sans boucles $B_M$ et $B_N$.  Le morphisme $f$ est \textit{quasi-libre} si pour tout $n$ et tout $b\in (B_M)_n$,
$$f_n(b)\neq 0 ~\Rightarrow ~ f_n(b)\in (B_N)_n.$$
Si deux morphismes sont quasi-rigides, alors leur composition l'est aussi. 
\end{defi}

L'objectif de cette partie est de démontrer le théorème suivant: 

\begin{theo}
\label{theo:theo dans le cas preque quasi-rigide}
Soit un carré commutatif dans $\CDA$ tel que tous les complexes dirigés augmentés admettent des bases unitaires et sans boucles : 
\begin{equation}
\tag{$c_0$}
\label{eq:c}
\xymatrix{
K\ar[r]^{k^1} \ar[d]_{k^0} &M_1 \ar[d]^{l^1} \\
M_0 \ar[r]_{l^0}&  M}
\end{equation}
et tel que tous les morphismes soient  quasi-libres. On note $B_K,~B_{M_0},~B_{M_0},~B_{M}$ les bases de $K,~M_0,~M_1,~ M$.
Alors, le carré \ref{eq:c} est cocartésien si et seulement si  le carré induit dans les ensembles:
\begin{equation}
\tag{$c_1$}
\label{eq:cn}
\xymatrix{
B_K\cup\{0\}\ar[r]^{k^0} \ar[d]_{k^1} & B_{M_0}\cup\{0\} \ar[d]^{l^0}\\
B_{M_1}\cup\{0\} \ar[r]_{l^1}&  B_M\cup\{0\}}.
\end{equation}
est cocartésien.
\end{theo}

Jusqu’à la fin de cette partie, on fixe un carré commutatif dans $\CDA$ tel que tous les complexes dirigés augmentés admettent des bases unitaires et sans boucles:
$$
\xymatrix{
K\ar[r]^{k^1} \ar[d]_{k^0} &M_1 \ar[d]^{l^1} \\
M_0 \ar[r]_{l^0}&  M.}
$$
et tel que tous les morphismes soient  quasi-libres. On fixe pour la suite un entier $n$.

\begin{defi}
On définit le morphisme $\gamma: K_n\to \mathbb{Z}$ qui envoie tous les éléments de la base sur $1$.
Si $a$ est un élément de $K$, sa \textit{taille}, noté $t(a)$, est l'entier $\gamma((a)_+)+ \gamma((a)_-)$.
\end{defi}

\begin{defi}
On définit $\sim$ comme la plus petite relation transitive, réflexive et symétrique sur $(B_{M_0})_n\cup\{0\}\coprod (B_{M_1})_n\cup\{0\}$ telle que pour $a\in (B_K)_n\cup\{0\}$, $k^0_n(a)\sim k^1_n(a)$.
Pour un entier $k$, on définit la relation $\sim_k$ sur les éléments de  $(M_0)_n\oplus (M_1)_n$:
$$ a\sim_k a':\equiv \mbox{ il existe $c\in K_n$ tel que $t(c)=k$ et $a-a'=k^0_n(c)-k^1_n(c)$}$$
Enfin, on définit la relation transitive, réflexive et symétrique $\sim_\infty$ sur les éléments de  $(M_0)_n\oplus (M_1)_n$:
$$ a\sim_{\infty} a':\equiv \mbox{il existe $c\in K_n$ tel que $a-a'=k^0_n(c)-k^1_n(c)$}$$
Remarquons que  $a\sim_{\infty} a'$ si et seulement si il existe $k\in\mathbb{N}$ tel que $a\sim_k a'$.
\end{defi}

\begin{rem}
\label{rem:sur les pusout dans cda}
Par définition des sommes amalgamées dans les groupes et dans les ensembles, on a des isomorphismes:
$$(B_{M_0})_n\cup\{0\}\coprod_{(B_K)_n\cup \{0\}} (B_{M_1})_n\cup\{0\} \cong ((B_{M_0})_n\cup\{0\}\coprod (B_{M_1})_n\cup\{0\}\big)_{/\sim}$$
et 
$$ M_0\coprod_{K} M_1 \cong (M_0 \oplus M_1)_{/\sim_{\infty}}.$$
\end{rem}

\begin{lem}
\label{lem:sim et sim infini coiincident}
Soient $b,b'\in (B_{M_0})_n\cup\{0\}\coprod (B_{M_1})_n\cup\{0\}$. Alors $b\sim b'$ si et seulement si $b\sim_{\infty} b'$.
\end{lem}
\begin{proof}
Tout d'abord, remarquons que pour tout $b\in (B_K)_n\cup\{0\}$, $k^0_n(b)\sim_{\infty} k^1_n(b)$. De plus, la relation $\sim_{\infty}$ est transitive, réflexive et symétrique, et donc $b\sim b'$ implique $b\sim_{\infty}b'$. 

Pour la contraposée, on va montrer par récurrence sur $k$ que  $b\sim_k b'$ implique $b\sim b'$.  
Donnons nous $b$ et $b'$ telles que $b\sim_0 b'$. Il existe donc $c\in K_n$ de taille $0$, et tel que $b-b'= k^0_n(c)-k^1_n(c)$. Or, la taille de $c$ étant zéro, $c$ est nul, et donc $b = b'$. La relation $\sim$ étant réflexive, on a bien $b\sim b'$.

Supposons maintenant la propriété démontrée au rang $k$. On se donne donc $b$ et $b'$ telles que $b\sim_{k+1} b'$. Par définition, il existe donc $c\in K_n$ de taille $(k+1)$, tel que $b-b'= k^0_n(c)-k^1_n(c)$. Si $b=b'$, alors on a bien $b\sim b'$. Supposons donc que $b\neq b'$. Quitte à échanger $b$ et $b'$, on peut supposer que $b\neq 0$.
L'élément $b$ est donc dans le support de $k^0_n((c)_+)$ ou dans celui de $k^1_n((c)_-)$. Plaçons nous dans le premier cas, l'autre étant analogue. Le morphisme $k^0 $  étant quasi-libre, il existe un élément $d$ de $(B_{M_0})_n$, appartenant au support de $(c)_+$, et tel que $k^0_n(d)=b$. Si on définit $\tilde{c}:= c-d$, la taille de $\tilde{c}$ est $k$, et 
$$k^1_n(d)-b'= k^0_n(\tilde{c})-k^1_n(\tilde{c}).$$
On a alors
$$b=k^0_n(d)\sim k^1_n(d)\sim _k b',$$
et l'hypothèse de récurrence ainsi que la transitivité de $\sim$ implique $b\sim b'$. Cela conclut la preuve.
\end{proof}

\begin{proof}[Démonstration du théorème \ref{theo:theo dans le cas preque quasi-rigide}]
Supposons tout d'abord que le carré \ref{eq:c} est cocartésien. Pour que le carré \ref{eq:cn} soit cocartésien, la remarque \ref{rem:sur les pusout dans cda} indique qu'il suffit que pour tout entier $n$, pour tout $$b,b'\in (B_{M_0})_n\cup\{0\} \coprod_{(B_K)_n\cup\{0\}}(B_{M_1})_n\cup\{0\} ,$$ $b\sim_\infty b'$ si et seulement si $b\sim b'$. C'est exactement ce que stipule le lemme \ref{lem:sim et sim infini coiincident}.

Réciproquement, supposons que le carré  \ref{eq:cn} est cartésien. Notons $l := l^0-l^1 : M_0\oplus M_1\to M$ le morphisme induit par $l^0$ et $l^1$. Pour que le carré \ref{eq:c} soit cocartésien, la remarque \ref{rem:sur les pusout dans cda} indique qu'il suffit de montrer que pour tout entier $n$, pour tout $a\in (M_0\oplus M_1)_n$ tel que $l_n(a)= 0$, on a $a\sim_{\infty}0$. Donnons nous donc un entier $n$ et un élément $a$ vérifiant ces conditions.

On va procéder par récurrence sur $t(a)$. Si $t(a)=0$ alors $a$ est nul. On a alors $a\sim 0$, et le lemme \ref{lem:sim et sim infini coiincident} implique $a\sim_{\infty}0.$ 

Supposons maintenant que $t(a)=k+1$. Quitte à remplacer $a$ par $-a$, on peut supposer que $(a)_+\neq 0$. Soit $b$ un élément de la base appartenant au support de $(a)_+$. Deux cas sont alors à envisager. 

Le premier est celui ou $l(b)=0$. Le lemme \ref{lem:sim et sim infini coiincident} implique alors  $b\sim_{\infty} 0$. De plus, l'élément $\tilde{a}:= a-b$ vérifie $t(\tilde{a}) = k$, et $l_n(\tilde{a})=0$. L'hypothèse de récurrence implique alors que $\tilde{a}\sim_{\infty}0$. On peut alors en déduire que $a = \tilde{a}+b\sim_\infty 0$. 

Le deuxième cas correspond à celui où $l(b)$ est un élément de la base non nul. Cependant $l(a)=0$, et il existe donc un élément de la base $b'$ qui appartient au support de $(a)_-$ tel que $l(b')=l(b)$. Le lemme \ref{lem:sim et sim infini coiincident} implique donc $b'\sim_{\infty}b$.  L'élément $\tilde{a}:= a-b+b'$ vérifie $t(\tilde{a}) = k-1$, et $l_n(\tilde{a})=0$.  L'hypothèse de récurrence implique alors que $\tilde{a}\sim_{\infty}0$, et donc $a = \tilde{a}-b+b'\sim_\infty 0.$ 
\end{proof}

}

\subsection{Quasi-rigidité}
L'objectif de cette partie est de donner des conditions suffisantes pour que $\nu$ préserve les sommes amalgamées. 

\begin{defi}
Soit $f:M\to N$  un morphisme entre deux complexes dirigés augmentés admettant des bases unitaires et sans boucles $B_M$ et $B_N$.  Le morphisme $f$ est \textit{quasi-rigide} si pour tout $n$, et tout $b\in (B_M)_n$,
$$f_n(b)\neq 0 ~\Rightarrow ~ f_n(b)\in B_N\mbox{ et }\nu(f)\langle b \rangle = \langle f_n(b)\rangle.$$
\end{defi}

\note{
\begin{rem}
Si deux morphismes sont quasi-rigides, alors leur composition l'est aussi. 
De plus, on peut remarquer que les morphismes quasi-rigides sont quasi-libres.
\end{rem}

\begin{rem}
Dans \cite[\S 3.2]{ara2020joint}, Ara et Maltsiniotis définissent la notion de morphisme \textit{rigide} qui correspond aux morphismes $f$ tel que pour tout   $b\in (B_M)_n$,
$$ f_n(b)\in B_N\mbox{ et }\nu(f)\langle b \rangle = \langle f_n(b)\rangle.$$
Ainsi, tout morphisme rigide est quasi-rigide, et un morphisme quasi-rigide $f$ est rigide si, pour tout élément $b\in (B_M)_n$, $f_n(b)\neq 0$. Les deux notions coïncident donc pour les monomorphismes.
\end{rem}

\begin{rem}
Un morphisme quasi-libre $f:M\to N$ est quasi-rigide si et seulement si pour tout élément $b\in (B_M)_n$, tel que $f_n(b)\neq 0$, pour tout $\alpha\in\{-,+\}$ et pour tout $k<n$, $f_{k}(d^\alpha_{k}(b))=d^\alpha_k(f_n(b))$.
\end{rem}
}
\begin{prop}
\label{prop:autre caracterisation de la quasi-rigidite}
Soit $f:M\to N$  un morphisme entre deux complexes dirigés augmentés admettant des bases $B_M$ et $B_N$.
Alors les deux conditions suivantes sont équivalentes : 
\begin{enumerate}
\item $f$ est quasi-rigide;
\item pour tout $n$, et tout  $b\in (B_M)_n$, 
$$f_n(b)\neq 0~ \Rightarrow ~f_n(b)\in B_N\mbox{ et }\forall k< n,~ f_k(\langle b \rangle^-_k)\wedge  f_k(\langle b \rangle^+_k)=0.$$
\end{enumerate} 
\end{prop}
\begin{proof}
Supposons tout d'abord que $f$ est quasi-rigide. On a directement $f_n(b)\in B_N$ et pour tout $k<n$ et $b\in (B_M)_n$ tels que $f_n(b)\neq 0$:
$$ f_k(\langle b \rangle^-_k)\wedge  f_k(\langle b \rangle^+_k) =  \langle f_n(b) \rangle^-_k \wedge  \langle f_n(b) \rangle^+_k = 0.$$

Supposons maintenant que $f$ vérifie la deuxième condition. 
On a $f_n(b)\in B_N$ et pour tout $k<n$ et $b\in (B_M)_n$, tels  que $f_n(b)\neq 0$,
$$f_k(\langle b \rangle^+_k)-  f_k(\langle b \rangle^-_k) =\partial_{k} f_{k+1}(\langle b \rangle^+_{k+1})$$
et comme $f_k(\langle b \rangle^-_k)\wedge  f_k(\langle b \rangle^+_k)=0$,
$$
\begin{array}{rclcl}
f_k(\langle b \rangle^+_k) &=& (\partial_{k} f_{k+1}(\langle b \rangle^+_{k+1}))_+&=& \partial^+_{k} f_{k+1}(\langle b \rangle^+_{k+1})\\
f_k(\langle b \rangle^-_k) &=&(\partial_{k} f_{k+1}(\langle b \rangle^+_{k+1}))_-&=& \partial^-_{k} f_{k+1}(\langle b \rangle^+_{k+1}).
\end{array}
$$
Par une récurrence simple, on en déduit que $\nu(f)(\langle b\rangle) = \langle f_n(b) \rangle$.			
\end{proof}

\begin{rem}
\label{rem:injection sur les bases implique quasi-ridige}
La proposition précédente implique qu'un morphisme $f:M\to N$ qui  induit une injection sur les bases  est quasi-rigide.
\end{rem}

Le reste de cette partie est dédié à la démonstration du théorème suivant:

\begin{theo}
\label{theo:theo dans le cas fort quasi-rigide}
Soit une somme amalgamée dans $\CDA$ telle que tous les complexes dirigés augmentés admettent des bases unitaires et sans boucles : 

$$
\xymatrix{
K\ar[r]^{k^1} \ar[d]_{k^0} &M_1 \ar[d]^{l^1} \\
M_0 \ar[r]_{l^0}& \pushoutcorner M}
$$
et telle que tous les morphismes soient  quasi-rigides.
Alors le carré induit dans $\infcat $:
$$
\xymatrix{
\nu K\ar[r]^{\nu k^1} \ar[d]_{\nu k^0} &\nu M_1 \ar[d]^{\nu l^1} \\
\nu M_0 \ar[r]_{\nu l^0}&\nu  M}
$$
est une somme amalgamée.
\end{theo}
Jusqu’à la fin de cette partie, on fixe une somme amalgamée dans $\CDA$ telle que tous les complexes dirigés augmentés admettent des bases unitaires et sans boucles,  et vérifiant les conditions du théorème précédent.

$$
\xymatrix{
K\ar[r]^{k^1} \ar[d]_{k^0} &M_1 \ar[d]^{l^1} \\
M_0 \ar[r]_{l^0}& \pushoutcorner M.}
$$

On note $B_K,~B_{M_0},~B_{M_0},~B_{M}$ les bases de $K,~M_0,~M_1,~ M$. La quasi-rigidité des morphismes implique l'existence, pour tout $n$, d'un carré commutatif dans la catégorie des ensembles, qui est cartésien selon le théorème  \ref{theo:theo dans le cas preque quasi-rigide}: 
$$\xymatrix{
(B_K)_n\cup\{0\}\ar[r]^{k_n^0} \ar[d]_{k_n^1} & (B_{M_0})_n\cup\{0\} \ar[d]^{l_n^0}\\
(B_{M_1})_n\cup\{0\} \ar[r]_{l_n^1}&  (B_M)_n\cup\{0\}}.$$
Les morphismes induits $l_n:(B_{M_0})_n\cup\{0\}\coprod (B_{M_1})_n\cup\{0\}\to  (B_M)_n\cup\{0\}$ sont des surjections et admettent donc des sections $s_n$.

On note $N$ la somme amalgamée du diagramme suivant : 
$$
\xymatrix{
\nu K\ar[r]^{\nu k^1} \ar[d]_{\nu k^0} &\nu M_1 \ar[d]^{j_1} \\
\nu M_0 \ar[r]_{ j_0}& \pushoutcorner N}.
$$
et on a donc par la propriété universelle de la somme amalgamée, un morphisme $\phi: N\to \nu M$.
\note{
Rappelons que
si $B$ est une base unitaire et sans boucles pour un complexe dirigé augmenté $L$, on note $\langle B\rangle$ l'ensemble 
$\{\langle b \rangle, b\in B\}.$
}

 On définit l'application
$$\begin{array}{rclcl}
j:&\langle B_{M_0}\rangle \coprod  \langle B_{M_1}\rangle&\to & N\\
&b&\mapsto & j_0(\langle b\rangle)&\mbox{ si $b\in B_{M_0}$} \\
&b&\mapsto & j_1(\langle b\rangle)&\mbox{ si $b\in B_{M_1}$}.\\
\end{array}$$

Pour tout $n$, on définit $$E_n :=\{j(\langle s_n(b)\rangle), ~b\in (B_M)_n\}$$ et on pose alors
 $E_{\leq n}:= \cup_{k\leq n} E_n$ et $E = \cup_{ n\in \mathbb{N}} E_n$. On note $\tilde{E}_n$ l'ensemble des cellules de $N$  générées par composition par  $E_{\leq n}$. \note{
Remarquons que la restriction à $E_n$ du  morphisme $\phi$ induit une bijection
$E_n\cong \langle B\rangle.$
}

\begin{defi}
Pour un complexe dirigé augmenté $L$, on note $\tau_n(L)$ le complexe dirigé augmenté défini par: 
$$
\begin{array}{rcll}
(\tau_n(L))_k  &:= &L_k &\mbox{ si $k\leq n$}\\
 &:= &0 &\mbox{ si $k> n$}\\
\end{array}
$$
Cette assignation s’étend  en un foncteur :
$$
\begin{array}{rcll}
\tau_n : & \CDA &\to &\CDA\\
& L&\mapsto& \tau_n(L).\\
\end{array}
$$
Si $B$ est une base unitaire et sans boucles pour $L$, alors $B_{\leq n}:=\cup_{k\leq n} B_k$ est une base unitaire et sans boucles pour $\tau_n(L)$.

Pour une $\omega$-catégorie $C$, on note $\tau_n(C)$ la $\omega$-catégorie vérifiant pour $k\leq n$,
$$(\tau_n(C))_k : =C_k$$
et dont toutes les $k$-cellules pour $k>n$ sont des unités.
Cette assignation s'étend  en un foncteur :
$$
\begin{array}{rcll}
\tau_n : & \infcat &\to &\infcat\\
& C&\mapsto& \tau_n(C).\\
\end{array}
$$
Si $E$ est une base atomique et sans boucles pour $C$, alors $E_{\leq n}:=\cup_{k\leq n} E_k$ est une base atomique et sans boucles pour $\tau_n(C)$.
On a enfin :
$$
\begin{array}{rcl}
\lambda \circ \tau_n &= &\tau_n \circ \lambda\\
\nu \circ \tau_n &= &\tau_n \circ \nu.
\end{array}
$$
\end{defi}

\note{\begin{rem}
Un morphisme $f:K\to L$ entre deux complexes dirigés augmentés (resp. un morphisme $f:C\to D$ entre deux $\omega$-catégories) est un isomorphisme si et seulement si $\tau_n(f)$ est un isomorphisme pour tout $n$.
\end{rem}
}

\begin{prop}
\label{prop:equivalence entre p q et r}
Soit $n$ un entier. Les assertions suivantes sont équivalentes:
\begin{enumerate}
\item Pour tout  $b\in (B_{M_0})_{\leq n}\coprod (B_{M_1})_{\leq n} $, $j(\langle b\rangle)$  est dans $\tilde{E}_n$;
\item La $\omega$-catégorie $\tau_n(N)$ est générée par composition par $E_{\leq n}$;
\item  L'ensemble $E_{\leq n}$ est une base atomique et sans boucles pour $\tau_n(N)$;
\item Le morphisme $\tau_n(\phi) : \tau_n(N) \to \tau_n(\nu M)$  est un isomorphisme.
\end{enumerate}
\end{prop}

\begin{proof}
On a directement $(3)\Rightarrow (2) \Rightarrow (1)$. L'implication $(4)\Rightarrow (3)$ provient du fait que le foncteur
$\phi$ induit pour tout $n$ une bijection entre $E_n$ et $\langle (B_M)_n\rangle$.

Supposons $(1)$. Les $\omega$-catégories $\tau_n(\nu M_0)$ et $\tau_n(\nu M_1)$ sont générées par composition par les ensembles $\langle (B_{M_0})_{\leq n}\rangle$ et $\langle (B_{M_1})_{\leq n}\rangle$. Par la construction de la somme amalgamée dans les $\omega$-catégories, $\tau_n(\nu N)$ est donc générée par composition par $j_0(\langle (B_{M_0})_{\leq n}\rangle)$ et $j_1(\langle (B_{M_1})_{\leq n}\rangle)$. L'assertion $(1)$ stipule que ces éléments sont eux-mêmes générés par  $E_{\leq n}$, qui génère donc $\tau_n( N)$.

Supposons maintenant $(2)$.   Le foncteur $\lambda$ est un adjoint à gauche et préserve donc les sommes amalgamées. Le morphisme $\tilde{\phi} :\lambda N\xrightarrow{\sim}M$  est alors une équivalence et induit, pour tout $k$, une bijection entre $[E_k]$ et $B_k$, qui est une base sans boucle.
\note{L'ensemble $E_{\leq n}$ est donc une base sans boucles pour $\tau_n(N)$.} Il reste à montrer qu'elle est  atomique.  Donnons nous donc $m \leq n$ et $j(\langle s_m(b)\rangle)\in E_m$. On peut supposer sans perte de généralité que  $ s_m(b)$ appartient à $(B_{M_0})_m$.
On a alors, pour tout $k<m$,
 
$$\def\arraystretch{1.4}
\begin{array}{rcl}
\tilde{\phi}( [d_k^-(j_0(\langle s_m(b)\rangle)]_k \wedge [d_k^+(j_0(\langle s_m(b)\rangle)]_k) &= &
\tilde{\phi}( [d_k^-(j_0(\langle s_m(b)\rangle)]_k)\wedge \tilde{\phi}( [d_k^+(j_0(\langle s_m(b)\rangle)]_k)\\
&= &
\tilde{\phi}\lambda j_0( [d_k^-(\langle s_m(b)\rangle)]_k)\wedge \tilde{\phi}\lambda j_0( [d_k^+(\langle s_m(b)\rangle)]_k)\\
 &=& l^0_{k}( \langle s_m(b)\rangle_k^-) \wedge l^0_{k} (\langle s_m(b)\rangle_k^+)
\end{array}
$$
Le morphisme $l^0$ étant quasi-rigide,
la proposition  \ref{prop:autre caracterisation de la quasi-rigidite} implique que $ l^0_{k} (\langle s_m(b)\rangle_k^- )\wedge l^0_{k} (\langle s_m(b)\rangle_k^+)=0$. 
On a alors $$\tilde{\phi}( [d_k^-(j_0(\langle s_m(b)\rangle)]_k \wedge [d_k^+(j_0(\langle s_m(b)\rangle)]_k)=0$$ d'où $$ [d_k^-(j_0(\langle s_m(b)\rangle)]_k \wedge [d_k^+(j_0(\langle s_m(b)\rangle)]_k=0.$$
L'ensemble $E_{\leq n}$ est donc une base atomique et sans boucles pour $\tau_n(N)$. 

Supposons enfin $(3)$. 
\note{ 
Le morphisme $\tau_n(\phi)$ est la composition des deux morphismes suivants:
$$\tau_n (N)\to \nu\lambda \tau_n(N)\xrightarrow{\nu \tau_n(\tilde{\phi})}\tau_n(\nu M).$$
Comme $\tau_n (N)$ admet une base atomique et sans boucles, le premier est un isomorphisme. De plus, on a remarqué plus haut que $\tau_n(\tilde{\phi})$ est un isomorphisme, et donc $\nu \tau_n(\tilde{\phi})$ l'est aussi.}
\end{proof}

\begin{lem}
\label{lem:lem thechnique}
Soit $n$ un entier. Soient $b,c$ deux éléments de $(B_{M_0})_{n+1} \coprod (B_{M_1})_{n+1}$.
\begin{enumerate}
\item  Si $l_{n+1}(b)=l_{n+1}(c)\neq 0$, alors  $j(\langle b\rangle)=j(\langle c \rangle)$;
\item Si $l_{n+1}(b) =0$, alors   $j(\langle b\rangle)$ est une unité.
\end{enumerate}
\end{lem}
\note{\begin{proof}
On se donne deux  éléments $b$ et $c$, et on suppose que $l_{n+1}(b)=l_{n+1}(c)\neq 0$. Le théorème \ref{theo:theo dans le cas preque quasi-rigide} et la remarque \ref{rem:sur les pusout dans cda} impliquent alors que $b\sim a$. Par définition de la relation $\sim$, il existe une famille finie $\{c_p\}_{p< m}$ d'éléments de $(B_K)_{n+1}$, qui ne sont envoyés sur $0$ ni par $k_0$ et ni par $k_1$, ainsi que $\alpha\in \{0,1\}$  tel que 
$$b = k_{n+1}^{\alpha}(c_0)~~~~~ k_{n+1}^{\alpha_{p+1}}(c_p)=k_{n+1}^{\alpha_{p+1}}(c_{p+1})~~~~~c = k_{n+1}^{\alpha_{m}}(c_{m-1}),$$
où $(\alpha_{p})_{p\leq m}$ est la suite qui alterne entre $0$ et $1$ en commençant par $\alpha$.
On a alors pour tout  $p< m-1$,
$$
j(\langle k_{n+1}^{\alpha_{p}}(c_p)\rangle) =j(\langle k_{n+1}^{\alpha_{p+1}}(c_p)\rangle)
= j(\langle k_{n+1}^{\alpha_{p+1}}(c_{p+1})\rangle)$$
et par suite,
$$j(\langle b\rangle)=  j(\langle k_{n+1}^{\alpha_{m-1}}(c_{m-1})\rangle)=j(\langle k_{n+1}^{\alpha_{m}}(c_{m-1})\rangle) =j(\langle c \rangle).$$

Supposons maintenant que  $l_{n+1}(b) =0$. Pour les mêmes raisons que plus haut, il existe une famille finie $\{c_p\}_{p< m}$ d'éléments de $(B_K)$ qui ne sont envoyés sur $0$ ni par $k_0$ et ni par $k_1$, ainsi qu'un entier $\alpha\in \{0,1\}$  tel que 
$$b = k_{n+1}^{\alpha}(c_0)~~~~~ k_{n+1}^{\alpha_{p+1}}(c_p)=k_{n+1}^{\alpha_{p+1}}(c_{p+1})~~~~~ k_{n+1}^{\alpha_{m}}(c_{m-1})=0,$$
où $(\alpha_{p})_{p\leq m}$ est la suite qui alterne entre $0$ et $1$ en commençant par $\alpha$, et on a encore une fois:
$$
j(\langle k_{n+1}^{\alpha_{p}}(c_p)\rangle) 
= j(\langle k_{n+1}^{\alpha_{p+1}}(c_{p+1})\rangle)$$
Le fait que $k_{n+1}^{\alpha_{m}}(c_{m-1})=0$ implique que $\nu k^{\alpha_{m}}(\langle c_{m-1}\rangle)$ est une identité, et donc 
$$j(\langle b\rangle)=  j(\langle k_{n+1}^{\alpha_{m-1}}(c_{m-1})\rangle)= j(\nu k^{\alpha_{m-1}}\langle c_{m-1}\rangle)=j(\nu k^{\alpha_{m}}\langle c_{m-1}\rangle)$$
l'est aussi.
\end{proof}}

\begin{proof}[Démonstration du théorème \ref{theo:theo dans le cas fort quasi-rigide}]
Montrons par récurrence sur $n$ que le morphisme $\tau_n(\phi)$ est un isomorphisme.

Pour l’initialisation, remarquons qu'on a des isomorphismes $M_0\cong E_0\cong (\nu M)_0$, et donc $\tau_0(\phi)$ est un isomorphisme.
Supposons maintenant la propriété vraie au rang $n$. Selon la proposition \ref{prop:equivalence entre p q et r}, il suffit de montrer que pour tout  $b\in (B_{M_0})_{n+1}\coprod (B_{M_1})_{n+1} $, $j(\langle b\rangle)$  est dans $\tilde{E}_n$.

Supposons tout d'abord que  $l_{n+1}(b)=0$. Le lemme \ref{lem:lem thechnique} indique que $j(\langle b\rangle)$ est une unité et est donc de la forme $1_x^{n+1}$ où $x$ est une $n$-cellule. L'hypothèse de récurrence implique alors que $x\in \tilde{E}_{n}$ et donc $j(\langle b\rangle)\in \tilde{E}_{n}\subset \tilde{E}_{n+1}$.

Supposons maintenant que $l_{n+1}(b)$ n'est pas nul. Comme $l^0$ et $l^1$ sont quasi-rigides, cela implique que $l_{n+1}(b)$ est un élément de $(B_M)_{n+1}$. On a donc, selon le lemme \ref{lem:lem thechnique}, l'égalité $j(\langle b\rangle) = j(\langle s_{n+1}l_{n+1}(b)\rangle)$ et donc $j_i(\langle b\rangle)\in \tilde{E}_{n+1}$, ce qui conclut la preuve.
\end{proof}

\subsection{Équations dans une $\omega$-catégorie}

\label{section:equation dans une omega cat}
Dans cette partie, on va formaliser une notion d'équation dans une $\omega$-catégorie. On se fixe pour la suite une $\omega$-catégorie $C$.
\begin{defi}
\label{def:in et partial in}
On définit $I_n$ comme la $\omega$-catégorie qui, pour tout $k<n$, possède deux  $k$-cellules n’étant pas des unités, notées $e_k^+,e_k^-$, une $n$-cellule qui n'est pas une unité, notée $e_n$, et dont toutes les cellules de dimension strictement supérieure à $n$ sont des unités.
Les applications sources et buts sont définies par : 
$d_l^\alpha e_k^\beta = e_l^\alpha$ pour  $l<k< n$ et $d_l^\alpha e_n = e_l^\alpha$  pour $l<n$ et $\alpha,\beta\in\{-,+\}$.
Il y a alors une bijection canonique entre les cellules de dimension $n$  de $C$ et les morphismes $I_n\to C$.

On définit aussi la $\omega$-catégorie $\partial I_n $ obtenue de $I_n$ en enlevant la cellule $e_n$. Posons alors 
$$i^- : I_{n-1}\simeq I_{n-1} \coprod\emptyset\to I_{n-1}\coprod_{\partial I_{n-1}}I_{n-1} 
~~~~~~~~~i^+ : I_{n-1}\simeq \emptyset\coprod I_{n-1} \to I_{n-1}\coprod_{\partial I_{n-1}}I_{n-1}$$
Il existe alors un unique isomorphisme 
$$
\phi :  I_{n-1}\coprod_{\partial I_{n-1}}I_{n-1} \xrightarrow{\sim} \partial I_n~~~~~
\mbox{tel que $\phi \circ i^\alpha (e_{n-1}) = e_{n-1}^\alpha$ pour $\alpha\in\{-,+\}$}$$
Par abus de langage la compostion $ \phi\circ i^\alpha : I_{n-1}\to \partial I_n$  pour $\alpha\in\{-,+\}$ sera aussi notée $i^\alpha.$

\end{defi}

\begin{defi}
Soient $P$ une $\omega$-catégorie et $a,b$ deux $(n-1)$-cellules parallèles de $P$. Ces données définissent un morphisme 
$$(a,b) : \partial I_n = I_{n-1}\coprod_{\partial I_{n-1}} I_{n-1}\xrightarrow{a\coprod b}  P.$$
La somme amalgamée suivante est notée $P[ _ax_b]$ :
$$
\xymatrix{
 \partial I_n \ar[r]^{(a,b)} \ar[d] & P \ar[d] \\
 I_n \ar[r]_{(x)}  &	 \pushoutcorner P[ _ax_b].
}$$
\end{defi}

\begin{defi}
\label{defi:equation}
Soient $P$ une $\omega$-catégorie, $a,b$ deux $(n-1)$-cellules parallèles de $P$, et $c,d$ deux $n$-cellules parallèles de $P[ _ax_b]$. On définit alors 
$$\textbf{Eq}_P(y : c(x)\to d(x)) := (P[ _ax_b])[ _cy_d].$$
On dit que cette $\omega$-catégorie est \textit{une équation }lorsqu'elle admet une base atomique et sans boucles. 
\end{defi}

\note{
\begin{defi} 
Soit $C$ une catégorie admettant une base atomique et sans boucles. On définit par induction la notion de \textit{décomposition d'une cellule $c\in C$ en éléments de la base}, notée $\dec(c)$:
\begin{enumerate}
\item Si $e$ est un élément de la base, $\{e\}$ est \textit{une décomposition en éléments de la base} de $a$, et pour tout $n>d(e)$, $\{(1^n_e)\}$ est une \textit{décomposition en éléments de la base} de $1^n_e$.
\item Si $a$ et $b$ sont deux cellules $k$-composables, $dec(a)$ et $dec(b)$ sont respectivement des décompositions en éléments de la base de $a$ et de $b$, le $5$-uplet $(b,k,a,\dec(b),\dec(a))$ est une \textit{décomposition en éléments de la base} de $b*_k a$.
\end{enumerate}
Comme $C$ est générée par composition par les éléments de la base, toute cellule admet une décomposition en éléments de la base.
Si $e$ est un élément de la base, on définit par induction \textit{l’occurrence de $e$ dans une décomposition en éléments de la base $dec(c)$}, noté $\lambda_e^{\dec(c)}$:
\begin{enumerate}
\item Si $\dec(c)=\{f\}$ où $f$ est un élément de la base, $\lambda_e^{\dec(c)}:=1$ si $e=f$, et $\lambda_e^{\dec(c)}:=0$ sinon.
\item Si $\dec(c)=\{1^n_f\}$ où $f$ est un élément de la base et $n>d(f)$, $\lambda_e^{\dec(c)}:=0$. 
\item Si $\dec(c)=(a,k,b,\dec(a),\dec(b))$, $\lambda_e^{\dec(c)}:=\lambda_e^{\dec(a)}+ \lambda_e^{\dec(b)}$.
\end{enumerate}
L'élément $e$ \textit{n’apparaît pas dans une décomposition en éléments de la base} $\dec(c)$ de $c$ si $\lambda_e^{\dec(c)}=0$.
\end{defi}

\begin{rem}
\label{rem:sur l'occurence dans une décomposition.}
Soient $n$ un entier et $C$ une $\omega$-catégorie admettant une base atomique et sans boucles. Selon le corollaire \ref{cor:ajdonction entre chaine et lambda avec unite et counite explicite}, on a un isomorphisme
$$\eta:C\to \mu \lambda C.$$
 Soient $e$ un élément de la base et $c$ une cellule telle que 
$d(e)=d(c)=n$. Pour toute décomposition de $c$ en éléments de la base $\dec(c)$,  on peut déduire de la remarque 
\ref{rem:rem sur la compo explicite des chaines} que 
$\lambda^{\dec(c)}_e =\lambda^{\eta(c)}_{[e]_n}.$  
\end{rem}
}

\begin{prop}
Soit $\textbf{Eq}_P(y : c(x)\to d(x))$ une équation et $E$ une base de cette $\omega$-catégorie. Alors \note{$x$ et $y$ sont des éléments de la base}, et 
\begin{enumerate}
\item \note{Pour toute décomposition $\dec(c)$ et $\dec(d)$  de $c$ et $d$ en éléments de la base, $\lambda_x^{\dec(c)}\leq 1$ et $\lambda_x^{\dec(d)}\leq 1$.}
\item La cellule $x$ ne peut pas apparaître à la fois dans les décompositions de $c$ et dans celles de $d$.
\end{enumerate}
\end{prop}
\begin{proof}
\note{Par la définition des sommes amalgamées dans les $\omega$-catégories, $x$ et $y$ ne peuvent ni s'exprimer comme  compositions de deux cellules qui ne sont pas toutes deux des unités, ni comme unités d'autres cellules. Ainsi, $x$ et $y$ sont compris dans tout ensemble qui génère par composition $\textbf{Eq}_P(y : c(x)\to d(x))$.   En particulier, $x$ et $y$ appartiennent à $E$.

Soient $\dec(c)$ et $\dec(d)$  des décompositions de $c$ et $d$ en éléments de la base. 
Selon le corollaire \ref{cor:ajdonction entre chaine et lambda avec unite et counite explicite}, on a un isomorphisme
$$\eta: \textbf{Eq}_P(y : c(x)\to d(x))\to \mu\lambda \textbf{Eq}_P(y : c(x)\to d(x))$$
qui envoie les $n$-cellules $b$ sur  des chaînes cohérentes $\eta(b)$. 
En utilisant la remarque \ref{rem:sur l'occurence dans une décomposition.} et la proposition \ref{prop:pas de double}, on en déduit:
$$\lambda_x^{\dec(c)} = \lambda_{[x]_n}^{\eta(c)}\leq 1~~~~~~~~ \lambda_x^{\dec(d)} = \lambda_{[x]_n}^{\eta(d)}\leq 1.$$

Interressons nous maintenant à la deuxième assertion. La cellule $y$ étant un élément d'une base atomique, on a 
$$[c]_n\wedge [d]_n =[d^+_n y]_n\wedge [d^-_n y]_n=0$$ d'où 
$$\lambda_x^{\dec(c)} =\lambda_{[x]_n}^{\eta(c)}=\lambda_{[x]_n}^{[c]_n}=0 \mbox{ ~~ou~~ }
\lambda_x^{\dec(d)} =\lambda_{[x]_n}^{\eta(d)}=\lambda_{[x]_n}^{[d]_n}=0,$$ 
ce qui conclut la preuve.}
\end{proof}

\begin{defi}
Une\textit{ équation à paramètres dans $C$} est la donnée d'une equation $\textbf{Eq}_P(y : c(x)\to d(x))$ ainsi que d'un diagramme:
$$
\xymatrix{
P \ar[r]^p \ar[d] & C \\
\textbf{Eq}_P(y : c(x)\to d(x)) & .
}$$

Une \textit{pré-solution}  de  l'équation $\textbf{Eq}_P(y : c(x)\to d(x)) $ avec paramètre $p$ dans $C$ est un relèvement $$l:\textbf{Eq}_P(y : c(x)\to d(x)) \to C$$ faisant commuter le triangle induit. Une présolution est \textit{une solution} lorsque $y$ est envoyé sur une cellule faiblement inversible. 

On dit que l’équation $\textbf{Eq}_P(y : c(x)\to d(x)) $\textit{ admet toujours des pré-solutions dans $C$ }lorsqu'il existe une pré-solution pour tout choix de paramètre $p: P\to C$.

On dit que l’équation $\textbf{Eq}_P(y : c(x)\to d(x)) $ \textit{admet toujours des solutions dans $C$} lorsqu'il existe une solution pour tout choix de paramètre $p: P\to C$.
\end{defi}

\begin{rem}
La terminologie provient du fait qu'une équation étant définie par des sommes amalgamées, trouver une solution consiste à exhiber des cellules $x$ et $y$ vérifiant les conditions voulues. Il y a donc une vraie analogie avec la notion habituelle d'équation.
\end{rem}

\begin{defi}
Par abus de langage on note encore
$i^\alpha$ la composition $ I_{n-1}\overset{i^\alpha}{\to}\partial I_n\to I_n$ pour $\alpha\in\{-,+\}$.
On définit $P^+$ comme étant la somme amalgamée: 
$$
\xymatrix{
I_{n-1} \ar[r]^{i^+} \ar[d]_{i^+} & I_n \ar[d]^{i_1} \\
I_{n} \ar[r]_{i_2} & P^+ \pushoutcorner.
}$$

On 	note $e_k^\alpha$ les  cellules de $P^+$ dans l'image de $i_1$ et $f_k^\alpha$ celles dans l'image de $i_2$. On a alors $f_{n-1}^+ = e_{n-1}^+$ et $ f_k^\alpha = e_k^\alpha  \mbox{ pour $k < n-1$ et $\alpha\in\{-,+\}$} $ et on définit:
$$
\begin{array}{rcllll}

~~~~~~~~~~~~~~~~~~~~~~~~~~~~~~~~~~~
~~~~~~~~~~~~~~~~~~a_{n-1} &:=& f_{n-1}^-&&&\\
b_{n-1} &:=& e_{n-1}^-&&&\\
c_{n-1} &:=& f_{n-1}^+ &=& e_{n-1}^+&\\
i_k^\alpha &:=& f_k^\alpha &=& e_k^\alpha& \mbox{pour $k < n-1$ et $\alpha\in\{-,+\}$}   
\end{array}
$$
$$
\mbox{pour $k < n-1$ et $\alpha\in\{-,+\}$} 
$$
On définit l'équation $ \textbf{Eq}( y : f_n \to e_n*_{n-1} x )$ de paramètre $P^+$:

$$
\xymatrix{  	 
\partial I_n \ar[r]^{(f_{n-1}^-,e_{n-1}^-)}	 \ar[d]&P^+ \ar[d] \\	 
I_n \ar[r] & P^+[_{f_{n-1}^-}x_{e_{n-1}^-}] \pushoutcorner\\
}
~~~~~~~~~~
\xymatrix{  	 
\partial I_n \ar[r]^{(f_n, e_n*_{n-1} x )~~~~~~}\ar[d] & P^+[_{f_{n-1}^-}x_{e_{n-1}^-}]\ar[d]\\
I_n \ar[r] & \textbf{Eq}( y : f_n \to e_n*_{n-1} x ). \pushoutcorner
}$$
De façon symétrique, on définit l'équation :
$$P^- \to \textbf{Eq}( y : f_n \to  x *_{n-1} e_n).$$

\end{defi}

\begin{rem}
\label{rem:def explicite de l'equation}
On peut expliciter les cellules de ces $\omega$-catégories qui ne sont pas des unités.
$$\begin{array}{lclc}
\textbf{Eq}( y : f_n \to e_n*_{n-1} x )_{n+1} &=& \{y\}& \\
\textbf{Eq}( y : f_n \to e_n*_{n-1} x )_{n} &=& \{f_n,e_n,x, e_n*_{n-1}x\}&\\
\textbf{Eq}( y : f_n \to e_n*_{n-1} x )_{n-1} &=& \{a_{n-1},b_{n-1},c_{n-1}\}&\\
\textbf{Eq}( y : f_n \to e_n*_{n-1} x )_{l} &=& \{i_l^-,i_l^+\} & \mbox{pour $l<n-1$}
\end{array}$$
Et les applications sources et buts sont définies par:
$$\begin{array}{lclclclc}
d_n^-(y)&=& f_n & ~~~~~~~~~~~&d_n^+(y)&=&  e_n*_{n-1} x &\\
d_{n-1}^-(f_n)&=& a_{n-1} & ~~~~~~~~~~~&d_{n-1}^+(f_n)&=&  c_{n-1}& \\
d_{n-1}^-(e_n)&=& b_{n-1} & ~~~~~~~~~~~&d_{n-1}^+(e_n)&=&  c_{n-1} &\\
d_{n-1}^-(x)&=& a_{n-1} & ~~~~~~~~~~~&d_{n-1}^+(x_n)&=&  b_{n-1} &\\
d_{n-2}^-(a_{n-1})&=& i_{n-2}^- & ~~~~~~~~~~~&d_{n-2}^+(a_{n-1})&=&  i_{n-2}^+ &\\
d_{n-2}^-(b_{n-1})&=& i_{n-2}^- & ~~~~~~~~~~~&d_{n-2}^+(b_{n-1})&=&  i_{n-2}^+ &\\
d_{n-2}^-(c_{n-1})&=& i_{n-2}^- & ~~~~~~~~~~~&d_{n-2}^+(c_{n-1})&=&  i_{n-2}^+& \\
d_{l-1}^-(i_{l}^\alpha)&=& i_{l-1}^- & ~~~~~~~~~~~&d_{l-1}^+(i_{l}^\alpha)&=&  i_{l-1}^+ &  \\
\end{array}
$$
$~~~~~~~~~~~~~~~~~~\mbox{pour $l<n-1$ et $\alpha\in\{-,+\}$}$

\note{On vérifie aisément que $\{y,f_n,e_n,a_{n-1},b_{n-1},c_{n-1}\}\cup \{i_l^-,i_l^+,~l<n-1\}$
est une base sans boucles et atomique pour $\textbf{Eq}( y : f_n \to e_n*_{n-1} x )$.}
\end{rem}

\begin{prop}
\label{prop:caracterisation1}
Une $\omega$-catégorie $C$ est $k$-triviale si et seulement si les équations ${\textbf{Eq}(y : f_n \to e_n*_{n-1} x)} $ admettent toujours des pré-solutions  pour $n>k$. Ces pré-solutions sont alors des solutions. 
\end{prop}
\begin{proof}
\note{
Supposons donc que les équations $\textbf{Eq}(y : f_n \to e_n*_{n-1} x )$  admettent toujours des pré-solutions  pour $n> k$ et donnons nous une $n$-cellule $a$ avec $n> k$. On va construire par récurrence une suite d'ensembles $\{E^m\}_{m\in\mathbb{N}}$ telle que  $E:=\cup_{m\in\mathbb{N}}E^m$ soit un ensemble d'inversibilité comprenant $a$.

On définit donc $E^0 :=\{a\}$. Supposons $E^m$ construit. Soit $b$ une $(n+m)$ cellule de $E^m$.
Comme les équations ${\textbf{Eq}(y : f_{n+m}\to e_{n+m}*_{n+m-1} x)} $ admettent des pré-solutions pour tout choix de paramètre  $p : P^+\to C$, il existe des éléments $x,x',y,y',y'',z$  vérifiant :
$$
\begin{array}{crccl}
y &:& 1_{d^-_{n+m-1} b}  &\to & x*_{n+m-1}b \\
y' &: &1_{d^+_{n+m-1} b} &\to   &  x'*_{n+m-1}x \\
z&:&1_{d^-_{n+m} y'} &\to&     y''*_{n+m}y'
\end{array}
$$
On définit enfin la cellule $\tilde{y}$ comme la  $(n+m)$-composition:
$$
\xymatrix{
1_{d^+_{n+m-1} b} \ar@{.>}[rrrrr]^{\tilde{y}} \ar[rd]_{y'} &&&&&b*_{n+m-1} x\\
&x'*_{n+m-1}x \ar[rrr]_{x' *_{n+m-1} y*_{n+m-1} x~~~~~~~~~~~~~} &&& x'*_{n+m-1}x*_{n+m-1} b *_{n+m-1} x \ar[ru]_{~~~~~~~~y''*_{n+m-1}b *_{n+m-1} x} 
}
$$

On pose alors $(E^m)_b := \{x,y,\tilde{y}\}$ et on définit $E^{m+1}$ comme la réunion de $E^m$ et des ensembles $(E^m)_b$ pour toute $(n+m)$-cellule $b\in E^m$.

Enfin, on définit $E:= \cup_{m\in \mathbb{N}} E^m$.
Cet ensemble vérifie bien les conditions voulues. Toutes les cellules de $C$ de dimension strictement supérieure à $k$ sont donc faiblement inversibles, c'est-à-dire que $C$ est $k$-triviale. 

Réciproquement, supposons que $C$ est $k$-triviale. Le point (1) de la proposition \ref{prop:propdedivision a droite} implique directement le résultat. }
\end{proof}

\section{Nerf de Street}

\subsection{Nerf d'une $\omega$-Catégorie}

Dans la partie précédente, on a construit un foncteur $\nu: \CDA \to \infcat$ et  un foncteur $\mu : \CDAB\to \infcat$ tel que $\nu_{|\CDAB}\cong\mu$. On va s'en servir pour construire une adjonction entre la catégorie des ensembles simpliciaux et la catégorie des $\omega$-catégories. 

\note{\begin{notation}
Lorsque $C$ est une $\omega$-catégorie admettant une base atomique et sans boucles, et $e$ un élément de la base, on notera aussi $e$ l'élément $[e]_{d(e)}\in (\lambda C)_{d(e)}.$
\end{notation}}
\begin{defi}
Soit $X$ un ensemble simplicial. On note $C_\bullet(X)$ le complexe de chaînes réduit associé à $X$:
 $$C_n(X) :=\mathbb{Z} \{v\in X_n : v \mbox{ non dégénéré }\}$$
$$\begin{array}{rccl}
\partial_{n+1}  :& C_{n+1}(X)&\to &C_n(X) \\
&v&\mapsto& \sum (-1)^i d_i v
\end{array}
$$
où par convention dans cette somme $d_iv = 0$ si $d_iv$ est un simplexe
dégénéré.

On définit aussi, pour tout $n$, le monoïde additif $C_n^*(X)$ engendré par les $n$-simplexes non dégénérés, et une augmentation $e:C_0(X)\to \mathbb{Z}$ qui envoie les $0$-simplexes sur $1$. Le triplet $(C_n(X),C_n^*(X),e)$ est alors un complexe dirigé augmenté. De plus, ce complexe admet une base, donnée par l'ensemble des simplexes non dégénérés de $X$.

\note{ Dans \cite[Exemple 3.8]{steiner}, Steiner montre que cette base est unitaire et sans boucles.
}
\end{defi}

On définit  alors l'objet cosimplicial suivant dans $\infcat$: 
$$[n]\to \mu C_\bullet (\Delta[n]).$$
 et cela nous permet alors de définir une adjonction
$$|\var| : \xymatrix{ \textbf{Sset}  \ar@/^8pt/[r]& \ar@/^8pt/[l] \infcat} :  \mathcal{N}. $$

\subsection{Résolution d'équations et relèvements}
Dans cette partie, on va montrer comment on peut déduire qu'une $\omega$-catégorie est $0$ ou $1$-triviale grâce aux propriétés de relèvement de son nerf. L’objectif est de montrer le corollaire  \ref{cor : trivialite detecte par le relevement}. On fixe pour la suite une  $\omega$-catégorie $C$. On utilisera ici les $\omega$-catégories $I_n$ et $\partial I_n$ que l'on a définies en \ref{def:in et partial in}. Dans cette partie, pour un morphisme $f : X\to \N(C)$, on notera aussi $f : |X|\to C$ le morphisme obtenu par adjonction.

\begin{prop}
L'application suivante est surjective:
$$
\begin{array}{rcl}
 \N(C)_n = Hom (|\Delta_n|,C ) &\to & C_n \\
 f &\mapsto & f(\langle i_n\rangle)
 \end{array}
 $$ 
 où $i_n$ est l'unique simplexe non dégénéré de dimension $n$ de $\Delta[n]$, vu comme un élément de $C_n(\Delta[n])$.
 \label{prop:surjectivite du nerf}
\end{prop}

Pour démontrer cette proposition, on a besoin du lemme suivant.
\begin{lem}
\label{lem:pour la prop surjectivite du nerf}
Il existe un morphisme de complexes dirigés augmentés
$$ p: C_\bullet(\Delta[n])\to \lambda(I_n)$$
tel que $p_n(i_n) = e_n$
\end{lem}
\begin{proof}
Pour un entier $l\leq n$ et une famille d'entiers décroissante $\{k_i\}_{i\leq l}$, on définit :
 $d^{k_1,k_2,...,k_l} = d_{k_1}...d_{k_l} i_n$. Tous les simplexes non dégénérés de $\Delta[n]$ peuvent s’écrire sous cette forme d'une unique façon. On définit alors
$$\begin{array}{rccl}
p : &C_\bullet(\Delta[n])& \to& \lambda( I_n) \\
&i_n &\mapsto & e_n\\
&d^{k_1,k_2,...,k_l} &\mapsto &  
\left\{\begin{matrix}
  e_{n-l}^+& \mbox{si $k_1=0$} \\
 e_{n-l}^-& \mbox{si $k_1=1$} \\
 0 & \mbox{ si $k_1 >1$}\\
\end{matrix}\right.
 \\
\end{array}$$ 
Il suffit maintenant de vérifier que cela définit bien un morphisme de complexes dirigés augmentés. La seule vérification non triviale est la compatibilité avec les différentielles. Il faut donc montrer que pour tout $m$-simplexe non dégénéré $x$, on a $p_{m-1}(\partial_{m-1}(x))=\partial_{m-1}(p_m(x))$. Si $x=i_n$, on a 
$$ p_{n-1}(\partial_{n-1}(i_n))= p_{n-1} (\sum_{0\leq i\leq n} (-1)^id_ii_n) =e_{n-1}^+ - e_{n-1}^- = \partial_{n-1} (e_n)=\partial_{n-1} (p_n(i_n)).$$
Supposons maintenant que $x$ soit sous la forme $x = d^{k_1,k_2,...,k_l}$ avec $l\leq n$. La règle simpliciale 
$$ d_id_j = d_{j-1}d_i ~,~ i<j$$
implique que 
$$d_i~d^{k_1,...,k_l} = d^{k_1',k_2',...,k_{l+1}'}$$
avec 
$$k_1' = \left\{ 
\begin{matrix}
i &\mbox{si $i\geq k_1$}\\
k_1-1 &\mbox{si $i<k_1$}\\
\end{matrix}
\right.
$$
et par suite
$$
p_{n-l-1}(d_i~d^{k_1,...,k_l}) = 
\left\{
\begin{matrix}
0 & \mbox{si} & 
\left\{
\mbox{ou~~}
\begin{matrix}
k_1\geq 3 \\
k_1\leq 2 & \mbox{et}& i\geq 2
\end{matrix}
\right.\\

e_{n-l-1}^- & \mbox{si} & 
\left\{
\mbox{ou~~}
\begin{matrix}
k_1= 2 & \mbox{et}& i\leq 1 \\
k_1\leq 1 & \mbox{et}& i=1 
\end{matrix}
\right.\\

e_{n-l-1}^+ & \mbox{si} &~~ k_1\leq 1 \mbox{~~et~~} i=0

\end{matrix}
\right.
$$

On en déduit que si $k_1=0$ ou $k_1 = 1$, on a 
$$p_{n-l-1}(\partial_{n-l-1}(d^{k_1,..,k_l})) = p_{n-l-1}(\sum_{i\leq n-l} (-1)^i d_i~d^{k_1,..,k_l})=e^+_{n-l-1}-e^-_{n-l-1} = \partial_{n-l-1}(p_{n-l}(d^{k_1,..,k_l})),$$
si $k_1=2$, on a 
$$p_{n-l-1}(\partial_{n-l-1}(d^{k_1,..,k_l})) = p_{n-l-1}(\sum_{i\leq n-l} (-1)^i d_i~d^{k_1,..,k_l})=e^-_{n-l-1}-e^-_{n-l-1} =0= \partial_{n-l-1}(p_{n-l}(d^{k_1,..,k_l})),$$
et si $k_1\geq 3$, on a 
$$p_{n-l-1}(\partial_{n-l-1}(d^{k_1,..,k_l})) = p_{n-l-1}(\sum_{i\leq n-l} (-1)^i d_i~d^{k_1,..,k_l}) =0= \partial_{n-l-1}(p_{n-l}(d^{k_1,..,k_l})).$$
\end{proof}
\begin{proof}[Démonstration de la proposition \ref{prop:surjectivite du nerf}]
C'est une conséquence directe du lemme précédent. En effet,  une cellule  $c : a\to b$  de dimension $n$ correspond à un morphisme  $f : I_n\to C$. 
Le théorème \ref{theo} implique que  $\mu \lambda(I_n)\cong I_n$. On définit le morphisme composé suivant:
$$\lfloor c \rfloor :  |\Delta[n]| \xrightarrow{\mu(p)} \mu\lambda (I_n)\cong I_n \xrightarrow{f} C$$
et on a bien $\lfloor c \rfloor (\langle i_n\rangle) = c$.
\end{proof}

\begin{prop}
Si l'ensemble simplicial $\N(C)$  a la propriété de relèvement par rapport aux inclusions $\Lambda^2[n+1]\to \Delta[n+1]$, alors l'équation $\textbf{Eq}( y :  f_n \to e_n*_{n-1}x)$  admet toujours une pré-solution dans $ C$.
\label{prop:complexe de Kan implique inversible}
\end{prop}
Pour prouver cette proposition,  on a besoin de quelques lemmes. On pose pour la suite $P:= P^+$.

\begin{lem}
Il existe un carré commutatif:
$$\xymatrix{
C_\bullet(\Lambda^2[n+1])\ar[r]^{q'} \ar[d] & \lambda(P) \ar[d] \\ 
C_\bullet(\Delta[n+1]) \ar[r]_{q~~~~~~~~~}& \lambda(\textbf{Eq}( y : f_n \to e_n*_{n-1} x )) &
}$$
tel que $q_{n+1}(i_{n+1}) = y$ et $q_n(d^2) = x$.
\label{lem:carre commutatif}
\end{lem}
\begin{proof}
On va réutiliser les notations de la preuve précédente. Pour $l\leq n$ et $\{k_i\}_{i\leq l}$ une famille décroissante d'entiers \note{positifs}, on définit $d^{k_1,k_2,...,k_l} = d_{k_1}...d_{k_l} i_{n+1}$. Rappelons que tous les simplexes non dégénérés de $\Delta[{n+1}]$ peuvent s’écrire sous cette forme de façon unique. On a donné dans la remarque \ref{rem:def explicite de l'equation}  la description explicite des cellules de $\textbf{Eq}( y : f_n \to e_n*_{n-1} x )$.
On définit  alors

$$\begin{array}{rccl}
q : &C_\bullet(\Delta[{n+1}])& \to& \lambda(\textbf{Eq}( y : f_n \to e_n*_{n-1} x ))  \\
&i_{n+1} &\mapsto & y\\
&d^{k_1}&\mapsto &  
\left\{\begin{array}{cl}
 e_{n}& \mbox{si $k_1=0$} \\
 f_{n}& \mbox{si $k_1=1$} \\
 x& \mbox{si $k_1 =2$ }\\
  0& \mbox{sinon} \\
\end{array}\right.\\

&d^{k_1,k_2}&\mapsto & 
\left\{\begin{array}{cl}
 c_{n-1}& \mbox{si $(k_1,k_2)=(0,0)$} \\
 b_{n-1}& \mbox{si $(k_1,k_2)=(1,0)$} \\
 a_{n-1}& \mbox{si $(k_1,k_2)=(1,1)$} \\
  0& \mbox{sinon} \\
\end{array}\right.\\

(l>2)&d^{k_1,k_2,...,k_l}&\mapsto & 
\left\{\begin{array}{cl}
 i_{n-l+1}^+& \mbox{si $k_1=0$} \\
 i_{n-l+1}^-& \mbox{si $k_1=1$} \\
  0& \mbox{sinon} \\
\end{array}\right.\\

\end{array}$$

Il faut maintenant vérifier que cette application est compatible avec les différentielles.

$$
\def\arraystretch{1.4}\begin{array}{lclc}
q_n(\partial_n i_{n+1}) &=& q_n(\sum_{i\leq n+1 } (-1)^i d_i i_{n+1})= e_n + x -f_n = \partial_n(y)=\partial_n(q_{n+1}(i_{n+1}))\\
q_{n-1}(\partial_{n-1} d^0 )&=& q_{n-1}(\sum_{i\leq n} (-1)^i d^{i,0}) =  c_{n-1}- b_{n-1} =  \partial_{n-1}(e_n)=\partial_n(q_{n}(d^0))\\
q_{n-1}(\partial_{n-1} d^1 )&=& q_{n-1}(d^{0,0}+\sum_{0<i\leq n} (-1)^i d^{i,1}) =  c_{n-1}- a_{n-1} =  \partial_{n-1}(f_n)=\partial_n(q_{n}(d^1))\\
q_{n-1}(\partial_{n-1} d^2 )&=& q_{n-1}(d^{1,0} - d^{1,1}+\sum_{1<i\leq n} (-1)^i d^{i,2}) =  b_{n-1}- a_{n-1} =  \partial_{n-1}(x)=\partial_n(q_{n}(d^2))\\
q_{n-1}(\partial_{n-1} d^k )&=& 0 =\partial_n(q_{n}(d^k))~~~~\mbox{pour $k>2$}\\
\end{array}$$
Pour les simplexes de dimension strictement inférieure à $n$, la vérification est identique à celle présente dans la preuve du lemme \ref{lem:pour la prop surjectivite du nerf}.
On remarque  de plus que la restriction de $q$ à $\Lambda^2[n+1]$ se factorise par $P$. On définit donc $q' := q_{|\Lambda^2[n+1]}$. Cela conclut la preuve.
\end{proof}

\note{

\begin{lem}
Soit $a$ une chaîne cohérente de  $\lambda \textbf{Eq}( y : f_n \to e_n*_{n-1} x )$. On a alors $(a)_{|a|} = a$.
\end{lem}
\begin{proof}
Donnons nous une chaîne cohérente $a$.
Selon le corollaire \ref{cor:ajdonction entre chaine et lambda avec unite et counite explicite}, on a un isomorphisme
$$\eta:\textbf{Eq}( y : f_n \to e_n*_{n-1} x )\to \mu \lambda \textbf{Eq}( y : f_n \to e_n*_{n-1} x ).$$
Il existe donc une unique cellule $c$ telle que $\eta(c)=a$. Si $c$ est un élément de la base, la chaîne $\eta(c)$ est alors réduite à un élément, et vérifie donc la propriété voulue. La description explicite des cellules de cette $\omega$-catégorie donnée dans la remarque \ref{rem:def explicite de l'equation} montre que l'unique cellule qui n'est ni une unité, ni un élément de la base est $e_n*_{n-1} x$. Il suffit donc de vérifier que $\eta(e_n*_{n-1} x)$ vérifie la propriété. Or on a:
$$\def\arraystretch{1.4}\begin{array}{rcl}
\eta(e_n*_{n-1}x)&:=&[e_n*_{n-1}x]_n + \sum_{k<n} ([d_k^+ (e_n*_{n-1}x)]_k - \partial_k^+([(d_{k+1}^+ e_n*_{n-1}x)]_{k+1}))\\
&=&e_n  +x +c_{n-1} - c_{n-1} + \sum_{k<n-1}  i_k^+-i_k^+  \\
&=& e_n + x\\
\end{array}$$
Et donc $$(\eta(e_n*_{n-1}x))_n = (e_n + x)_n = e_n + x =\eta(e_n*_{n-1}x).$$
\end{proof}

}

\begin{lem}
\label{lem:quasi-rigidite de q}
Les morphismes $q$ et $q'$ sont quasi-rigides. 
\end{lem}
\begin{proof}
On va montrer que pour tout élément $b$ de la base de $C_\bullet(\Delta[n+1])$, 
$$q_k(b)\neq 0 \Rightarrow \mu(q)(b) = q_k( b).$$
\note{ 
Soit $b$ un élément de la base tel que $q_k(b)\neq 0$. La chaîne $\mu(q)(b)$ vérifie 
$((\mu(q)(b))_{|b|} = q_k( b) $. Or cette chaîne est cohérente, et en appliquant le lemme précédent, on obtient $\mu(q)(b) = q_k( b)$.}
\end{proof}

\begin{lem}
Dans la catégorie des $\omega$-catégories, il existe une somme amalgamée: 
$$\xymatrix{
|\Lambda^2[n+1]|\ar[r]^{\tilde{q'}} \ar[d] & P \ar[d] \\ 
|\Delta[n+1]| \ar[r]_{\tilde{q}~~~~~~~~~}& \pushoutcorner \textbf{Eq}( y : f_n \to e_n*_{n-1} x )&
}$$
telle que $\tilde{q}_{n+1}(i_{n+1}) = y$ et $\tilde{q}_n(d^2) = x$.
\label{lem:somme amalgame}
\end{lem}
\begin{proof}
Notons $B_{\Lambda^2[n+1]}$, $B_{\Delta[n+1]}$, $B_P$ et $B_{\textbf{Eq}}$ les bases de $C_\bullet(\Lambda^2[n+1])$, $C_\bullet(\Delta[n+1])$, $\lambda(P)$ et $ \lambda(\textbf{Eq}( y : f_n \to e_n*_{n-1} x ))$. Selon le lemme
\ref{lem:carre commutatif}, il existe un diagramme commutatif :
\begin{equation}
\label{comdiag}
\xymatrix{
C_\bullet(\Lambda^2[n+1])\ar[r]^{q'} \ar[d] & \lambda(P) \ar[d] \\ 
C_\bullet(\Delta[n+1]) \ar[r]_{q~~~~~~~~~}& \lambda(\textbf{Eq}( y : f_n \to e_n*_{n-1} x )). &
}
\end{equation}
Ce diagramme induit une carré commutatif d'ensembles:
$$\xymatrix{
B_{\Lambda^2[n+1]}\cup \{0\}\ar[r]^{q'} \ar[d] & B_P\cup \{0\} \ar[d] \\ 
B_{\Delta[n+1]}\cup \{0\}\ar[r]_{q}& \pushoutcorner B_{\textbf{Eq}}\cup \{0\}&
}.$$
\note{Remarquons de plus que $$B_{\Delta[n+1]}\cup \{0\}:=B_{\Lambda^2[n+1]}\cup \{0\}\cup \{i_{n+1},d^2\}~~~~~B_{\textbf{Eq}}\cup \{0\}:=B_{P}\cup \{0\}\cup \{y,x\}.$$ Le morphisme $q$ envoie $i_{n+1}$ sur $y$, et $d^2$ sur $x$. Le carré précédent est donc une somme amalgamée dans la catégorie des ensembles.
}
Le théorème \ref{theo:theo dans le cas preque quasi-rigide} implique alors que le diagramme \eqref{comdiag} est une somme amalgamée dans \CDA . De plus,  selon la remarque \ref{rem:injection sur les bases implique quasi-ridige},
les morphismes verticaux sont  quasi-rigides car il induisent des injections sur les bases, et les morphismes horizontaux le sont aussi selon le lemme \ref{lem:quasi-rigidite de q}.

 Le corollaire \ref{theo:theo dans le cas fort quasi-rigide} implique alors que le diagramme suivant est une somme amalgamée: 
$$\xymatrix{
|\Lambda^2[n+1]|\ar[r]^{\nu (q')} \ar[d] & \nu \lambda P \ar[d] \\ 
|\Delta[n+1]| \ar[r]_{\nu(q)~~~~~~~~~}& \pushoutcorner\nu \lambda \textbf{Eq}( y : f_n \to e_n*_{n-1} x ).
}$$
Enfin, comme $P$ et $\textbf{Eq}( y : f_n \to e_n*_{n-1} x )$ sont des $\omega$-catégories admettant des bases sans boucles et atomiques, on a des isomorphismes : 
$$\nu \lambda P \cong P ~~~~~~~~~ \nu \lambda \textbf{Eq}( y : f_n \to e_n*_{n-1} x ) \cong \textbf{Eq}( y : f_n \to e_n*_{n-1} x )$$
\end{proof}
\begin{proof}[Démonstration de la proposition \ref{prop:complexe de Kan implique inversible}]
Soit $C$ une $\omega$-catégorie telle que $\N(C)$  ait la propriété de relèvement par rapport aux inclusions $\Lambda^2[n+1]\to \Delta[n+1]$.  Notons $\mathcal{L}$ la classe des morphismes ayant la propriété de relèvement à gauche par rapport à $C\to 1$. Par adjonction, $\mathcal{L}$ comprend $|\Lambda^2[n+1]|\to |\Delta[n+1]|$. De plus, $\mathcal{L}$ est stable par image directe et comprend donc, selon le lemme \ref{lem:somme amalgame}, le morphisme $P\to \textbf{Eq}( y : f_n \to e_n*_{n-1} x )$. Cela conclut la preuve.
\end{proof}

\begin{cor}
\label{cor : trivialite detecte par le relevement}
Soit $C$  une $\omega$-catégorie. 
\begin{enumerate}
	\item Si $\N(C)$  a la propriété de relèvement par rapport aux inclusions $\Lambda^2[n+1]\to \Delta[n+1]$ pour tout $n>0$, alors $C$ est $0$-trivial;
	\item Si $\N(C)$  a la propriété de relèvement par rapport aux inclusions $\Lambda^2[n+1]\to \Delta[n+1]$  pour tout $n>1$, alors $C$ est $1$-trivial;\end{enumerate}
\end{cor}
\begin{proof}
Supposons tout d'abord que pour tout $n>0$, $\N(C)$ a la propriété de relèvement par rapport aux inclusions $\Lambda^2[n+1]\to \Delta[n+1]$. Alors la proposition \ref{prop:complexe de Kan implique inversible} induit que pour tout $n>0$, les équations $\textbf{Eq}(y : f_n \to e_n*_{n-1} x) $ admettent toujours des pré-solutions dans $C$, et donc selon la proposition \ref{prop:caracterisation1}, $C$ est $0$-trivial. 

De façon analogue, si pour tout $n>1$, $\N(C)$ a la propriété de relèvement par rapport aux inclusions $\Lambda^2[n+1]\to \Delta[n+1]$  alors  les équations $\textbf{Eq}(y : f_n \to e_n*_{n-1} x )$ admettent toujours des pré-solutions.	 La proposition \ref{prop:caracterisation1} indique que  $C$ est alors $1$-trivial.
 \end{proof}

\subsection{Équations représentées par les inclusions de cornets}
Soit $n$ un entier quelconque. On sait déjà que $|\Delta[n]|$ est une $\omega$-catégorie admettant une base sans boucles et atomique. Le morphisme 
$|\Lambda^r[n]| \to |\Delta[n]|$ est obtenu en ajoutant librement une $(n-1)$-cellule, et une $n$-cellule. La $\omega$-catégorie $|\Delta[n]|$ est donc bien une équation au sens de la définition \ref{defi:equation}. L'objectif de cette partie est d'expliciter cette équation. Pour cela, commençons par rappeler la décomposition explicite des chaînes cohérentes du théorème \ref{theo:decomposition explicite}.
Soit $a := \sum_{i\leq m}  b_i +  r_{|a|_c}(a)$ une chaîne cohérente sous forme ordonnée. On définit : 
$$\beta_k  := b_k +~(d_{|a|_c}^-(\sum_{i<k}b_i)\diagdown d_{|a|_c}^+ b_k) \vee (d_{|a|_c}^+(\sum_{k<i\leq m}b_i) \diagdown d_{|a|_c}^- b_k)~+ r_{|a|_c}(a)$$
on a alors $$a = \beta_0 *_{|a|_c} \beta_1 *_{|a|_c}... *_{|a|_c} \beta_m .$$

\begin{defi}
\label{defi:definition de gamma}
Soient un entier $n>0$ et un entier $i\leq n$. On pose $\alpha := +$ si $i$ est pair et $\alpha := -$ si $i$ est impair.
La chaîne $d_{n-1}^\alpha i_n$ est de degré de composition $n-2$. On peut donc l'exprimer comme une composition de chaînes de degré de composition strictement inférieur, et on définit alors $\gamma$ comme étant le facteur comprenant $d^i$, et $a,b$ les $(n-1)$-cellules vérifiant:

$$d_{n-1}^\alpha i_n =a *_{n-2}\gamma*_{n-2} b.$$
La chaîne $\gamma$ est alors de degré de composition inférieur à $(n-3)$. On va répéter ce processus sur $\gamma$ afin "d'isoler" $d^i$.

On va définir  par une récurrence descendante sur $1\leq k \leq n$, une famille $(a^i_k,b_k^i,\gamma^i_k)_{1\leq k\leq n-1}$ vérifiant pour tout $k$,
\begin{enumerate}
\item  $a_k^i$ et $b_k^i$ sont des $k$-cellules,
\item $\gamma^i_k$ est une chaîne cohérente de degré de composition inférieur ou égal à $k-2$ et $r_{k-2}(\gamma^i_k) = 0$,
\item $\gamma^i_k$ comprend $d^i$, 
\item $\begin{array}{crcl}
\mbox{si $k=n-1$} &d_{n-1}^\alpha i_n &=&a^i_{n-1} *_{n-2}\gamma^i_{n-1} *_{n-2} b^i_{n-1}\\
\mbox{si $k< n-1$} & \gamma^i_{k+1} &=&a^i_k *_{k-1} \gamma^i_k *_{k-1} b^i_k.
\end{array}
$
\end{enumerate}

Pour cela, on pose tout d'abord alors $(a^i_{n-1},b_{n-1}^i,\gamma^i_{n-1}) := (a,b,\gamma)$. Supposons $(a^{i}_{k+1},b_{k+1}^i,\gamma^i_{k+1})$ construit pour $k<n-1$. 

Si $|\gamma^i_{k+1}|_c<k-1$, on pose 
$(a^{i}_{k},b_{k}^i,\gamma^i_{k}):=
(1_{d_{k-1}^+\gamma^i_{k+1} },1_{d_{k-1}^-\gamma^i_{k+1} },\gamma^i_{k+1}).$

Si $|\gamma^i_{k+1}|_c=k-1$, on définit $\gamma^i_{k}$ comme étant le facteur comprenant $d^i$ dans la décomposition de $\gamma^i_{k+1}$. Les cellules $a^{i}_{k}$ et $b_{k}^i$ sont les $k$-cellules vérifiant $\gamma^i_{k+1} =a^i_k *_{k-1} \gamma^i_k *_{k-1} b^i_k$. 
\label{definion des ai bi et gammai}

Le faite que $\gamma^i_k$ soit une chaîne cohérente de degré de composition inférieur ou égal à $k-2$ et que $r_{k-2}(\gamma^i_k) = 0$ provient de la construction explicite de la factorisation présentée dans le théorème \ref{theo:decomposition explicite}.
En particulier la chaîne cohérente $\gamma^i_1$ est de degré de composition $-1$, et est donc réduit à un singleton selon la proposition \ref{prop:coherende de degre moin un est reduit a un singloton}, d'où $\gamma^i_1=d^i$.
Enfin, remarquons que par construction, $\gamma^i_k$ est $(k-2)$-parallèle à $i_n$.
\end{defi}
\note{
\begin{rem}
\label{rem:sur les compose de gamma}
Soient $n$ un entier, et $p<q\leq n$ deux entiers de même parité. Si $p$ et $q$ sont pairs, on a $d^p\odot_{n-1}d^q$, et si $p$ et $q$ sont impairs, $d^q\odot_{n-1}d^p$.
Si $n$ est pair, on en déduit donc que les écritures ordonnées de $d_{n-1}^-i_n$ et $d_{n-1}^+i_n $ sont
$$d_{n-1}^-i_n = d^{n-1}+d^{n-3}+...+d^{1}~~~~d_{n-1}^+i_n = d^0 +d^2+...+d^n.$$
et on a alors par construction
$$d_{n-1}^-i_n = \gamma^{n-1}_{n-1}*_{n-2}\gamma^{n-3}_{n-1}*_{n-2} \cdots  *_{n-2} \gamma^{1}_{n-1}~~~~
d_{n-1}^+i_n =\gamma^0_{n-1}*_{n-2}\gamma^2_{n-1}*_{n-2} \cdots  *_{n-2} \gamma^{n}_{n-1}.$$
De façon analogue, si $n$ est impair, on a
$$d_{n-1}^-i_n =\gamma^n_{n-1}*_{n-2}\gamma^{n-2}_{n-1}*_{n-2} \cdots  *_{n-2} \gamma^{1}_{n-1}~~~~
d_{n-1}^+i_n =\gamma^0_{n-1}*_{n-2}\gamma^2_{n-1}*_{n-2} \cdots  *_{n-2} \gamma^{n-1}_{n-1}$$
\end{rem}}

\begin{exe}
Soient $n=4$ et $i=2$. On a alors $\alpha = +$. En se servant des notations et calculs de l'exemple \ref{exe : delta 4}, on a alors:$$\def\arraystretch{1.4}
\begin{array}{rclllllllll}
a_3^2& = & \sigma_{1 2 3 4 }+ \sigma_{0 1 4} &~~~& a_2^2& = & \sigma_{1 2 3} + \sigma_{0 1} + \sigma_{3 4}&~~~& a_1^2& = & 1_{\sigma_0} \\
\gamma_3^2 & = &\sigma_{0 1 3 4 }+ \sigma_{1 2 3}&~~~&\gamma_2^2 & = & \sigma_{0 1 3 4} &~~~&\gamma_1^2 & = & \sigma_{0 1 3 4}
\\
b_3^2& = &\sigma_{0 1 2 3}+ \sigma_{0 3 4}&~~~&b_2^2& = & 1_{\sigma_{0 4}}&~~~&b_1^2& = &1_{\sigma_4}.\\
\end{array}
$$
\end{exe}

\begin{rem}
Les cellules $a^i_k$ et $b^i_k$ sont des  compositions de chaînes cohérentes dont tous les éléments sont des simplexes de $\Lambda^i[n]$. Elles sont donc elles-mêmes des cellules  de $|\Lambda^i[n]|$. 
Les chaînes $d_{n-1}^+ i_n$ et $d_{k}^{\bar{\alpha}}(d^i)$ pour $\bar{\alpha}\in\{+,-\}$ et $k\leq n-2$ sont composées de simplexes de $\Lambda^i[n]$ et sont donc des cellules  $|\Lambda^i[n]|$. 
\end{rem}
Si $i$ est impair, on  a donc un isomorphisme en dessous de $|\Lambda^i[n]|$: 
$$|\Delta[n]| \cong \textbf{Eq}^+_{i,\Delta[n]} := \textbf{Eq}\big(y : (a^i_{n-1}*_{n-2}(...(a^i_1*_0 x *_0 b^i_1)...) *_{n-2} b^i_{n-1}) \to d_{n-1}^+ i_n \big),$$
et si $i$ est  pair, on a 
un isomorphisme en dessous de $|\Lambda^i[n]|$: 
$$|\Delta[n]| \cong   \textbf{Eq}^-_{i,\Delta[n]} := \textbf{Eq}\big(y : d_{n-1}^- i_n \to  (a^i_{n-1}*_{n-2}(...(a^i_1*_0 x *_0 b^i_1)...) *_{n-2} b^i_{n-1})\big).$$

\begin{prop}
Soit $C$ une catégorie $0$-triviale. Alors $\N(C)$ est un complexe de Kan. 
\label{prop:1trivial implique complexe de Kan}
\end{prop}
\begin{proof}
L'ensemble simplicial $\N(C)$ est un complexe de Kan, si et seulement si  pour tout entier $n>1$ et tout entier $i\leq n$, $C$ a la propriété de relèvement par rapport aux morphismes $|\Lambda^i[n]| \to |\Delta[n]|$  et donc si et seulement si pour tout entier $n>1$  et pour tout entier $i\leq n$, dans le cas où $i$ est impair, l’équation $\textbf{Eq}^+_{i,\Delta[n]}$ a  une pré-solution dans $C$ pour tout choix de paramètre  $f : |\Lambda^i[n]|\to C$, et dans le cas où $i$ est pair, l’équation $\textbf{Eq}^-_{i,\Delta[n]}$ a  une pré-solution dans $C$ pour tout choix de paramètre  $f : |\Lambda^i[n]|\to C$.
Montrons donc cette dernière assertion. On se donne donc un entier $n>1$, un entier impair $i\leq n$ et un choix de paramètre quelconque $f : |\Lambda^i[n]|\to C$.

Considérons les équations suivantes : 
$$
\textbf{Eq}_k := \textbf{Eq}\big(y:(a^i_{n-1}*_{n-2}(...(a^i_k*_{k-1} x *_{k-1} b^i_k)...) *_{n-2} b^i_{n-1})\to d_{n-1}^{+} i_n \big)
$$ $$
\mbox{où pour $\alpha\in \{-,+\}$, $d_{n-2}^\alpha(x)= 
a^i_{k-1}*_{k-2}(...(a^i_1*_0 d^\alpha_{n-2}(d^i) *_0 b^i_1)...) *_{k-2} b^i_{k-1}
$}
$$

On va montrer par une récurrence descendante sur  $k\leq n-1$ que les équations $\textbf{Eq}_k$ admettent de telles solutions.

Dans le cas $k=n-1$, il faut trouver une $(n-1)$-cellule $x$ vérifiant
$${f(a^i_{n-1})*_{n-2} x *_{n-2} f(b^i_{n-1})\sim f(d_{n-1}^{+} i_n)}.$$ Les cellules $f(a^i_{n-1})$ et $f(b^i_{n-1})$ étant faiblement inversibles, on peut appliquer le premier point de la proposition \ref{prop:propdedivision a droite} pour obtenir la cellule $x$ recherchée.

Supposons maintenant que l’équation $\textbf{Eq}_{k+1}$ admet des solutions pour les paramètres $f$ et soit $\tilde{x}$ l'une d'entre elles. Comme 
comme $C$ est $0$-trivial, cette cellule vérifie 
$$(f(a^i_{n-1})*_{n-2}(...(f(a^i_{k+1})*_{k} \tilde{x} *_{k} f(b^i_{k+1}))...) *_{n-2} f(b^i_{n-1}))\sim f(d_{n-1}^{+} i_n).$$
On définit alors:
$$\def\arraystretch{1.4}\begin{array}{rcl}
s &:= &f(a^i_{k-1})*_{k-2}(...(f(a^i_1)*_0 f(d^-_{n-2}(d^i)) *_0 f(b^i_1))...) *_{k-2} f(b^i_{k-1}) \\
t &:= &f(a^i_{k-1})*_{k-2}(...(f(a^i_1)*_0 f(d^+_{n-2}(d^i)) *_0 f(b^i_1))...) *_{k-2} f(b^i_{k-1}),
\end{array}
$$
et on a $$f(a^i_{k})*_{k-1} s *_{k-1} f(b^i_{k}) = d_{n-2}^- \tilde{x} 
\mbox{~~~ et~~~ }f(a^i_{k})*_{k-1} t *_{k-1} f(b^i_{k}) = d_{n-2}^+\tilde{x}.$$
Comme $f(a^i_{k})$ et $f(b^i_{k})$ sont faiblement inversibles, on peut appliquer le deuxième point de la proposition \ref{prop:propdedivision a droite} qui assure l'existence d'une cellule $x$ vérifiant 
$f(a^i_{k-1})*_{k-2} x *_{k-2} f(b^i_{k-1}) \sim  \tilde{x}$.
On a alors 
$$(f(a^i_{n-1})*_{n-2}(...(f(a^i_k)*_{k-1} x *_{k-1} f(b^i_k))...) *_{n-2} f(b^i_{n-1}))\to f(d_{n-1}^{+} i_n) $$
et $x$ est une solution de $\textbf{Eq}_k$ pour les paramètres $f$.

On peut donc trouver des solutions pour les équations $\textbf{Eq}_k$ pour tout choix de paramètre, et en particulier pour $\textbf{Eq}_1$ qui est égale à $\textbf{Eq}^+_{i,\Delta[n]}$. On peut montrer de façon analogue que les équations $\textbf{Eq}^-_{i,\Delta[n]}$ admettent des solutions pour tout choix de paramètre.
Cela prouve qu'on peut relever les inclusions de cornets,  et donc que $\N(C)$ est un complexe de Kan.
\end{proof}

\begin{prop}
Soit $C$ une catégorie $1$-triviale. Alors $\N(C)$ est une quasi-catégorie. 
\label{prop:2trivial implique quasi categorie}
\end{prop}

\begin{lem}
\label{lemetechnique}
Soit un entier $0<i<n$. Alors $a_1^i$ et $b_1^i$ sont des unités.
\end{lem}
\begin{proof}
On a par définition l'égalité suivante:
$$\gamma^i_2 = a^1_i *_0d^i *_0 b^1_i.$$
Or pour $0<i<n$, $d_0(d^i) = 0$ et $d_1(d^i)=n$, comme la $\omega$-catégorie $\Delta[n]$ est sans boucles, les cellules  $a^1_i$ et $b^1_i$ sont forcément des unités.
\end{proof}

\begin{proof}[Démonstration de la proposition \ref{prop:2trivial implique quasi categorie}]
On va procéder de façon analogue à la preuve de la proposition \ref{prop:1trivial implique complexe de Kan}.

L'ensemble simplicial $\N(C)$ est une quasi-catégorie, si et seulement si  pour tout entier $n>1$ et tout entier $0<i< n$, $C$ a la propriété de relèvement par rapport aux morphismes $|\Lambda^i[n]| \to |\Delta[n]|$  et donc si et seulement si pour tout entier $n>1$  et pour tout entier $0<i< n$, dans le cas où $i$ est impair, l’équation $\textbf{Eq}^+_{i,\Delta[n]}$ a  une pré-solution dans $C$ pour tout choix de paramètre  $f : |\Lambda^i[n]|\to C$, et dans le cas où $i$ est pair, l’équation $\textbf{Eq}^-_{i,\Delta[n]}$ a  une pré-solution dans $C$ pour tout choix de paramètre  $f : |\Lambda^i[n]|\to C$.

On se donne un entier $n>1$, un entier  impair $0<i<n$, une application $f:|\Lambda^i[n]|\to C$ et comme plus haut, on définit pour $k\geq 2 $:
$$
\textbf{Eq}_k := \textbf{Eq}\big(y:(a^i_{n-1}*_{n-2}(...(a^i_k*_{k-1} x *_{k-1} b^i_k)...) *_{n-2} b^i_{n-1})\to d_{n-1}^{+} i_n \big)
$$ $$
\mbox{où pour $\alpha\in \{-,+\}$, $d_{n-2}^\alpha(x)= 
a^i_{k-1}*_{k-2}(...(a^i_1*_0 d^\alpha_{n-2}(d^i) *_0 b^i_1)...) *_{k-2} b^i_{k-1}
$}
$$

Pour $k\geq 2$, les cellules $a^i_k$ et $b_i^k$ sont de dimension au moins $2$, et donc, par hypothèse, les cellules $f(a^i_k)$ et $f(b_i^k)$ sont faiblement inversibles. 
On peut alors montrer par une récurrence descendante sur $k$, de la même façon que dans la preuve de la proposition \ref{prop:1trivial implique complexe de Kan}, que pour $k\geq 2$, l'équation $\textbf{Eq}_k$ admet une solution pour les paramètres $f$.

Or selon le lemme \ref{lemetechnique}, $a^i_1$ et  $b^i_1$ sont des unités et donc 
$\textbf{Eq}_2 = \textbf{Eq}^+_{i,\Delta[n]}$. Cela prouve donc que cette équation a toujours des solutions dans $C$. De façon analogue, pour tout entier $n>0$ et tout entier $0<i<n$, l'équation $\textbf{Eq}^-_{i,\Delta[n]}$ admet toujours des solutions dans $C$. 
Cela prouve donc que $\N(C)$ est une quasi-catégorie.
\end{proof}
On peut alors résumer les résultats précédents en un seul théorème:

\begin{theo}
\label{theo:theo principale}
Soit $C$ est une $\omega$-catégorie. Les trois assertions sont équivalentes : 
\begin{enumerate}
	\item  L'ensemble simplicial $\N(C)$ a la propriété de relèvement par rapport aux inclusions ${\Lambda^2[n+1]\to \Delta[n+1]}$ pour tout $n>0$ (resp. pour tout $n>1$);
	\item L'ensemble simplicial $\N(C)$ est un complexe de Kan (resp. une quasi-catégorie);
	\item La $\omega$-catégorie $C$ est $0$-triviale (resp. $1$-triviale).
\end{enumerate}
\end{theo}

\section{Généralisation au nerf complicial}
\subsection{Ensembles compliciaux}
On ne donnera  ici que les définitions et résultats qui nous seront utiles pour la suite. Pour une introduction détaillée, voir \cite{riehl}.

\begin{defi}
Une \textit{stratification} d'un ensemble simplicial $X$ est un sous-ensemble $tX\subset\coprod_{n>0} X_n$  qui contient l'ensemble des simplexes dégénérés.
\end{defi}

Un ensemble stratifié est un couple $(X,tX)$. On dit  des simplexes étant dans $tX$ qu'ils  sont \textit{marqués}. Pour $(X,tX)$ et $(Y,tY)$ des ensembles stratifiés,
un morphisme  d'ensembles simpliciaux $f: X\to Y$ est stratifié si $f(tX)\subset tY$. On note  $\strat$ la catégorie des ensembles stratifiés.

\begin{defi}
Une inclusion $ i :U\to V$ entre ensembles stratifiés est :
\begin{enumerate}
	\item \textit{régulière}, notée $\to_r$, si un simplexe est marqué dans $U$ si et seulement il l'est dans $V$;
	\item \textit{pleine}, notée $\to_e$, si le morphisme est l'identité sur les ensembles simpliciaux sous-jacents. 
\end{enumerate}
\end{defi}

\begin{notation}
Si $i : U\to V$ est une inclusion et $V'$ est une stratification de $V$, il existe une unique stratification de $U$ rendant $i$ régulière. Elle correspond à celle où un simplexe de $U$  est marqué si et seulement si son image par $i$ l'est. Cette stratification sera notée $U^\circ$. On a alors $i : U^\circ\to_r V'$.
\end{notation}

\begin{defi}
Pour $0\leq k \leq n$ on définit l'ensemble stratifié $\Delta^k[n]$ dont l'ensemble simplicial sous-jacent est $\Delta[n]$. Les simplexes marqués sont ceux qui comprennent $\{k-1,k,k+1\}\cap [n]$.

On définit l’ensemble stratifié $\Delta^k[n]'$, obtenu à partir de $\Delta^k[n]$ et en marquant la $(k-1)$- et la $(k+1)$-face. L'ensemble stratifié $\Delta^k[n]''$ est obtenu en marquant tous les simplexes de codimension $1$.
\end{defi}

\begin{defi}
Un ensemble stratifié $X$ est un \textit{ensemble complicial}, si $X\to 1$ a la propriété de relèvement par rapport aux morphismes 
$ \Lambda^k[n]^\circ \to_r  \Delta^k[n]$ et $\Delta^k[n]'\to_e \Delta^k[n]''$.
\end{defi}

\begin{defi}
Un ensemble complicial est \textit{$k$-trivial} si toutes les  cellules de dimension strictement supérieure à $k$ sont marquées.
\end{defi}

\begin{defi}
On définit $\Delta[3]^{eq}$ comme étant l'ensemble stratifié sur $\Delta[3]$ où $[02]$ et $[13]$ sont marqués ainsi que tous les simplexes de dimension au moins 2.  On note $\Delta[3]^\#$ l'ensemble stratifié où tous les simplexes non dégénérés sont marqués.
\end{defi}

\begin{defi}
On note $\star: \textbf{Sset}\times \textbf{Sset}\to \textbf{Sset}$ le joint d'ensembles simpliciaux. On l'étend aux ensembles stratifiés de la façon suivante: soient $U$ et $V$ deux ensembles stratifiés, un simplexe
$v :\Delta[n]\to U\star V$ est marqué dès lors qu'un des morphismes induit $v_1 : \Delta[i]\to U$,  $v_2 : \Delta[n-i-1]\to V$ l'est. 
Cela permet donc de définir le\textit{ joint d'ensembles stratifiés} :$\star: \strat\times \strat\to \strat$.
\end{defi}

\begin{defi}
\label{defi:ensembles complicial sature}
Un ensemble complicial  est \textit{saturé} s'il a la propriété de relèvement par rapport aux morphismes
$\Delta[3]^{eq}\star \Delta[n]\to \Delta[3]^\#\star \Delta[n]$ et $\Delta[n]\star \Delta[3]^{eq}\to \Delta[n]\star\Delta[3]^\#.$
\end{defi}

\subsection{Nerfs et ensembles compliciaux}
 \begin{defi}
 \label{defi:stratification de n(c)}
 On définit une stratification sur $\N(C)$. Un simplexe $v \in \N(C)_n$ correspond à un morphisme entre $\omega$-catégories : $v:|\Delta[n]|\to C$. On note $i_n$ l'unique simplexe non dégénéré de dimension $n$ de $\Delta[n]$. Le simplexe $v$ est marqué si  $v(i_n)$ est une $n$-cellule faiblement inversible. Par abus de langage on note aussi $\N(C)$ l'ensemble stratifié obtenu.
 \end{defi}

Le but de cette section est de montrer que $\N(C)$ muni de cette stratification est un ensemble complicial. 
On va procéder de la même façon que dans les preuves des propositions \ref{prop:1trivial implique complexe de Kan} et \ref{prop:2trivial implique quasi categorie}.
Avant cela, on a besoin de plusieurs résultats:

\begin{prop}
\label{prop:chaines est faible inversible}
Soient $K$ un complexe dirigé augmenté admettant une base sans boucles et unitaire, $C$ une $\omega$-catégorie, et un morphisme $f : \mu K \to C$. Soit $a$ une chaîne cohérente telle que tout élément de $a$ de degré $|a|$ soit envoyé par $f$ sur une cellule faiblement inversible. Alors $f(a)$ est faiblement inversible.
\end{prop}

\begin{proof}
Rappelons que $|a|$ est la dimension maximale des éléments de la base présents dans $a$, et donc $|a|> |a|_c$. 
On va procéder par récurrence sur le degré de composition. Lorsque $|a|_c = -1$, alors, selon la proposition \ref{prop:coherende de degre moin un est reduit a un singloton}, $a$ est réduit à un unique élément, et la propriété est trivialement vraie.

Supposons maintenant le résultat vrai pour les chaînes de degré de composition $m$, et donnons nous une chaîne cohérente $a:= \sum_{i\leq n} b_i +r_{|a|_c}(a) $ écrite sous forme ordonnée, de degré de composition $m+1$  et telle que pour tout $i\leq n$, si $|b_i| =|a|$, alors $f(b_i)$ est faiblement inversible. 

Selon le théorème \ref{theo:decomposition explicite}, on a 
$$a = \beta_0 *_{|a|_c} \beta_1 *_{|a|_c}... *_{|a|_c} \beta_m $$
où $$\beta_k  := b_k +~(d_{|a|_c}^-(\sum_{i<k}b_i)\diagdown d_{|a|_c}^+ b_k) \vee (d_{|a|_c}^+(\sum_{k<i\leq m}b_i) \diagdown d_{|a|_c}^- b_k)~+ r_{|a|_c}(a).$$

Pour tout $k$ le degré de composition de $\beta_k$ est inférieur ou égal à $m$. Deux cas de figure se présentent alors. Le premier est celui où $|\beta_k|=|a|$. L'élément $b_k$ est le seul de degré $|\beta_k|$ et il est alors envoyé par $f$ sur une cellule faiblement inversible. On peut donc  appliquer l'hypothèse de récurrence qui implique que  $f(\beta_k)$ est faiblement inversible. Le second cas est celui où $|\beta_k|<|a|$. La $|a|$-cellule correspondant à $\beta_k$ est alors une unité, donc inversible et \textit{a fortiori} faiblement inversible. 

La cellule $f(a)$ est donc une composition de $|a|$-cellules faiblement inversibles et est donc elle-même faiblement inversible.
\end{proof}

\begin{prop}
\label{prop:expression d'un face}
Soient  $x$ un $n$-simplexe non dégénéré de $\Delta[m]$, et $d$ une $k$-face de $x$ avec $k\geq 1$. Alors il existe une chaîne cohérente $a \in C_\bullet(\Delta[m])$, comprenant $d$ et  $(k-1)$-parallèle à $x$.
\end{prop}

\begin{proof}
On va   montrer le résultat pour tout couple $(x,d)$ où $x$ est un $n$-simplexe, et $d$ une $k$-face de $x$, par récurrence sur $n-k$. 

Supposons  tout d'abord $n-k=1$, c'est-à-dire $k=n-1$. On peut se ramener au cas où $m=n$, $x=i_n$ et $d$ est un $(n-1)$-simplexe de $\Delta[n]$.
Soit $i$ l'entier tel que $d^i = d$. On définit $\alpha = +$ si $i$ est pair, $\alpha=-$ sinon. La chaîne $d_{n-1}^\alpha i_n$ comprend $d$, et est $(n-2)$-parallèle à $x$.

Supposons maintenant  que le résultat est vrai pour $n-k =m$ et montrons le pour $n-k=m+1$. On se ramène  encore une fois au cas où $m=n$, $x=i_n$ et $d$ est un $k$-simplexe de $\Delta[n]$.
Il existe un entier $i$ tel que $d$ soit une $k$-face de $d^i$. On peut alors appliquer l'hypothèse de récurrence sur $(d^i,d)$ et il existe donc une chaîne cohérente $\tilde{a}$, comprenant $d$, et  $(k-1)$-parallèle à $d^i$. 

Or la chaîne $\gamma^i_{k+1}$, défini en \ref{defi:definition de gamma}, est de degré de composition $k-1$, comprend $d^i$, et peut donc s'exprimer sous la forme 
$\gamma^i_{k+1} = d^i+ c$ où $c$ est de degré $k$. Selon le corollaire \ref{cor:lembis}, la chaîne $\tilde{a}+c$ est donc $(k-1)$-parallèle à $\gamma^i_{k+1}$, et donc $(k-1)$-parallèle à $i_n$. De plus, elle comprend $d$ et vérifie donc les conditions voulues.
\end{proof}

\begin{lem}
\label{lem:condition pour une relation}
Soient
$a$ une chaîne cohérente de $C_\bullet (\Delta[m])$ et  $b$,$b'\in a$  deux simplexes de $\Delta[m]$, de dimensions strictement supérieures au degré de composition de $a$. Supposons de plus que $b$ et $b'$ aient une $|a|_c$-face en commun. On a alors $$b\odot_{|a|_c}b' \mbox{ ou }b'\odot_{|a|_c}b.$$
\end{lem}
\begin{proof}
Les éléments $b$ et $b'$ jouent des rôles symétriques, on peut donc supposer que $|b'|\geq |b|$. La chaîne $a$ peut s'exprimer sous la forme $a = b'+b+c$ où $|b+c| = |a|_c+1$.
On sait de plus qu'il existe une $(|a|_c+1)$-face $\tilde{b}$ de $b'$ telle que $\tilde{b}$ et $b$ aient une $|a|_c$-face en commun. Il existe donc $\alpha,\beta\in\{-,+\}$ tel que $d^\alpha_{|a|_c}b\wedge d^\beta_{|a|_c} \tilde{b}\neq 0$.

Selon la proposition \ref{prop:expression d'un face}, il existe une chaîne $a'$, $|a|_c$-parallèle à $b'$ et qui comprend $\tilde{b}$. Il existe donc $\tilde{a}$ tel que $a'=\tilde{a}+\tilde{b}$. Le corollaire \ref{cor:lembis} indique que la chaîne $\tilde{a}+\tilde{b}+b+c$ est $|a|_c$-parallèle à $a$ et est donc cohérente. La propriété \ref{prop: plus de fourche } appliquée à la chaîne $\tilde{a}+\tilde{b}+b+c$
implique alors que $\alpha = -\beta$. On suppose $\alpha=-$, l'autre cas étant similaire. On a donc $d^-_{|a|_c} b\wedge d^+_{|a|_c} \tilde{b}\neq 0$. En utilisant encore une fois la propriété \ref{prop: plus de fourche }, on sait que pour tout $v \in \tilde{a}$ on a  $d_{|a|_c} ^- b\wedge d_{|a|_c} ^- v =  0$, d'où, selon le lemme \ref{lem:pseudolinearite}: $$d_{|a|_c} ^- b\wedge d_{|a|_c} ^- (\tilde{a})\leq d_{|a|_c} ^- b\wedge \sum_{v\in \tilde{a}} d_{|a|_c}^-v=  0.$$ On a donc
$$\def\arraystretch{1.4}
\begin{array}{rcl}
d^-_{|a|_c} b\wedge d^+_{|a|_c} b' &=&d^-_{|a|_c} b\wedge d^+_{|a|_c} (\tilde{a}+ b)\\
& =  & d^-_{|a|_c} b\wedge \big((d^+_{|a|_c} \tilde{b}\diagdown d_{|a|_c} ^- (\tilde{a}))_+
+ (d_{|a|_c} ^+ (\tilde{a})\diagdown d^-_{|a|_c}\tilde{b})_+\big)\\
&\geq & d^-_{|a|_c} b\wedge (d^+_{|a|_c} \tilde{b}\diagdown d_{|a|_c} ^- (\tilde{a}))_+\\
&\geq &d^-_{|a|_c} b\wedge d^+_{|a|_c} \tilde{b}\\
&\neq& 0.
\end{array}
$$
On a alors obtenu $b\odot_{|a|_c}b'$.
\end{proof}

\begin{prop}
\label{prop:complicial axiom un}
Soit $i\leq n$ un entier.
Tout simplexe dans $\gamma^i_k$ différent de $d^i $ comprend $\{i\}$. 
\end{prop}
\begin{proof}
La chaîne $\gamma^i_k$ est de degré de composition $k-2$ et vérifie $r_{k-2}(\gamma^i_k) = 0$. On peut donc l'exprimer sous la forme $\gamma^i_k = d^i+c$ où  $c$ est homogène de degré $k-1$. Donnons nous un $(k-1)$-simplexe $v$ quelconque ne comprenant pas $i$. C'est donc en particulier une $k$-face de $d^i$, et donc selon la proposition \ref{prop:expression d'un face}, il existe une chaîne $a$, $(k-2)$-parallèle à $d^i$ et comprenant $v$. Le corollaire \ref{cor:lembis} implique alors que la chaîne $a +c$ est $(k-2)$-parallèle à $d^i$, et donc cohérente. La proposition \ref{prop:pas de double} implique que $v\notin c$. Les simplexes apparaissant dans $\gamma^i_k$ et différents de $d^i $ comprennent donc $\{i\}$. 
\end{proof}

\begin{rem}
Un simplexe $v\in d_{n-1}^\alpha i_n\diagdown d^i$  est de la forme $d^k$ pour $k\neq i$. Rappelons que $\Delta[n]\cong \kappa^*\kappa[n]$, et donc le lemme \ref{lem:condition pour une relation} appliqué à $\kappa [n]$ implique que $v$ et $d^i$ sont comparables pour la relation $\odot_{n-2}$.
En posant $\alpha^v_i := +$ si $i$ est en position paire dans $v$, et $\alpha^v_i  = -$ si $i$ est en position impaire dans $v$, on a alors 
$$
\begin{array}{rl}
v\odot_{n-2}d^i &\mbox{si $\alpha^v_i =-$}\\
d^i\odot_{n-2}v &\mbox{si $\alpha^v_i =+$}.
\end{array}
$$

De même, la proposition \ref{prop:complicial axiom un} implique que tout $v\in \gamma^i_k \diagdown d^i$ comprend $i$. Il existe donc un entier $k$ tel que $d_k v$ ne comprenne pas $i$ et donc $d_k v$ est une face de $d_i$. En posant encore une fois $\alpha^v_i := +$ si $i$ est en position paire dans $v$, et $\alpha^v_i  = -$ si $i$ est en position impaire dans $v$, on a alors 
$$
\begin{array}{rl}
v\odot_{k-2}d^i &\mbox{si $\alpha^v_i =-$}\\
d^i\odot_{k-2}v &\mbox{si $\alpha^v_i =+$}.
\end{array}
$$
\label{rem:sur la comparabilite}
\end{rem}

\begin{prop}
\label{prop:complicial axiom}
Soient $i$ un entier tel que $0\leq i\leq n$, et $\alpha=+$ si $i$ est pair, et $\alpha=-$ sinon. Alors
tout simplexe dans $d^{\alpha}_{n-1} i_n$ différent de $d^i$ comprend $\{ i-1, i, i+1\}\cap [n]$ et
tout simplexe dans $\gamma^i_k$ différent de $d^i $ comprend $\{ i-1, i, i+1\}\cap [n]$. 
\end{prop}

On a besoin de deux lemmes:

\begin{lem} Pour $k<n-1$, on a les inégalités suivantes
$$\def\arraystretch{1.4}
\begin{array}{rcl}
\gamma^i_{n-1}&\leq& d^i + \sum_{v\in d_{n-1}^\alpha i_n\diagdown d^i} d^{\alpha^v_i}_{n-2}(v)\\
\gamma^i_{k}&\leq& d^i + \sum_{v\in  \gamma^i_{k+1}\diagdown d^i} d^{\alpha^v_i}_{k-1}(v).
\end{array}
$$
\label{lem:inegalite}
\end{lem}
\begin{proof}
La chaîne $\gamma^i_{n-1}$ est définie comme étant le facteur  comprenant $d^i$ dans la décomposition de $d_{n-1}^\alpha i_n$. 
On écrit cette chaîne sous forme ordonnée: $d_{n-1}^\alpha i_n = \sum_{i\leq m}b_i$. On dénote par $l$ l'entier vérifiant $b_l = d^i$. Comme la base de $C_\bullet(\Delta[n])$ est sans boucles, on déduit de la remarque \ref{rem:sur la comparabilite} que
pour tout $j<l$, $b_j\odot_{n-2} d^i$ et pour tout $j>l$, $d^i\odot_{n-2} b_j$. Selon la décomposition explicite du théorème \ref{theo:decomposition explicite}, on a donc:
$$\def\arraystretch{1.8}\begin{array}{rcl}
\gamma^i_{n-1} &=&d^i + d_{n-2}^+(\sum^{v\in d_{n-1}^\alpha i_n\diagdown d^i}_{ d^i\odot_{n-2}v} v)\diagdown d_{n-2}^-(d^i) +
 d_{n-2}^-(\sum^{v\in d_{n-1}^\alpha i_n\diagdown d^i}_{ v\odot_{n-2}d^i} v)\diagdown d_{n-2}^+(d^i)\\
 &\leq &d^i + d_{n-2}^+(\sum^{v\in d_{n-1}^\alpha i_n\diagdown d^i}_{ d^i\odot_{n-2}v} v)+
 d_{n-2}^-(\sum^{v\in d_{n-1}^\alpha i_n\diagdown d^i}_{ v\odot_{n-2}d^i} v).
\end{array}
$$

En appliquant le lemme \ref{lem:pseudolinearite}, on obtient bien,
$$\gamma^i_{n-1}\leq  d^i + \sum_{v\in d_{n-1}^\alpha i_n\diagdown d^i} d^{\alpha^v_i}_{n-2}(v).
$$

De même $\gamma^i_{k}$ est défini comme étant le facteur  comprenant $d^i$ dans la décomposition de $\gamma^i_{k+1}$. On écrit cette chaîne sous forme ordonnée: $\gamma^i_{k+1}  = \sum_{i\leq m} b_i$ et on dénote par $l$ l'entier vérifiant $b_l = d^i$. Comme plus haut,
pour tout $j<l$, $b_j\odot_{k-1} d^i$ et pour tout $j>l$, $d^i\odot_{k-1} b_j$. On a donc
$$\def\arraystretch{1.8}
\begin{array}{rcl}
\gamma^i_{k} &=&d^i + d_{k-1}^+(\sum^{v\in \gamma^i_{k+1}\diagdown d^i}_{ d^i\odot_{k-1}v} v)\diagdown d_{k-1}^-(d^i) +
 d_{k-1}^-(\sum^{v\in \gamma^i_{k+1}\diagdown d^i}_{ v\odot_{k-1}d^i} v)\diagdown d_{k-1}^+(d^i)\\
&\leq & d^i + d_{k-1}^+(\sum^{v\in \gamma^i_{k+1}\diagdown d^i}_{ d^i\odot_{k-1}v} v) +
 d_{k-1}^-(\sum^{v\in \gamma^i_{k+1}\diagdown d^i}_{ v\odot_{k-1}d^i} v).
\end{array}
$$
d'où
$$\gamma^i_{k}\leq d^i + \sum_{v\in  \gamma^i_{k+1}\diagdown d^i} d^{\alpha^v_i}_{k-1}(v).
$$
\end{proof}

\begin{lem}
Soit $v$ un $k$-simplexe comprenant $\{ i-1, i, i+1\}\cap [n]$. Alors les éléments de $d_{k-1}^{\alpha^v_i}(v)$ comprennent  $\{ i-1, i+1\}\cap [n]$.
\label{lem: i i plus un}
\end{lem}
\begin{proof}
Donnons nous un tel simplexe. On suppose que $\alpha^v_i=+$, c'est-à-dire que  $i$ est en position paire. Les 
entiers $i-1$ et $i+1$ sont donc en position impaire. Or
$d_{k-1}^{\alpha^v_i}(v) = d_{k-1}^{+}(v) = d_p v$.  Un simplexe $\tilde{v}\in d_{k-1}^{\alpha^v_i}(v)$ est donc de la forme $d_{2j}v$, et comprend $\{ i-1, i+1\}\cap [n]$.
\end{proof}

\begin{proof}[Démonstration de la proposition \ref{prop:complicial axiom}]
Remarquons tout d'abord que les éléments de $d_{n-1}^\alpha i_n\diagdown d^i$ sont de la forme $d^j$ pour un $j$ de la même parité que $i$ et différent de $i$. Ils comprennent donc $\{ i-1, i, i+1\}\cap [n]$.

Montrons maintenant par une récurrence descendante sur $k$ que tout $v\in \gamma^i_k\diagdown d^i$ comprend ${\{ i-1, i, i+1\}\cap [n]}$. Commençons donc par le cas $k=n-1$, et donnons nous $v\in\gamma^i_{n-1}\diagdown d^i$. Selon le lemme \ref{lem:inegalite}, il existe donc $\tilde{v}\in d_{n-1}^\alpha i_n\diagdown d^i$ tel que $ v\in d^{\alpha^v_i}_{n-2}(\tilde{v})$. Le lemme \ref{lem: i i plus un} implique donc que $v$ comprend ${\{ i-1, i+1\}\cap [n]}$ et la proposition \ref{prop:complicial axiom un} que $v$ comprend $i$.

Supposons maintenant le résultat vrai pour les simplexes de $\gamma^i_{k+1}$ et donnons nous $v\in \gamma^i_{k}\diagdown d^i$. Le lemme \ref{lem:inegalite} implique qu'il existe donc $\tilde{v}\in \gamma^i_{k+1}\diagdown d^i$ tel que $ v\in d^{\alpha^v_i}_{k-1}(\tilde{v})$. Par hypothèse de récurrence $\tilde{v}$ comprend ${\{ i-1, i, i+1\}\cap [n]}$, et le lemme \ref{lem: i i plus un} et la proposition \ref{prop:complicial axiom un} impliquent alors le résultat.
\end{proof}

\begin{prop}
\label{prop: n(c) est un ensemble complicial}
L'ensemble stratifié $\N(C)$ est un ensemble complicial.	
\end{prop}

Rappelons que pour tout entier $n>0$ et tout entier $i\leq n$, si $i$ est pair, on  a  un isomorphisme en dessous de $|\Lambda^i[n]|$: 
$$|\Delta[n]| \cong \textbf{Eq}^+_{i,\Delta[n]} := \textbf{Eq}\big(y : (a^i_{n-1}*_{n-2}(...(a^i_1*_0 x *_0 b^i_1)...) *_{n-2} b^i_{n-1}) \to d_{n-1}^+ i_n \big),$$
et si $i$ est  impair, on a 
un isomorphisme en dessous de $|\Lambda^i[n]|$: 
$$|\Delta[n]| \cong   \textbf{Eq}^-_{i,\Delta[n]} := \textbf{Eq}\big(y : d_{n-1}^- i_n \to  (a^i_{n-1}*_{n-2}(...(a^i_1*_0 x *_0 b^i_1)...) *_{n-2} b^i_{n-1})\big).$$

\begin{proof}
Soient $i\leq n$ un entier et $\alpha=+$ si $i$ est pair, et $\alpha=-$ sinon.
Montrons que $\N(C)$ a la propriété de relèvement par rapport à l'inclusion d'ensembles compliciaux $\Lambda^i[n]^\circ\to \Delta^i[n]$. 
On se donne un morphisme $f: \Lambda^i[n]^\circ\to \N( C)$. Cela correspond à un morphisme $f : |\Lambda^i[n]|\to C$ qui envoie tout simplexe  de $\Lambda^i[n]$ comprenant $\{ i-1, i, i+1\}\cap [n]$ sur une cellule faiblement inversible.

Soient $k\leq n+1$ et $v$ un $k$-simplexe dans $ a^i_k$ ou $ b^i_k$. La proposition \ref{prop:complicial axiom}, implique que $v$ comprend $\{ i-1, i, i+1\}\cap [n]$, et est donc  envoyé par $f$ sur une cellule faiblement inversible. La proposition \ref{prop:chaines est faible inversible} implique donc que le morphisme $f$ envoie $a^i_k$ et $b^i_k$ sur des cellules faiblement inversibles.

  On peut donc procéder de la même façon que dans la preuve de la proposition \ref{prop:1trivial implique complexe de Kan} pour trouver, pour tout $\alpha\in\{-,+\}$, une pré-solution à l’équation $\textbf{Eq}^\alpha_{i,\Delta[n]}\cong |\Delta[n]|$ pour les paramètres $f : |\Lambda^i[n]|\to C$.

On veut maintenant montrer que pour tout entier $n>0$ et $i\leq n$, $\N(C)$  a la propriété de relèvement par rapport à l'inclusion d'ensembles compliciaux $\Delta^i[n]'\to \Delta^i[n]''$. La donnée d'un morphisme $f:\Delta^i[n]'\to \N(C)$ est équivalente à celle d'un morphisme $f:\textbf{Eq}^\alpha_{i,\Delta[n]}\cong|\Delta^i[n]| \to C$ qui envoie  $d^{i-1}$, $d^{i+1}$ et les simplexes  
comprenant $\{i-1,i,i+1\}\cap[n]$ sur des cellules faiblement inversibles. Ce morphisme se factorise par $\Delta^i[n]''$ si et seulement si $f(x)$ est faiblement inversible. 

On définit $\alpha=+$ si $i$ est pair, et $\alpha=-$ sinon.
Tous les $(n-1)$-simplexes de $d_{n-1}^\alpha i_n$ sont envoyés sur des cellules faiblement inversibles et selon la proposition \ref{prop:chaines est faible inversible}, cela implique que $d_{n-1}^\alpha i_n$ est envoyé sur une cellule faiblement inversible. Pour les mêmes raisons que plus haut les chaînes  $a^i_k, b^i_k$ et $d_{n-1}^\alpha i_n$ sont envoyées sur des cellules faiblement inversibles.

Donnons nous un tel morphisme $f$. On a donc 
$$  (f(a^i_{n-1})*_{n-2}(...(f(a^i_1)*_0 f(x) *_0 f(b^i_1))...) *_{n-2} f(b^i_{n-1})) \sim  f(d_{n-1}^\alpha 
i_n)
$$
Une application répétée du corollaire \ref{cor;compatibilite de faible inversible} implique alors le résultat. 
\end{proof}

\begin{prop}
\label{prop:n(c) est sature}
L'ensemble complicial $\N(C)$ est saturé.
\end{prop}
\begin{proof}
On va montrer par récurrence  sur $n\geq -1$ que $\N(C)$ a la propriété de relèvement à droite par rapport aux morphismes $\Delta[3]^{eq}\star\Delta[n-3]\to \Delta[3]^{\#}\star\Delta[n-3]$, où on définit pour un ensemble complicial quelconque $X$, $X\star \Delta[-1] := X$. 

Commençons par se donner un morphisme $g:\Delta[3]^{eq}\to \N(C)$. Cela correspond à un morphisme $g : |\Delta[3]| \to C$ qui envoie tout simplexe  de $\Delta[3]$ comprenant $\{ 0,2\}$ ou $\{1,3\}$ sur une cellule faiblement inversible.  Ce morphisme se factorise  par $|\Delta[3]^{eq}|$ si et seulement si $g([0,1]), g([0,3]), g([1,2])$ et $g([2,3])$ sont des cellules faiblement inversibles. Remarquons alors que l'on a 
$$g([1,2])*_0 g([0,1])\sim g([0,2])\mbox{   ~~et~~   }g([2,3])*_0 g([1,2])\sim g([1,3])$$
Les cellules $g([0,2])$ et $g([1,3])$ étant par hypothèse faiblement inversibles, on en déduit que $g([0,1]), g([1,2])$ et $ g([2,3])$  le sont aussi. Enfin, la relation 
$$g([0,3])\sim g([2,3])*_0g([1,2])*_0g([0,1])$$ implique que $g([0,3])$ est aussi faiblement inversible.

Supposons maintenant le résultat vrai pour $n\geq 3$. 
On se donne un morphisme $$g: \Delta[3]^{eq}\star \Delta[n-2] \to \N( C).$$ Cela correspond à un morphisme $$g : |\Delta[n+1]|\cong|\Delta[3]\star \Delta[n-2]|\to C$$ qui envoie tout simplexe  comprenant $\{ 0,2\}$ ou $\{1,3\}$ sur une cellule faiblement inversible. De plus, l'hypothèse de récurrence implique que  tout simplexe $v$ de dimension strictement inférieure ou égale à $(n-2)$ et tel que $\{0,1,2,3\}\cap v$ soit de cardinal au moins $2$, est envoyé par $g$ sur une cellule faiblement inversible. Ce morphisme se factorise par $|\Delta[3]^{eq}\star \Delta[n-2]|$ si et seulement si $g(d^{0,1}), g(d^{0,3}), g(d^{1,2})$ et $g(d^{2,3})$ sont des cellules faiblement inversibles.

On va tout d'abord s’intéresser au morphisme induit:
$$f: |\Delta[n]|\cong |\Delta[2]\star \Delta[n-2]|\xrightarrow{ |\Delta [d_3] \star \Delta[n-3]|}|\Delta[3]\star \Delta[n-2]|\xrightarrow{~g~}C.$$

\note{La cellule $f(i_n)$ est  alors faiblement inversible. On a donc $f(d^-_{n-1}(i_n))\sim f(d^+_{n-1}(i_n))$. Selon la remarque \ref{rem:sur les compose de gamma}, 
si $n$ est pair, on a alors
$$f(\gamma^{n-1}_{n-1})*_{n-2}f(\gamma^{n-3}_{n-1})*_{n-2} \cdots  *_{n-2} f(\gamma^{1}_{n-1})
 \sim f(\gamma^0_{n-1})*_{n-2}f(\gamma^2_{n-1})*_{n-2} \cdots  *_{n-2} f(\gamma^{n}_{n-1})$$
et si $n$  est impair,
$$f(\gamma^n_{n-1})*_{n-2}f(\gamma^{n-2}_{n-1})*_{n-2} \cdots  *_{n-2} f(\gamma^{1}_{n-1})
 \sim f(\gamma^0_{n-1})*_{n-2}f(\gamma^2_{n-1})*_{n-2} \cdots  *_{n-2} f(\gamma^{n-1}_{n-1}).$$
}
Pour tout $i=1$ ou $i>2$, la $(n-1)$-cellule $f(d^i)$ est faiblement inversible, et la proposition \ref{prop:chaines est faible inversible}
 implique que $f(\gamma_{n-1}^i)$ l'est aussi.

Le corollaire \ref{cor;compatibilite de faible inversible} implique alors que  $f(\gamma^0_{n-1})*_{n-2}f(\gamma^2_{n-1})$ est  une cellule faiblement inversible. La cellule $f(\gamma^0_{n-1})$ (resp. $f(\gamma^2_{n-1})$) est donc faiblement inversible à gauche (resp. à droite).

Pour $i=0,2$, selon la proposition \ref{prop:complicial axiom}, tous les simplexes apparaissant dans la décomposition de $\gamma^i_{n-1}$ et différents de $d^i$ sont envoyés sur des cellules faiblement inversibles.  Une application répétée du corollaire \ref{cor;compatibilite de faible inversible} implique alors que $g(d^{0,3}) = f(d^0)$ (resp. $g(d^{2,3})=f(d^2)$) est faiblement inversible à gauche (resp. à droite).

En étudiant
$$
\begin{array}{r}
|\Delta[n]|\cong |\Delta[2]\star \Delta[n-2]|\xrightarrow{ |\Delta [d_0] \star \Delta[n-3]|}|\Delta[3]\star \Delta[n-2]|\xrightarrow{~g~}C\\
\end{array}
$$
on montre de la même façon que $g(d^{0,1})$ (resp. $g(d^{0,3})$) est faiblement inversible à gauche (resp. à droite). On est déduit donc que $g(d^{1,2})$ est faiblement inversible, ce qui implique que $g(d^{0,3})$ et $g(d^{2,3})$ le sont aussi.
\end{proof}

\begin{theo}
\label{theo: n(c) est un ensemble complicial}
Soit $C$ une $\omega$-catégorie. La stratification présentée à la définition \ref{defi:stratification de n(c)} munit $\mathcal{N}(C)$ d'une structure d'ensemble complicial saturé qui est $k$-triviale si et seulement si $C$ l'est. 
\end{theo}
\begin{proof}
C'est une application directe des propositions \ref{prop: n(c) est un ensemble complicial} et \ref{prop:n(c) est sature}.
\end{proof}

\printbibliography

\end{document}